\documentclass[a4paper,reqno,12pt]{amsart}
\usepackage{amscd,amsmath,amssymb}
\usepackage{pstricks}
\usepackage{latexsym,amsbsy,mathrsfs}
\usepackage{xy}
\usepackage{xypic}
\usepackage{pgf,tikz}
\usepackage{float}
\usepackage{appendix}
\usepackage{booktabs}%use Table
\usepackage{hyperref}

\newtheorem{thm}{Theorem}[section]
\newtheorem{lem}[thm]{Lemma}
\newtheorem{cor}[thm]{Corollary}
\newtheorem{prop}[thm]{Proposition}

\newtheorem{innercustomthm}{{\bf Theorem}}
\newenvironment{customthm}[1]
{\renewcommand\theinnercustomthm{#1}\innercustomthm}
{\endinnercustomthm}

\newtheorem*{innercustomconj}{{\bf GLS Conjecture}}
\newenvironment{customconj}[1]
  {\innercustomconj}
  {\endinnercustomconj}

\usepackage{latexsym,todonotes}

\theoremstyle{definition}
\newtheorem{example}[thm]{Example}

\newtheorem{defn}[thm]{Definition}

\newtheorem{conj}[thm]{Conjecture}
\newtheorem{rem}[thm]{Remark}

\numberwithin{equation}{thm}

\define \cs {\mathcal S}

\normalbaselines
\begin{document}
\title[Representation-tame GLS algebras and a revised GLS conjecture]
{Representation-tame Gei\ss-Leclerc-Schr\"{o}er algebras and a revised GLS conjecture on root systems}

\author[Qiang Dong]{Qiang Dong}
\address{School of Mathematics and Statistics, Fujian Normal University, Fuzhou 350117, 
P.R.China}
\email{dongqiang@fjnu.edu.cn}
\author[Zengqiang Lin]{Zengqiang Lin\textsuperscript{*} }
\thanks{\textsuperscript{*}Corresponding author}
\address{School of Mathematical Sciences, Huaqiao University, Quanzhou 362021, P.R.China}
\email{zqlin@hqu.edu.cn}

\author[Ming Lu]{Ming Lu}
\address{Department of Mathematics, Sichuan University, Chengdu 610064, P.R.China}
\email{luming@scu.edu.cn}

	\author[Shiquan Ruan]{Shiquan Ruan}
	\address{ School of Mathematical Sciences,
		Xiamen University, Xiamen 361005, P.R.China}
	\email{sqruan@xmu.edu.cn}

\subjclass[2020]{Primary  16G10, 16G20, 16G70}
\keywords{Gei\ss-Leclerc-Schr\"{o}er algebras (GLS algebras), tame representation type, $\tau$-locally free modules, Gei\ss-Leclerc-Schr\"{o}er conjecture (GLS conjecture), equivariantization}
\begin{abstract}
Using Galois covering theory and equivariant techniques, we classify
all connected representation-tame Gei\ss--Leclerc--Schr\"oer (GLS) algebras
in terms of their defining triples $(C,D,\Omega)$.
We then study the GLS algebras $H(\widetilde{CD}_n)$ with minimal
symmetrizers, where $n\geq 2$ and
$\widetilde{CD}_2=\widetilde{B}_2$.
%These algebras are not string algebras in the sense of
%Butler and Ringel, they
By realizing these algebras as basic algebras of
$\mathbb{Z}_2$-skew group algebras of the string algebras
$H(\widetilde{C}_{2n-2})$,  we introduce extended strings and extended
bands to parameterize the connected components of the
Auslander--Reiten quivers of $H(\widetilde{CD}_n)$ and determine
their shapes. We further classify the indecomposable
$\tau$-locally free modules over representation-tame GLS algebras
of affine type and show that their rank vectors form precisely
the set of positive roots together with explicitly described
non-root vectors in the positive cone of the root lattice.
This yields a precise revision of the GLS conjecture for
representation-tame GLS algebras of affine type.
\end{abstract}

\dedicatory{Dedicated to Professor Yanan Lin
on the occasion of his 70th birthday}

\maketitle

\setcounter{tocdepth}{1}
	
	\tableofcontents
    
	\section{Introduction}
	Given a symmetrizable generalized Cartan matrix $C$, a symmetrizer $D$ of
	$C$, and an orientation $\Omega$, Gei\ss, Leclerc, and Schr\"oer
	introduced a finite-dimensional $1$-Iwanaga--Gorenstein algebra
	\[
		H(C,D,\Omega),
	\]
	now commonly called a GLS algebra; see
	\cite{GLSIII,GLSI,GLSII,GLSIV,GLSV}.  These algebras extend several
	features of the representation theory of quivers, and provide a natural connection among
	representation theory, Lie theory, and cluster algebras.

	By Drozd's trichotomy, every finite-dimensional algebra over an
	algebraically closed field is representation-finite, tame, or wild; see
	\cite{Dro1977,Dro1980}.  The representation type of $H(C,D,\Omega)$
	depends on the Cartan matrix, the symmetrizer, and, in some cases, the
	orientation.  The representation-finite GLS algebras were classified in
	\cite{GLSI}.  Our first main result completes the tame side of the
	classification.  The proof uses explicit Galois coverings of GLS algebras
	and convex hypercritical subcategories in these coverings as obstructions
	to tameness; see \cite{NS1997,Unger1990} for hypercritical algebras and
	\cite{LS,delaPena1990} for the covering-theoretic criteria used below.

	\begin{customthm}{\bf A}[Theorem~\ref{Tame}]
		\label{thm:A}
		Let $C$ be a connected symmetrizable generalized Cartan matrix,
		$D$ a symmetrizer of $C$, and $\Omega$ an orientation.
		Then $H(C,D,\Omega)$ is representation-tame if and only if
		one of the following conditions holds:
		\begin{enumerate}
			\item $C$ is of Dynkin type $A_2$, and
			$D=\operatorname{diag}(4,4)$;
			\item $C$ is of Dynkin type $A_4$,
			$D=\operatorname{diag}(2,2,2,2)$, and $\Omega$ is a linear
			orientation;
			\item $C$ is of affine type
			$\widetilde{A}_{n},\widetilde{C}_n,\widetilde{D}_n,
			\widetilde{CD}_n,\widetilde{E}_6,\widetilde{E}_7,
			\widetilde{E}_8$ or $\widetilde{B}_2$, and $D$ is minimal.
		\end{enumerate}
	\end{customthm}

	The simply-laced affine cases in Theorem~\ref{thm:A} are hereditary, while
	the GLS algebras of type $\widetilde C_n$ with minimal symmetrizer were
	studied in \cite{HLS2023}.  Thus the new Auslander--Reiten and rank-vector
	analysis is concentrated on the remaining nonsimply-laced affine family.
	We use the convention
	\[
		\widetilde{CD}_2=\widetilde B_2
	\]
	and consider $H(\widetilde{CD}_n)$ for all $n\geq 2$, with minimal
	symmetrizer.  For the orientation  fixed in Sections~\ref{sec:CDn}--\ref{Mainresult1},
	\[
		H(\widetilde{CD}_{n})
		={\bf k}Q(\widetilde{CD}_{n})/
		\langle\varepsilon_{1}^2\rangle,
	\]
	where the quiver $Q(\widetilde{CD}_n)$ is depicted in
	Figure~\ref{quiverForCD}.

	\begin{center}
		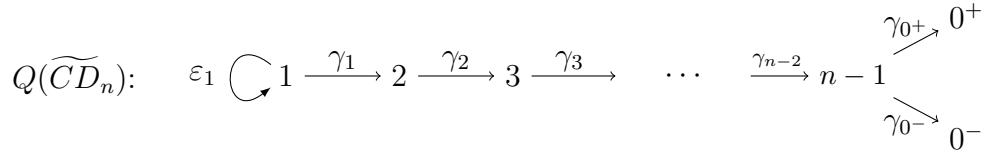
\begin{figure}[h]
			\begin{tikzpicture}
				\node () at (-2.1,0) {\(Q(\widetilde{CD}_{n})\):};
				\node (1) at (0.7,0) {1};
				\draw[-latex] (0.5,0.15) .. controls (-0.2,0.7) and (-0.2,-0.7) .. (0.5,-0.15);
				\node () at (-0.4,0) {\(\varepsilon_{1}\)};
				\node (2) at (2.2,0) {2};
				\node (3) at (3.7,0) {3};
				\node (4) at (5.2,0) {};
				\node () at (5.95,0) {\(\cdots\)};
				\node (5) at (6.7,0) {};
				\node (6) at (8.2,0) {\small{\(n-1\)}};
				\node (7) at (9.7,0.8) {\(0^{+}\)};
				\node (8) at (9.7,-0.8) {\(0^{-}\)};
				\draw[->] (1) to (2);
				\draw[->] (2) to (3);
				\draw[->] (3) to (4);
				\draw[->] (5) to (6);
				\draw[->] (6) to (7);
				\draw[->] (6) to (8);
				\node () at (1.45,0.2) {\(\gamma_{1}\)};
				\node () at (2.95,0.2) {\(\gamma_{2}\)};
				\node () at (4.45,0.2) {\(\gamma_{3}\)};
				\node () at (7.2,0.2) {\tiny{\(\gamma_{n-2}\)}};
				\node () at (8.9,0.65) {\(\gamma_{0^{+}}\)};
				\node () at (8.9,-0.65) {\(\gamma_{0^{-}}\)};
			\end{tikzpicture}
			\caption{The quiver of GLS algebra $H(\widetilde{CD}_{n})$}
			\label{quiverForCD}
		\end{figure}
	\end{center}

	The structural link with string algebras is provided by the observation
	that $H(\widetilde{CD}_n)$ is a basic algebra of the
	$\mathbb Z_2$-skew group algebra of $H(\widetilde C_{2n-2})$.
	The latter is a string algebra, whereas $H(\widetilde{CD}_n)$ itself is
	not.  Motivated by the string and band combinatorics of
	$H(\widetilde C_{2n-2})$, we introduce \emph{extended strings} and
	\emph{extended bands}.  Equivariantization then allows us to transport
	the relevant Auslander--Reiten information from
	$H(\widetilde C_{2n-2})$ to $H(\widetilde{CD}_n)$ and to determine the
	connected components of its Auslander--Reiten quiver.

	\begin{customthm}{\bf B}[Theorem~\ref{MainTheorem}]
		The Auslander--Reiten quiver of the GLS algebra
		$H(\widetilde{CD}_{n})$ consists of the following components:
		\begin{enumerate}
			\item one component containing all indecomposable preprojective
			modules and all indecomposable preinjective modules;
			\item components of type $\mathbb{Z}D_{\infty}$, indexed by
			$1$-ex-string pairs;
			\item components of type $\mathbb{Z}A_{\infty}^{\infty}$, indexed
			by non-symmetric full ex-strings;
			\item one stable tube of rank $n-1$;
			\item stable tubes of rank $2$, indexed by $2$-ex-string pairs;
			\item homogeneous tubes parametrized by ex-band data and the
			corresponding scalar parameters.
		\end{enumerate}
	\end{customthm}
	The indexing in this statement is understood modulo the equivalence
	relations specified precisely in Theorem~\ref{MainTheorem}.

	We next turn to the root--module correspondence for GLS algebras.
	The relevant class is that of indecomposable $\tau$-locally free modules,
	and the corresponding invariant is the rank vector rather than the
	ordinary dimension vector.  In Dynkin type, Gei\ss, Leclerc, and
	Schr\"oer proved that the rank map induces a canonical bijection between
	the isomorphism classes of indecomposable $\tau$-locally free modules and
	the positive roots of the simple Lie algebra $\mathfrak g(C)$ \cite{GLSI}.  They
	subsequently formulated the following conjecture for arbitrary
	symmetrizable generalized Cartan matrices; see \cite[Conjecture~5.3]{GLSII}.

	\begin{customconj}{}	There exists a canonical bijection between the set of positive roots
		of the Kac--Moody Lie algebra $\mathfrak g(C)$ attached to $C$ and the
		set of rank vectors of indecomposable $\tau$-locally free
		$H(C,D,\Omega)$-modules.
	\end{customconj}

Identifying each simple root $\alpha_i$ with the standard basis
vector corresponding to vertex $i$, the GLS conjecture can be
written as
\[
\bigl\{\underline{\operatorname{rank}}(M)
\mid M \text{ is indecomposable and } \tau\text{-locally free}\bigr\}
=\Delta^+,
\]
where $\Delta^+$ denotes the set of positive roots of
$\mathfrak g(C)$.
	 The conjecture holds for affine type
	$\widetilde C_n$ with minimal symmetrizer \cite{HLS2023}.  Lin and Su
	proved that, for an arbitrary affine GLS algebra, every positive root is
	the rank vector of an indecomposable $\tau$-locally free module, but they
	also constructed affine examples with rank vectors that are not roots;
	see \cite[Theorems~2--3]{LS2025}.  Thus the failure of the original
	conjecture in affine type lies in the converse inclusion: additional
	rank vectors may occur.

We determine all additional non-root rank vectors
in the representation-tame affine cases.
The simply-laced affine cases and type $\widetilde C_n$
are already known, so the remaining family is
$\widetilde{CD}_n$ with $n\geq2$.
For any orientation, Lemma~\ref{lem:long-mouth-pair}
constructs a stable tube of rank $2$ whose mouth modules
have positive long real rank vectors $\beta_1,\beta_2$
satisfying $\beta_1+\beta_2=2\delta$.
These vectors yield the following explicit revision
of the GLS conjecture.
Here $\Delta^+$ denotes the set of positive roots
of $\mathfrak g(C)$, and $\delta$ is the minimal
positive imaginary root.

	% We determine these additional rank vectors in all representation-tame
	% affine cases.  The only new cases that require analysis are
	% $\widetilde{CD}_n$ for $n\geq2$; the simply-laced affine cases and the
	% type $\widetilde C_n$ case are already known.  Let $\delta$ be the
	% minimal positive imaginary root, and let
	% \[
	% 	\Delta_{\mathrm{re},\mathrm{reg}\text{-}\mathrm{sim}}^{\ell,+}
	% \]
	% denote the set of minimal positive long real roots, with respect to the
	% standard root order, that occur as rank vectors of $\tau$-locally free
	% modules lying at the mouths of stable tubes.  Define
	% \[
	% 	\widetilde{\Delta}^{+}
	% 	=
	% 	\Delta^{+}\cup
	% 	\left\{
	% 		\alpha+(2r+1)\delta
	% 		\ \middle|\
	% 		\alpha\in
	% 		\Delta_{\mathrm{re},\mathrm{reg}\text{-}\mathrm{sim}}^{\ell,+},
	% 		\ r\in\mathbb Z_{\geq0}
	% 	\right\}.
	% \]
	% The vectors added to $\Delta^+$ need not themselves be roots.  The
	% precise revised statement is the following.

	\begin{customthm}{\bf C}[Theorem~\ref{Bijection}]
    \label{thm:C}
    Let $C$ be a Cartan matrix of affine type, and let
$H=H(C,D,\Omega)$ be a representation-tame GLS algebra.
\begin{enumerate}
  \item[(1)]
  If $C$ is simply laced or of type $\widetilde C_n$, then
  \[
    \bigl\{
      \underline{\operatorname{rank}}(M)
      \mid M \text{ is indecomposable and }
      \tau\text{-locally free}
    \bigr\}
    =\Delta^+.
  \]

  \item[(2)]
  If $C$ is of type $\widetilde{CD}_n$ with $n\geq2$,
  where $\widetilde{CD}_2=\widetilde B_2$, then
  \[
    \begin{aligned}
      &\bigl\{
        \underline{\operatorname{rank}}(M)
        \mid M \text{ is indecomposable and }
        \tau\text{-locally free}
      \bigr\}\\
      &\qquad
      =\Delta^+\cup
      \bigl\{
        \beta_i+(2r+1)\delta
        \mid i=1,2,\ r\in\mathbb Z_{\geq0}
      \bigr\},
    \end{aligned}
  \]
  where $\beta_1,\beta_2$ are obtained by the construction
  in the proof of Lemma~\ref{lem:long-mouth-pair}.
\end{enumerate}
	\end{customthm}

Thus the original GLS conjecture holds in case~(1).
In case~(2), its failure is measured precisely by
the two families $\beta_i+(2r+1)\delta$,
with $i=1,2$ and $r\in\mathbb Z_{\geq0}$.
Theorem~\ref{thm:C} therefore gives an exact description
of all non-root rank vectors in the
representation-tame affine setting.

	The paper is organized as follows.
	Section~\ref{sec:prelim} reviews the string algebra
	$H(\widetilde C_{2n-2})$ and its module category.
	Section~\ref{sec:CDn} introduces extended strings and extended bands for
	$H(\widetilde{CD}_n)$. Sections~\ref{sec:equivariantization} and~\ref{Mainresult1} develop the
	$\mathbb Z_2$-equivariant correspondence and determine the
	Auslander--Reiten quiver. Section~\ref{sec:galois} constructs locally bounded Galois coverings for GLS algebras and establishes criteria for detecting wildness. Section~\ref{sec:tame-class} proves the classification of
	representation-tame GLS algebras.
    Section~\ref{sec:tau-free} 
    classifies the indecomposable
$\tau$-locally free modules over representation-tame GLS algebras of affine type, determines their rank vectors, and establishes a revised version of the GLS conjecture.
	The appendices contain  explicit equivariant calculations for stable tubes, string modules
	and band modules.

	\vspace{2mm}
	\noindent{\bf Acknowledgments.}
	Q. Dong is supported by the National Natural Science Foundation of China
	(No. 12301054).
	Z. Lin is supported by the National Natural Science Foundation of China
	(No. 12471035) and Fujian Provincial Natural Science Foundation of China
	(No. 2024J01088).
	M. Lu is partially supported by the National Natural Science Foundation
	of China (No. 12671049, 12631002).
	S. Ruan is partially supported by Fundamental Research Funds for Central
	Universities of China (No. 20720250059), Fujian Provincial Natural
	Science Foundation of China (No. 2024J010006), and the National Natural
	Science Foundation of China (No. 12271448).

	%%%%%%%%%%%%%

%%%%%%%%%%%%
\section{The string algebra $H(\widetilde{C}_{2n-2})$}\label{sec:prelim}

%Finally, we fix the notation used throughout the paper. 
Let \(\mathbf{k}\) be an algebraically closed field of characteristic zero, with multiplicative group \(\mathbf{k}^\ast = \mathbf{k}\setminus\{0\}\). Multiplication in algebras and path algebras of quivers is read from left to right, and all modules are right modules unless stated otherwise. For a basic finite-dimensional \(\mathbf{k}\)-algebra $A$, let $\operatorname{mod}A$ stand for the category of finite-dimensional right $A$-modules, and $\operatorname{ind}A$ for its full subcategory consisting of indecomposable objects. The Auslander–Reiten translation is written as \(\tau = D\operatorname{Tr}\).

\subsection{GLS algebras}\label{GLSalgebra}
First, we recall the definition of GLS algebras from \cite[Section 1.4]{GLSI}.

A matrix $C=(c_{ij})\in M_n(\mathbb{Z})$ is a {\emph{symmetrizable generalized Cartan matrix}} if the following hold:
\begin{enumerate}
    \item[(1)] $c_{ii}=2$ for all $1\leq i\leq n$;
    \item[(2)] $c_{ij}\leq0$ for all $1\leq i\neq j\leq n$, and $c_{ij}\neq0$ if and only if $c_{ji}\neq0$;
    \item[(3)] there is a diagonal integer matrix $D=\operatorname{diag}(d_1,d_2,\dots,d_n)$ with $d_i\geq1$ for all $1\leq i\leq n$ such that $DC$ is symmetric.
\end{enumerate}
The matrix $D$ appearing in (3) is called a {\emph{symmetrizer}} of $C$. We say that a symmetrizer $D$ is {\emph{minimal}} if $d_1+d_2+\cdots+d_n$ is minimal. Henceforth, by a Cartan matrix we always mean a symmetrizable generalized Cartan matrix.

Let $C$ be a Cartan matrix. An {\emph{orientation}} of $C$ is a subset $\Omega\subseteq\{1,2,\dots,n\}\times\{1,2,\dots,n\}$ such that the following hold:
\begin{enumerate}
    \item $\{(i,j),(j,i)\}\cap\Omega\neq\varnothing$ if and only if $c_{ij}<0$;
    \item for each sequence $\{(i_1,i_2),(i_2,i_3),\dots,(i_t,i_{t+1})\}$ with $t\geq1$ and $(i_s,i_{s+1})\in\Omega$ for all $1\leq s\leq t$, we have $i_1\neq i_{t+1}$.
\end{enumerate}
For an orientation $\Omega$ of $C$, let $Q:=Q(C,\Omega)$ be the quiver with vertex set 
$$Q_0:=\{1,2,\dots,n\},$$ and arrow set
$$Q_1:=\{\alpha_{ij}^{(k)}:i\rightarrow j\mid(i,j)\in\Omega,1\leq k\leq \textup{gcd}(|c_{ij}|,|c_{ji}|)\}\cup\{\varepsilon_i:i\rightarrow i\mid1\leq i\leq n\}.$$

If $\textup{gcd}(|c_{ij}|,|c_{ji}|)=1$, we also write $\alpha_{ij}$ instead of $\alpha_{ij}^{(1)}$. We call $Q$ a quiver of type $C$. Let $Q^{\textup{o}}:=Q^{\textup{o}}(C,\Omega)$ be the quiver obtained from $Q$ by deleting all loops $\varepsilon_i$. Clearly, $Q^{\textup{o}}$ is a finite acyclic quiver.

\begin{defn}[{\cite[Section 1.4]{GLSI}}]
Let $C$ be a Cartan matrix, $D$ a symmetrizer of $C$, and $\Omega$ an orientation. The {\emph{GLS algebra}} associated to $(C,D,\Omega)$ is
$$H(C,D,\Omega):={\bf k}Q(C,\Omega)/I,$$
where $I$ is the ideal of ${\bf k}Q(C,\Omega)$ generated by the following relations:
\begin{itemize}
    \item[(H1)] for each vertex $i$, the nilpotency relation $$\varepsilon_i^{d_i}=0;$$
    \item[(H2)] for each $(i,j)\in\Omega$ and each $1\leq k\leq \textup{gcd}(|c_{ij}|,|c_{ji}|)$, the commutativity relation
    $$\varepsilon_i^{\frac{|c_{ji}|}{\textup{gcd}(|c_{ij}|,|c_{ji}|)}}\alpha_{ij}^{(k)}=\alpha_{ij}^{(k)}\varepsilon_j^{\frac{|c_{ij}|}{\textup{gcd}(|c_{ij}|,|c_{ji}|)|}}.$$
\end{itemize}
We write $H(C,\Omega)$ when $D$ is the minimal symmetrizer.  
\end{defn}

It is known that every GLS algebra is finite-dimensional and $1$-Iwanaga–Gorenstein. For simplicity, we refer to $H(C,D,\Omega)$ as a GLS algebra of type~$C$.

\subsection{The GLS algebra of type $\widetilde{C}_{2n-2}$}\label{GLSalgebraofC}

In this section, we consider the GLS algebra $$H(\widetilde{C}_{2n-2})={\bf{k}}Q(\widetilde{C}_{2n-2})/\langle\varepsilon_{1}^2,\varepsilon_{-1}^2\rangle$$ of type $\widetilde{C}_{2n-2}$, where the quiver is depicted in Figure \ref{quiverForC}. 

\begin{figure}[h]
\centering
\begin{tikzpicture}
\node () at (-2,0) {$Q(\widetilde{C}_{2n-2})$:};
\node () at (1.75,0.94) {$\beta_1$};
\node () at (3.25,0.94) {$\beta_2$};
\node () at (4.75,0.94) {$\beta_3$};
\node () at (7,0.94) {${\beta_{n-2}}$};
\node () at (9.2,0.5) {${\beta_{n-1}}$};
\node () at (1.75,-0.46) {$\beta_{-1}$};
\node () at (3.25,-0.46) {$\beta_{-2}$};
\node () at (4.75,-0.46) {$\beta_{-3}$};
\node () at (9.2,-0.5) {${\beta_{1-n}}$};
\node () at (7,-0.46) {${\beta_{2-n}}$};
\node (1) at (1,0.7) {$1$};
\node (2) at (2.5,0.7) {$2$};
\node (3) at (4,0.7) {$3$};
\node (4') at (5.5,0.7) {};
\node () at (6,0.7) {$\cdots$};
\node (4) at (6.5,0.7) {};
\node (5) at (8,0.7) {\tiny{$n-1$}};
\draw [->] (1) -- (2);
\draw [->] (2) -- (3);
\draw [->] (3) -- (4');
\draw [->] (4) -- (5);
\node (-1) at (1,-0.7) {\tiny{$-1$}};
\node (-2) at (2.5,-0.7) {\tiny{$-2$}};
\node (-3) at (4,-0.7) {\tiny{$-3$}};
\node (-4') at (5.5,-0.7) {};
\node () at (6,-0.7) {$\cdots$};
\node (-4) at (6.5,-0.7) {};
\node (-5) at (8,-0.7) {\tiny{$1-n$}};
\draw [->] (-1) -- (-2);
\draw [->] (-2) -- (-3);
\draw [->] (-3) -- (-4');
\draw [->] (-4) -- (-5);
\node (n) at (9.5,0) {$0$};
\draw [->] (-5) -- (n);
\draw [->] (5) -- (n);
\draw[-latex] (0.5,0.85) .. controls (-0.2,1.3) and (-0.2,0) .. (0.5,0.55);
\draw[-latex] (0.5,-0.55) .. controls (-0.2,0) and (-0.2,-1.3) .. (0.5,-0.85);
\node () at (-0.4,0.7) {$\varepsilon_1$};
\node () at (-0.4,-0.7) {$\varepsilon_{-1}$};
\end{tikzpicture}
\caption{The quiver of GLS algebra $H(\widetilde{C}_{2n-2})$}
\label{quiverForC}
\end{figure}
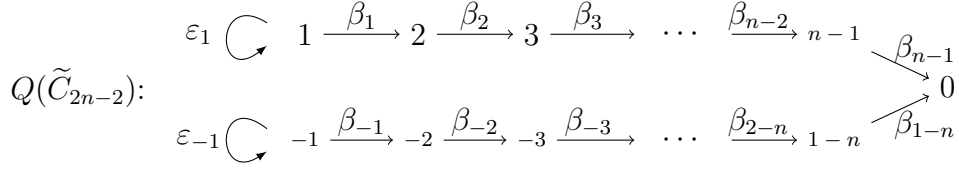

The GLS algebra $H(\widetilde{C}_{2n-2})$ is a string algebra. To describe the category of $H(\widetilde{C}_{2n-2})$-modules, we recall fundamental facts about string algebras following \cite{BR1987,HLS2023}.

Let $Q$ be a bound quiver with an admissible ideal $I$ such that $A={\bf{k}}Q/I$ is a string algebra. 
For any arrow $\alpha:i\to j$ in $Q$, we denote its source and target by $s(\alpha)=i$ and $t(\alpha)=j$, respectively.
We write $\alpha^{-1}$ for the {\emph{formal inverse}} of $\alpha$, with $s(\alpha^{-1})=j$ and $t(\alpha^{-1})=i$, and define $(\alpha^{-1})^{-1}=\alpha$. 

A {\emph{word}} $w$ of length $m\geq1$ is a formal composition $c_1c_2\cdots c_m$, where each $c_i$ is either an arrow in $Q$ or its formal inverse, satisfying $s(c_{i+1})=t(c_i)$ for $1\leq i< m$. We define the source and target of $w$ as $s(w)=s(c_1)$ and $t(w)=t(c_m)$, respectively. It is clear that a word can be viewed as a walk in $Q$.

For a word $w=c_1c_2\cdots c_m$, we define 
\begin{itemize}
    \item[-] its inverse $w^{-1}:=c_m^{-1}\cdots c_{2}^{-1} c_1^{-1}$;
    \item[-] its cyclic shift $w_{(i)}:=c_{i}c_{i+1}\cdots c_mc_1\cdots c_{i-1}$ for $1\leq i\leq m$ if $t(c_m)=s(c_1)$;
    \item[-] its truncation $w_{[i,j]}:=c_i c_{i+1}\cdots c_j$ for $1\leq i\leq j\leq m$.
\end{itemize}
For convenience, we set $w_{[i,i-1]}$ to be the primitive idempotent $e_{s(c_{i})}=e_{t(c_{i-1})}$, and write $w_{\geq i}=w_{[i,m]}$ and $w_{\leq j}=w_{[1,j]}$. 

A word $w=c_1c_2\cdots c_m$ of length $m\geq1$ is called a {\emph{string}} if $c_{i+1}\neq c_i^{-1}$ for all $1\leq i<m$, and no subword $c_ic_{i+1}\cdots c_{i+t}$ nor its inverse belongs to $I$. In addition, for each vertex $u$ of $Q$, there are two {\emph{trivial strings}} $1_{(u,1)}$ and $1_{(u,-1)}$ of length $0$, where $1_{(u,1)}^{-1}=1_{(u,-1)}$ and $1_{(u,-1)}^{-1}=1_{(u,1)}$. A non-trivial string $w$ is called a {\emph{band}} if its source and target coincide, and each power $w^r$ ($r\geq2$) is a string, but $w$ itself is not a proper power of a shorter string.

We write $\textup{St}(A)$ \big(resp. $\textup{Ba}(A)$\big) for the set of all strings (resp. bands) in $A$. Let $\rho$ be the equivalence relation on $\textup{St}(A)$ identifying each word with its inverse, and let $\rho'$ be the equivalence relation on $\textup{Ba}(A)$ identifying each band with its cyclic shifts and their inverses. Choose complete sets {$\overline{\textup{St}}(A)$} and {$\overline{\textup{Ba}}(A)$} of representatives of $\textup{St}(A)$ and $\textup{Ba}(A)$ modulo $\rho$ and $\rho'$, respectively. We write $w\sim w'$ whenever $w$ is equivalent to $w'$ (in either context).

To describe the strings and bands in $H(\widetilde{C}_{2n-2})$, we first define the index set. Set
\begin{equation}\label{Index I}\mathcal{I}=\{\delta_1\delta_2\cdots\delta_m\mid m\in\mathbb{N},\;\;\delta_i=1\textup{\;or\;}-1\textup{\;for\;}1\leq i\leq m\}.
\end{equation}
For a sequence $\delta=\delta_1\delta_2\cdots\delta_m\in\mathcal{I}$, we call $m$ the length of $\delta$. The unique index of length $0$ is denoted by $\varnothing$. The multiplication on $\mathcal{I}$ is defined by concatenation
$$(\delta_1\delta_2\cdots\delta_{m})\circ(\delta'_1\delta'_2\cdots\delta'_{m'})=\delta_1\delta_2\cdots\delta_{m}\delta'_1\delta'_2\cdots\delta'_{m'},$$
with the convention that $\varnothing$ acts as the identity, i.e., $\delta\circ\varnothing=\varnothing\circ\delta=\delta$ for all $\delta\in\mathcal{I}$. Moreover, $\mathcal{I}$ carries a shift operator $[1]$ defined by cyclic permutation $$(\delta_1\delta_2\cdots\delta_m)[1]=\delta_2\cdots\delta_m\delta_1,$$ and an involution given by $$(\delta_1\delta_2\cdots\delta_m)^{-1}=(-\delta_m)\cdots(-\delta_2)(-\delta_1).$$ 
We take \begin{equation}\label{I-even}\mathcal{I}_{e}\subseteq\mathcal{I}
\end{equation}
to be the subset consisting of indices of even length.

For the quiver $Q(\widetilde{C}_{2n-2})$ as in Figure \ref{quiverForC}, throughout the paper we denote the string
$$\beta_{\textup{max}}:=\beta_1\beta_2\cdots\beta_{n-1}\beta_{1-n}^{-1}\cdots\beta_{-2}^{-1}\beta_{-1}^{-1},$$ 
and for $\delta=\delta_1\delta_2\cdots\delta_m\in\mathcal{I}$, denote the string
$$\xi(\delta):=\varepsilon_1^{\delta_1}\beta_{\textup{max}}\varepsilon_{-1}^{\delta_2}\beta_{\textup{max}}^{-1}\varepsilon_1^{\delta_3}\beta_{\textup{max}}\cdots\beta_{\textup{max}}^{(-1)^{m}}\varepsilon_{(-1)^{m+1}}^{\delta_{m}}.$$

Note that each word of positive length admits a unique predecessor and a unique successor (each being an arrow or a formal inverse), such that their compositions with the word do not belong to $I$. 

By the shape of $Q(\widetilde{C}_{2n-2})$ and its defining relations, every non-trivial string is a truncation $(\xi(\delta))_{[i,j]}$
for some $\delta=\delta_1\delta_2\cdots\delta_{2m}\in\mathcal{I}_{e}$ and $1\leq i\leq j\leq 4nm-2n-2m+2$, while every band is equivalent to one of the form $\xi(\delta)\cdot\beta_{\textup{max}}^{-1}$ for some $\delta\in\mathcal{I}_{e}$.

%%%%%%%%%
\subsection{The module category of the GLS algebra $H(\widetilde{C}_{2n-2})$}
The aim of this subsection is to recall the Auslander--Reiten
quiver of $\operatorname{mod}H(\widetilde{C}_{2n-2})$.
To simplify the notation used below, we write $Q:=Q(\widetilde{C}_{2n-2})$ throughout this subsection.

Let $w=c_1c_{2}\cdots c_{m}$ be a string with the corresponding walk
$$\xymatrix{x_1\ar@{-}[r]^{c_1}&x_{2}\ar@{-}[r]^{c_{2}}&\cdots\ar@{-}[r]^{c_{m-1}}&x_m\ar@{-}[r]^{c_m}&x_{m+1}.}$$
The {\emph{string module}} $M(w)$ defined by $w$ is the representation $\big((M(w)_u)_{u\in Q_0},(M(w)_\alpha)_{\alpha\in Q_1}\big)$, where the vector spaces
$$M(w)_u=\left\{\begin{array}{ll}
    \oplus_{x_i=u}{\bf{k}}x_i & \textup{if}\;u=x_i\;\textup{for some}\;1\leq i\leq m+1,  \\
    0 & \textup{otherwise,}
\end{array}\right.$$
and the linear maps
$$M(w)_{\alpha}(x_i)=\left\{\begin{array}{ll}
    x_{i+1} & \textup{if}\;\alpha=c_{i}, \\
    x_{i-1} & \textup{if}\;\alpha^{-1}=c_{i-1},\\
    0 & \textup{otherwise.}
\end{array}\right.$$
Obviously, $M(w)$ is isomorphic to $M(w^{-1})$ as $H(\widetilde{C}_{2n-2})$-modules for any string $w$, and $M(1_{(u,j)})$ is the simple representation corresponding to the vertex $u$.

Now we assume that $w=c_1c_{2}\cdots c_m$ is a band as follows. 
$$\xymatrix{x_0\ar@{-}[r]^{c_m}\ar@{-}[rrd]_{c_{1}}&x_{m-1}\ar@{-}[r]^{c_{m-1}}&\cdots\ar@{-}[r]^{c_{4}}&x_{3}\ar@{-}[r]^{c_{3}}&x_{2}\\
&&x_{1}\ar@{-}[rru]_{c_{2}}&&}$$
To avoid confusion in subsequent calculations, we always assume that multiplication in a band is anticlockwise.

Let $V$ be a ${\bf{k}}$-vector space and $\varphi$ an automorphism of $V$. The {\emph{band module}} $M(w,\varphi)$ associated to $(w,\varphi)$ is the representation 
$\big((M(w,\varphi)_u)_{u\in Q_0},(M(w,\varphi)_\alpha)_{\alpha\in Q_1}\big)$,
where the vector spaces are given by
$$M(w,\varphi)_u=\left\{\begin{array}{ll}
   \oplus_{x_i=u}V_{x_i}  & \textup{if}\;u=x_i\;\textup{for some}\;i\in\{0,1,2,\cdots,m-1\},  \\
    0 & \textup{otherwise,}
\end{array}\right.$$
with each $V_{x_i}$ identified with $V$, and the linear maps are defined by
$$M(w,\varphi)_{\alpha}(v)=\left\{\begin{array}{ll}
    \varphi(v) & \textup{if}\;v\in V_{x_{0}}\;\textup{and}\;\alpha=c_1, \\
    \varphi^{-1}(v) & \textup{if}\;v\in V_{x_1}\;\textup{and}\;\alpha^{-1}=c_{1},\\
    v & \textup{if}\;v\in V_{x_{i-1}},\alpha=c_i,\;\textup{or}\;v\in V_{x_{i}},\alpha^{-1}=c_i,\;\textup{for}\;i\neq1,\\
    %v & \textup{if}\;v\in V_{x_{i-1}}\;\textup{and}\;\alpha^{-1}=c_i\;\textup{for}\;i\neq1,\\
    0 & \textup{otherwise,}
\end{array}\right.$$
where $x_m$ is identified with $x_0$. Obviously, $M(w,\varphi)\simeq M(w^{-1},\varphi^{-1})$ and $M(w,\varphi)\simeq M(w_{(i)},\varphi)$. 

\begin{example}\label{Example2.2}
Let $H(\widetilde{C}_{2})$ denote the GLS algebra of type $\widetilde{C}_2$. 
%Then $H(\widetilde{C}_{2})$ is a string algebra.
%Let $H(\widetilde{C}_{2})$ denote the GLS algebra of type $\widetilde{C}_2$ with underlying quiver $Q(\widetilde{C}_{2})$ and relations $\varepsilon^2_{1},\varepsilon^2_{-1}$, defined as follows.
% \begin{center}
% \begin{tikzpicture}[scale=1.1]
% \node (-20) at (0,0) {$1$};
% \node (-20) at (1.5,0) {$0$};
% \node (-20) at (3,0) {$-1$};
% \node (-20) at (0.75,0.2) {$\beta_1$};
% \node (-20) at (2.25,0.2) {$\beta_{-1}$};
% \draw [->] (0.25,0) -- (1.25,0);
% \draw [->] (2.75, 0) -- (1.75,0);
% \draw[-latex] (-0.2,0.2) .. controls (-0.8,0.5) and (-0.8,-0.5) .. (-0.2,-0.2);
% \draw[-latex] (3.35,0.2) .. controls (3.95,0.5) and (3.95,-0.5) .. (3.35,-0.2);
% \node (-20) at (-0.9,0) {$\varepsilon_1$};
% \node (-20) at (4.15,0) {$\varepsilon_{-1}$};
% \end{tikzpicture}    
% \end{center}
Let $w$ denote the non-trivial string $\varepsilon_1\beta_1\beta_{-1}^{-1}\varepsilon_{-1}^{-1}\beta_{-1}\beta_1^{-1}$. The corresponding string module $M(w)$  is given by the representation illustrated in Figure \ref{exampleM(w)}. 
\begin{figure}[h]
\centering
\begin{tikzpicture}
\node () at (-1.2,0) {$M(w):$}; 
\node (1) at (0,0) {${\bf k}^3$}; 
\node () at (0,1.3) {\tiny{$\begin{pmatrix}
0&0&0\\
1&0&0\\
0&0&0
\end{pmatrix}$}}; 
\node () at (1.7,0.6) {$\begin{pmatrix}
0&1&0\\
0&0&1
\end{pmatrix}$}; 
\node (2) at (3,0) {${\bf k}^2$}; 
\node () at (3.9,0.6) {$\begin{pmatrix}
1&0\\
0&1
\end{pmatrix}$}; 
\node () at (5.2,1.2) {$\begin{pmatrix}
0&1\\
0&0
\end{pmatrix}$}; 
\node (3) at (5.2,0) {${\bf k}^2$}; 
\draw [->] (1) -- (2);
\draw [->] (3) -- (2);
\draw[-latex] (-0.2,0.2) .. controls (-0.5,0.8) and (0.5,0.8) .. (0.2,0.15);
\draw[-latex] (5,0.2) .. controls (4.7,0.8) and (5.7,0.8) .. (5.4,0.15);
\end{tikzpicture}\caption{The string module $M(w)$}\label{exampleM(w)}
\end{figure}

Note that $w$ is also a band. Let $V={\bf k}^2$ and $\varphi:\;{\bf k}^2\rightarrow{\bf k}^2$ be the automorphism given by the Jordan block $J=J(\lambda,2)$ with eigenvalue $\lambda\neq0$. The band module $M(w,\varphi)$ corresponds to the representation illustrated in Figure \ref{exampleM(w,varphi)}.
\begin{figure}[h]
\centering
\begin{tikzpicture}
\node () at (-0.6,0) {$M(w,\varphi)$:};
\node (1) at (0.8,0) {${\bf k}^4$};
\node (2) at (3.6,0) {${\bf k}^4$};
\node (3) at (6,0) {${\bf k}^4$};
\node () at (4.8,0.3) {$I_4$}; 
\node () at (2.3,0.6) {$\begin{pmatrix}
J^{-1}&0\\
0&I_2
\end{pmatrix}$}; 
\node () at (0.7,1.25) {$\begin{pmatrix}
0&0\\
I_2&0
\end{pmatrix}$}; 
\node () at (6,1.25) {$\begin{pmatrix}
0&0\\
I_2&0
\end{pmatrix}$};
\draw[-latex] (0.6,0.2) .. controls (0.3,0.8) and (1.3,0.8) .. (1,0.15);
\draw[-latex] (5.8,0.2) .. controls (5.5,0.8) and (6.5,0.8) .. (6.2,0.15);
\draw [->] (1) -- (2);
\draw [->] (3) -- (2);
\end{tikzpicture}\caption{The band module $M(w,\varphi)$}\label{exampleM(w,varphi)}
\end{figure}
\end{example}

We denote by $\Gamma_A$ the Auslander--Reiten quiver of an algebra $A$; see \cite{RepTheIII,ARS1995} for background on almost split sequences and Auslander--Reiten quivers. For a string algebra $A$, a string $A$-module $M(w)$ is called {\emph{minimal}} if each irreducible map $M(w)\rightarrow M(w')$ is injective and each irreducible map $M(w'')\rightarrow M(w)$ is surjective in $\Gamma_{A}$. In this case, we also call $w$ a minimal string. 

Recall that for any indecomposable $A$-modules $M$ and $N$, the quotient $\textup{Irr}(M,N):=\textup{rad}_A(M,N)/\textup{rad}^2_A(M,N)$ is the space of irreducible morphisms.
For any string $w$, set 
$$a=\sum_{w'\in\overline{\textup{St}}(A)}\textup{dim}_{{\bf k}}\textup{Irr}(M(w'),M(w)) \text{\quad and\quad}
b=\sum_{w''\in\overline{\textup{St}}(A)}\textup{dim}_{{\bf k}}\textup{Irr}(M(w),M(w'')),$$ and we say that $w$ is {\emph{of type $(a,b)$}}.

Now we can state the structure of the Auslander--Reiten quiver of the GLS algebra $H(\widetilde{C}_{2n-2})$.

\begin{prop}[{\cite[Theorem A]{HLS2023}}]\label{ThmCn}
The Auslander--Reiten quiver of $H(\widetilde{C}_{2n-2})$ is described as follows.
\begin{enumerate}
    \item One component $\mathcal{H}(\widetilde{C}_{2n-2})_{{PI}}$ containing all indecomposable preprojective modules and indecomposable preinjective modules, which has the form $\mathbb{Z}A_{\infty}^{\infty}-(Q^{\textup{o}}(\widetilde{C}_{2n-2}))^{\text{op}}$.
    \item One stable tube $\mathcal{H}(\widetilde{C}_{2n-2})_{(1,1)}$ of rank $2n-2$, whose bottom consists of all minimal string modules of type $(1,1)$.
    \item Homogeneous tubes $\mathcal{H}(\widetilde{C}_{2n-2})_{w,\lambda}$, where $w\in\overline{\textup{Ba}}\big(H(\widetilde{C}_{2n-2})\big)$ and $\lambda\in{\bf{k}}^{*}$.
    \item Components $\mathcal{H}(\widetilde{C}_{2n-2})_{w}$ of type $\mathbb{Z}A_{\infty}^{\infty}$, indexed by the equivalence classes of minimal strings of type $(2,2)$.
\end{enumerate}
\end{prop}

\section{The GLS algebra of type $\widetilde{CD}_{n}$}
\label{sec:CDn}

In this section, we consider the following GLS algebra of type $\widetilde{CD}_{n}$ for $n\geq 2$: $$H(\widetilde{CD}_{n})={\bf{k}}Q(\widetilde{CD}_{n})/\langle\varepsilon_{1}^2\rangle,$$ where  the quiver $Q(\widetilde{CD}_{n})$ is depicted in Figure \ref{quiverForCD}. Throughout this section, we write $Q$ for $Q(\widetilde{CD}_{n})$ and identify $\widetilde{CD}_{2}$ with $\widetilde{B}_{2}$. 

%Specifically, $H(\widetilde{CD}_{n})$ is defined via the bound quiver $Q(\widetilde{CD}_{n})$ with the zero relation $\varepsilon_{1}^2$ and the fixed orientation given in Figure \ref{quiverForCD}.

Note that the GLS algebra $H(\widetilde{CD}_{n})$ is not a string algebra. To describe its Auslander–Reiten quiver, we need to extend the definitions of strings and bands.

\subsection{Extended strings and extended bands}\label{exstring}
We introduce a new symbol
\begin{figure}[H]
    \begin{center}
        \begin{tikzpicture}
        \node() at (-13.5,1) {$\gamma=(\gamma_{0^{+}},\gamma_{0^{-}}):$};
        \node(a) at (-11.5,0.5) {$0^{+}$};
        \node() at (-11,1.1) {$\gamma_{0^{+}}$};
        \node() at (-9,1.15) {$\gamma_{0^{-}}$};
        \node(b) at (-8.5,0.5) {$0^{-}$};
        \node(c) at (-10,1.5) {$n-1$};
\draw [->] (c) -- (a);  
\draw [->] (c) -- (b);        
        \end{tikzpicture}
    \end{center}
    \end{figure}
\noindent from vertex $n-1$ to the vertex pair $(0^{+},0^{-})$, together with its inverse
\begin{figure}[H]
    \begin{center}
        \begin{tikzpicture}       
\node() at (-9,1) {$\gamma^{-1}=(\gamma^{-1}_{0^{+}},\gamma^{-1}_{0^{-}}):$};
\node(x) at (-6.5,1.5) {$0^{+}$};
\node(y) at (-3.5,1.5) {$0^{-}$};
\node(z) at (-5,0.5) {$n-1$};
\node() at (-6,0.85) {$\gamma_{0^{+}}$};
\node() at (-3.9,0.85) {$\gamma_{0^{-}}$};
\draw [<-] (x) -- (z);  
\draw [<-] (y) -- (z);   
        \end{tikzpicture}
    \end{center}
    \end{figure}
\noindent from the vertex pair $(0^{+},0^{-})$ to vertex $n-1$. We regard the composition $\gamma\gamma^{-1}$ as a word of length $2$ from vertex $n-1$ (upper) to itself (lower), as illustrated in the following figure.
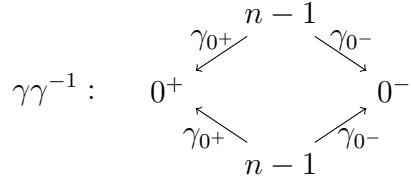
\begin{figure}[h]
    \begin{center}
        \begin{tikzpicture}
\node() at (-3,1) {$\gamma\gamma^{-1}:$};
    \node(1) at (0,2) {$n-1$};
    \node(2) at (-1.5,1) {$0^{+}$};
    \node(3) at (1.5,1) {$0^{-}$};
    \node(4) at (0,0) {$n-1$};
     \draw [->] (1) -- (2);   
\draw [->] (1) -- (3);  
\draw [->] (4) -- (2);   
\draw [->] (4) -- (3); 
\node() at (-0.9,1.65) {$\gamma_{0^{+}}$};
\node() at (0.95,1.65) {$\gamma_{0^{-}}$};
\node() at (-1,0.35) {$\gamma_{0^{+}}$};
\node() at (1.05,0.35) {$\gamma_{0^{-}}$};
        \end{tikzpicture}
    \end{center}
    \end{figure}

For simplicity, we denote by $\gamma_{\textup{max}}$ the element $$\gamma_{\textup{max}}:=\gamma_1\gamma_2\cdots\gamma_{n-2}\gamma\gamma^{-1}\gamma_{n-2}^{-1}\cdots\gamma_2^{-1}\gamma_1^{-1},$$
and set
$$\gamma_{\textup{semi}+}:=(\gamma_{\textup{max}})_{\leq n-2}\cdot\gamma_{0^+},\;\;\gamma_{\textup{semi}-}:=(\gamma_{\textup{max}})_{\leq n-2}\cdot\gamma_{0^-}.$$
We define the {\emph{truncation set}} $\Psi$ by
$$\Psi:=\{(\gamma_{\textup{max}})_{\leq k}|0\leq k\leq 2n-2\}\cup\{\gamma_{\textup{semi}+},\gamma_{\textup{semi}-}\}$$
and the {\emph{body set}} $\Theta$ by
$$\Theta:=\{\theta(\delta)=\varepsilon_1^{\delta_1}\gamma_{\textup{max}}\varepsilon_1^{\delta_2}\gamma_{\textup{max}}\varepsilon_1^{\delta_3}\cdots \gamma_{\textup{max}}\varepsilon_1^{\delta_k}\mid\delta=\delta_1\delta_2\cdots\delta_k\in\mathcal{I}\},$$
where $\mathcal{I}$ is the index set defined in \eqref{Index I}.

\begin{defn}\label{DefnExSt}
    % (1) A word $w$ is called an {\emph{extended string}} (abbr. {\emph{ex-string}}) with index $\delta(w)$ in $Q(\widetilde{CD}_{n})$ if it can be expressed in the form $\psi^{-1}\cdot\theta\cdot\psi'$ with $\psi,\psi'\in \Psi$ and $\theta=\theta\big(\delta(w)\big)\in\Theta$.
(1) A word $w$ is called an {\emph{extended string}} (abbr. {\emph{ex-string}}) in $Q(\widetilde{CD}_{n})$ if it can be expressed in the form $\psi^{-1}\cdot\theta\cdot\psi'$ with $\psi,\psi'\in \Psi$ and $\theta\in\Theta$.
    If $\theta=\theta(\delta)$ for some $\delta\in\mathcal{I}$, then $\delta$ is called the index of $w$, and denoted by $\delta(w)$.
    
        (2) A non-trivial ex-string $w$ is called an {\emph{extended band}} (abbr. {\emph{ex-band}}) in $Q(\widetilde{CD}_{n})$ if it has the same source and target, and each power $w^r$ is an ex-string, but $w$ is not a proper power of a shorter ex-string. 
\end{defn}

\begin{rem}
The full subquiver obtained from $Q(\widetilde{CD}_{n})$ by deleting the vertices $0^{+}$, $0^{-}$ and the arrows $\gamma_{0^{+}},\gamma_{0^{-}}$ defines a string algebra. The symbols $\gamma$ and $\gamma^{-1}$ are introduced to connect the two branches through $0^{\pm}$; they are not themselves arrows or formal inverses. Accordingly, the notion of an ex-string is only analogous to the usual notion of a string. In particular, $\gamma^{-1}\gamma$ cannot appear as a subword in an extended string, whereas $\gamma\gamma^{-1}$ can. 
\end{rem}

Denote by $\textup{Ex-St}\big(H(\widetilde{CD}_{n})\big)$ \big(resp. $\textup{Ex-Ba}(H(\widetilde{CD}_{n}))$\big) the set of all ex-strings (resp. all ex-bands) in $H(\widetilde{CD}_{n})$. On $\textup{Ex-St}\big(H(\widetilde{CD}_{n})\big)$, we define the equivalence relation by $w\sim w^{-1}$ for each ex-string $w$, and choose a complete set {\emph{$\overline{\textup{Ex-St}}\big(H(\widetilde{CD}_{n})\big)$}} of representatives relative to this equivalence relation. On $\textup{Ex-Ba}\big(H(\widetilde{CD}_{n})\big)$, we define the equivalence relation by $w\sim w^{-1}$ and $w\sim w_{(i)}$ for each ex-band $w$ and each $i\in\mathbb{Z}$, and choose a complete set {\emph{$\overline{\textup{Ex-Ba}}\big(H(\widetilde{CD}_{n})\big)$}} of representatives relative to this equivalence relation. Obviously, any ex-band is equivalent to 
$$\theta(\delta)\cdot\gamma_{\textup{max}}=\varepsilon_1^{\delta_1}\gamma_{\textup{max}}\varepsilon_1^{\delta_2}\gamma_{\textup{max}}\cdots\gamma_{\textup{max}}\varepsilon_1^{\delta_k}\gamma_{\textup{max}}$$
for some $\delta=\delta_1\delta_2\cdots\delta_k\in\mathcal{I}$.

The following ex-strings play an important role in parametrizing the connected components of the Auslander--Reiten quiver of the category of $H(\widetilde{CD}_{n})$-modules.

\begin{defn}\label{2-ex-string} Let $\theta\in\Theta$.
\begin{enumerate}
%     \item An ex-string of the form
% $\gamma_{\textup{max}}\cdot\theta\cdot\gamma_{\textup{max}}$ is called a {\emph{full ex-string}}. An ex-string of the form 
% $\gamma_{\textup{semi}\pm}^{-1}\cdot\theta\cdot\gamma_{\textup{max}}$ is called a {\emph{$1$-ex-string}}. An ex-string of the form $\gamma_{\textup{semi}\pm}^{-1}\cdot\theta\cdot\gamma_{\textup{semi}\pm}$ is called a {\emph{$2$-ex-string}}.
\item An ex-string of the form
$\gamma_{\textup{max}}\cdot\theta\cdot\gamma_{\textup{max}}$
(resp. $\gamma_{\textup{semi}\pm}^{-1}\cdot\theta\cdot\gamma_{\textup{max}}$  , or  $\gamma_{\textup{semi}\pm}^{-1}\cdot\theta\cdot\gamma_{\textup{semi}\pm}$) is called a {\emph{full ex-string}} (resp. {\emph{$1$-ex-string}}, or {\emph{$2$-ex-string}}).
\item The pair 
$$(\gamma_{\textup{semi}+}^{-1}\cdot\theta\cdot\gamma_{\textup{max}},\;\gamma_{\textup{semi}-}^{-1}\cdot\theta\cdot\gamma_{\textup{max}})$$
is called a {\emph{$1$-ex-string pair}}, and the pairs
$$\begin{array}{c}
(\gamma_{\textup{semi}+}^{-1}\cdot\theta\cdot\gamma_{\textup{semi}+},\;\gamma_{\textup{semi}-}^{-1}\cdot\theta\cdot\gamma_{\textup{semi}-}),\\(\gamma_{\textup{semi}-}^{-1}\cdot\theta\cdot\gamma_{\textup{semi}+},\;\gamma_{\textup{semi}+}^{-1}\cdot\theta\cdot\gamma_{\textup{semi}-})\end{array}$$
are called {\emph{$2$-ex-string pairs}}.
\end{enumerate}
\end{defn} 

\begin{example}
Let $n=2$ and $\delta=\delta_1\delta_2\cdots\delta_k\in\mathcal{I}$. We illustrate the following four ex-strings in Figure \ref{ExampleOfExSt}:
$$\textup{(a)\;}(\gamma\gamma^{-1})\cdot\theta(\delta)\cdot(\gamma\gamma^{-1}),\textup{\;(b)\;} \gamma_{0^{+}}^{-1}\cdot\theta(\delta)\cdot(\gamma\gamma^{-1}),\textup{\;(c)\;}\gamma_{0^{+}}^{-1}\cdot\theta(\delta)\cdot\gamma_{0^{+}},\textup{\;(d)\;} \gamma_{0^{-}}^{-1}\cdot\theta(\delta)\cdot\gamma_{0^{+}}.$$
Here, for each $1\leq l\leq k$, the symbol $\varepsilon_1^{\delta_l}$ denotes a downward arrow if $\delta_l=1$ and an upward arrow if $\delta_l=-1$.    
\end{example}
\begin{figure}[h]
 \centering
    \begin{tikzpicture}
\node (1) at (0,0) {\tiny{$1$}};
\node (2) at (-1,-0.6) {\tiny{$0^{+}$}};
\node (3) at (1,-0.6) {\tiny{$0^{-}$}};
\node (4) at (0,-1.2) {\tiny{$1$}};
\draw[->] (1) to (2);
\draw[->] (1) to (3);
\draw[->] (4) to (2);
\draw[->] (4) to (3); 
\node () at (-0.65,-0.2) {\tiny{$\gamma_{0^{+}}$}};
\node () at (0.65,-0.2) {\tiny{$\gamma_{0^{-}}$}};
\node () at (-0.65,-1) {\tiny{$\gamma_{0^{+}}$}};
\node () at (0.65,-1) {\tiny{$\gamma_{0^{-}}$}};
\node (5) at (0,-2.2) {\tiny{$1$}};
\node (6) at (-1,-2.8) {\tiny{$0^{+}$}};
\node (7) at (1,-2.8) {\tiny{$0^{-}$}};
\node (8) at (0,-3.4) {\tiny{$1$}};
\draw[->] (5) to (6);
\draw[->] (5) to (7);
\draw[->] (8) to (6);
\draw[->] (8) to (7);
\draw[-] (4) to (5);
\node () at (0.3,-1.7) {\tiny{$\varepsilon_{1}^{\delta_1}$}};
\node () at (-0.65,-2.4) {\tiny{$\gamma_{0^{+}}$}};
\node () at (0.65,-2.4) {\tiny{$\gamma_{0^{-}}$}};
\node () at (-0.65,-3.2) {\tiny{$\gamma_{0^{+}}$}};
\node () at (0.65,-3.2) {\tiny{$\gamma_{0^{-}}$}};
\node () at (0.3,-3.9) {\tiny{$\varepsilon_{1}^{\delta_{2}}$}};
\node (9) at (0,-4.4) {\tiny{$1$}};
\draw[-] (8) to (9);
\node (0) at (0,-5.4) {\tiny{$1$}};
\draw[dotted] (9) to (0);
\node () at (0.45,-5.9) {\tiny{$\varepsilon_{1}^{\delta_{k-1}}$}};
\node (a) at (0,-6.4) {\tiny{$1$}};
\node (b) at (-1,-7) {\tiny{$0^{+}$}};
\node (c) at (1,-7) {\tiny{$0^{-}$}};
\node (d) at (0,-7.6) {\tiny{$1$}};
\node () at (-0.65,-6.6) {\tiny{$\gamma_{0^{+}}$}};
\node () at (0.65,-6.6) {\tiny{$\gamma_{0^{-}}$}};
\node () at (-0.65,-7.4) {\tiny{$\gamma_{0^{+}}$}};
\node () at (0.65,-7.4) {\tiny{$\gamma_{0^{-}}$}};
\draw[->] (a) to (b);
\draw[->] (a) to (c);
\draw[->] (d) to (b);
\draw[->] (d) to (c);
\draw[-] (0) to (a);
\node (aa) at (0,-8.6) {\tiny{$1$}};
\node (ab) at (-1,-9.2) {\tiny{$0^{+}$}};
\node (ac) at (1,-9.2) {\tiny{$0^{-}$}};
\node (ad) at (0,-9.8) {\tiny{$1$}};
\node () at (-0.65,-8.8) {\tiny{$\gamma_{0^{+}}$}};
\node () at (0.65,-8.8) {\tiny{$\gamma_{0^{-}}$}};
\node () at (-0.65,-9.6) {\tiny{$\gamma_{0^{+}}$}};
\node () at (0.65,-9.6) {\tiny{$\gamma_{0^{-}}$}};
\node () at (0.3,-8.1) {\tiny{$\varepsilon_{1}^{\delta_k}$}};
\draw[->] (aa) to (ab);
\draw[->] (aa) to (ac);
\draw[->] (ad) to (ab);
\draw[->] (ad) to (ac);
\draw[-] (d) to (aa);
\node () at (0,-10.4) {(a)};
\node (q2) at (2.5,-0.6) {\tiny{$0^{+}$}};
\node (q4) at (3.5,-1.2) {\tiny{$1$}};
\draw[->] (q4) to (q2);
\node () at (2.85,-1) {\tiny{$\gamma_{0^{+}}$}};
\node (q5) at (3.5,-2.2) {\tiny{$1$}};
\node (q6) at (2.5,-2.8) {\tiny{$0^{+}$}};
\node (q7) at (4.5,-2.8) {\tiny{$0^{-}$}};
\node (q8) at (3.5,-3.4) {\tiny{$1$}};
\draw[->] (q5) to (q6);
\draw[->] (q5) to (q7);
\draw[->] (q8) to (q6);
\draw[->] (q8) to (q7);
\draw[-] (q4) to (q5);
\node () at (3.8,-1.7) {\tiny{$\varepsilon_{1}^{\delta_1}$}};
\node () at (2.85,-2.4) {\tiny{$\gamma_{0^{+}}$}};
\node () at (4.15,-2.4) {\tiny{$\gamma_{0^{-}}$}};
\node () at (2.85,-3.2) {\tiny{$\gamma_{0^{+}}$}};
\node () at (4.15,-3.2) {\tiny{$\gamma_{0^{-}}$}};
\node () at (3.8,-3.9) {\tiny{$\varepsilon_{1}^{\delta_{2}}$}};
\node (q9) at (3.5,-4.4) {\tiny{$1$}};
\draw[-] (q8) to (q9);
\node (q0) at (3.5,-5.4) {\tiny{$1$}};
\draw[dotted] (q9) to (q0);
\node () at (3.95,-5.9) {\tiny{$\varepsilon_{1}^{\delta_{k-1}}$}};
\node (qa) at (3.5,-6.4) {\tiny{$1$}};
\node (qb) at (2.5,-7) {\tiny{$0^{+}$}};
\node (qc) at (4.5,-7) {\tiny{$0^{-}$}};
\node (qd) at (3.5,-7.6) {\tiny{$1$}};
\node () at (2.85,-6.6) {\tiny{$\gamma_{0^{+}}$}};
\node () at (4.15,-6.6) {\tiny{$\gamma_{0^{-}}$}};
\node () at (2.85,-7.4) {\tiny{$\gamma_{0^{+}}$}};
\node () at (4.15,-7.4) {\tiny{$\gamma_{0^{-}}$}};
\draw[->] (qa) to (qb);
\draw[->] (qa) to (qc);
\draw[->] (qd) to (qb);
\draw[->] (qd) to (qc);
\draw[-] (q0) to (qa);
\node (qaa) at (3.5,-8.6) {\tiny{$1$}};
\node (qab) at (2.5,-9.2) {\tiny{$0^{+}$}};
\node (qac) at (4.5,-9.2) {\tiny{$0^{-}$}};
\node (qad) at (3.5,-9.8) {\tiny{$1$}};
\node () at (2.85,-8.8) {\tiny{$\gamma_{0^{+}}$}};
\node () at (4.15,-8.8) {\tiny{$\gamma_{0^{-}}$}};
\node () at (2.85,-9.6) {\tiny{$\gamma_{0^{+}}$}};
\node () at (4.15,-9.6) {\tiny{$\gamma_{0^{-}}$}};
\node () at (3.8,-8.1) {\tiny{$\varepsilon_{1}^{\delta_k}$}};
\draw[->] (qaa) to (qab);
\draw[->] (qaa) to (qac);
\draw[->] (qad) to (qab);
\draw[->] (qad) to (qac);
\draw[-] (qd) to (qaa);
\node () at (3.5,-10.4) {(b)};

\node (w2) at (6,-0.6) {\tiny{$0^{+}$}};
\node (w4) at (7,-1.2) {\tiny{$1$}};
\draw[->] (w4) to (w2);
\node () at (6.35,-1) {\tiny{$\gamma_{0^{+}}$}};
\node (w5) at (7,-2.2) {\tiny{$1$}};
\node (w6) at (6,-2.8) {\tiny{$0^{+}$}};
\node (w7) at (8,-2.8) {\tiny{$0^{-}$}};
\node (w8) at (7,-3.4) {\tiny{$1$}};
\draw[->] (w5) to (w6);
\draw[->] (w5) to (w7);
\draw[->] (w8) to (w6);
\draw[->] (w8) to (w7);
\draw[-] (w4) to (w5);
\node () at (7.3,-1.7) {\tiny{$\varepsilon_{1}^{\delta_1}$}};
\node () at (6.35,-2.4) {\tiny{$\gamma_{0^{+}}$}};
\node () at (7.65,-2.4) {\tiny{$\gamma_{0^{-}}$}};
\node () at (6.35,-3.2) {\tiny{$\gamma_{0^{+}}$}};
\node () at (7.65,-3.2) {\tiny{$\gamma_{0^{-}}$}};
\node () at (7.3,-3.9) {\tiny{$\varepsilon_{1}^{\delta_{2}}$}};
\node (w9) at (7,-4.4) {\tiny{$1$}};
\draw[-] (w8) to (w9);
\node (w0) at (7,-5.4) {\tiny{$1$}};
\draw[dotted] (w9) to (w0);
\node () at (7.45,-5.9) {\tiny{$\varepsilon_{1}^{\delta_{k-1}}$}};
\node (wa) at (7,-6.4) {\tiny{$1$}};
\node (wb) at (6,-7) {\tiny{$0^{+}$}};
\node (wc) at (8,-7) {\tiny{$0^{-}$}};
\node (wd) at (7,-7.6) {\tiny{$1$}};
\node () at (6.35,-6.6) {\tiny{$\gamma_{0^{+}}$}};
\node () at (7.65,-6.6) {\tiny{$\gamma_{0^{-}}$}};
\node () at (6.35,-7.4) {\tiny{$\gamma_{0^{+}}$}};
\node () at (7.65,-7.4) {\tiny{$\gamma_{0^{-}}$}};
\draw[->] (wa) to (wb);
\draw[->] (wa) to (wc);
\draw[->] (wd) to (wb);
\draw[->] (wd) to (wc);
\draw[-] (w0) to (wa);
\node (waa) at (7,-8.6) {\tiny{$1$}};
\node (wab) at (6,-9.2) {\tiny{$0^{+}$}};
\node () at (6.35,-8.8) {\tiny{$\gamma_{0^{+}}$}};
\node () at (7.3,-8.1) {\tiny{$\varepsilon_{1}^{\delta_k}$}};
\draw[->] (waa) to (wab);
\draw[-] (wd) to (waa);
\node () at (7,-10.4) {(c)};
\node (zq3) at (11.5,-0.6) {\tiny{$0^{-}$}};
\node (zq4) at (10.5,-1.2) {\tiny{$1$}};
\draw[->] (zq4) to (zq3); 
\node () at (11.15,-1) {\tiny{$\gamma_{0^{-}}$}};
\node (zq5) at (10.5,-2.2) {\tiny{$1$}};
\node (zq6) at (9.5,-2.8) {\tiny{$0^{+}$}};
\node (zq7) at (11.5,-2.8) {\tiny{$0^{-}$}};
\node (zq8) at (10.5,-3.4) {\tiny{$1$}};
\draw[->] (zq5) to (zq6);
\draw[->] (zq5) to (zq7);
\draw[->] (zq8) to (zq6);
\draw[->] (zq8) to (zq7);
\draw[-] (zq4) to (zq5);
\node () at (10.8,-1.7) {\tiny{$\varepsilon_{1}^{\delta_1}$}};
\node () at (9.85,-2.4) {\tiny{$\gamma_{0^{+}}$}};
\node () at (11.15,-2.4) {\tiny{$\gamma_{0^{-}}$}};
\node () at (9.85,-3.2) {\tiny{$\gamma_{0^{+}}$}};
\node () at (11.15,-3.2) {\tiny{$\gamma_{0^{-}}$}};
\node () at (10.8,-3.9) {\tiny{$\varepsilon_{1}^{\delta_{2}}$}};
\node (zq9) at (10.5,-4.4) {\tiny{$1$}};
\draw[-] (zq8) to (zq9);
\node (zq0) at (10.5,-5.4) {\tiny{$1$}};
\draw[dotted] (zq9) to (zq0);
\node () at (10.95,-5.9) {\tiny{$\varepsilon_{1}^{\delta_{k-1}}$}};
\node (zqa) at (10.5,-6.4) {\tiny{$1$}};
\node (zqb) at (9.5,-7) {\tiny{$0^{+}$}};
\node (zqc) at (11.5,-7) {\tiny{$0^{-}$}};
\node (zqd) at (10.5,-7.6) {\tiny{$1$}};
\node () at (9.85,-6.6) {\tiny{$\gamma_{0^{+}}$}};
\node () at (11.15,-6.6) {\tiny{$\gamma_{0^{-}}$}};
\node () at (9.85,-7.4) {\tiny{$\gamma_{0^{+}}$}};
\node () at (11.15,-7.4) {\tiny{$\gamma_{0^{-}}$}};
\draw[->] (zqa) to (zqb);
\draw[->] (zqa) to (zqc);
\draw[->] (zqd) to (zqb);
\draw[->] (zqd) to (zqc);
\draw[-] (zq0) to (zqa);
\node (zqaa) at (10.5,-8.6) {\tiny{$1$}};
\node (zqab) at (9.5,-9.2) {\tiny{$0^{+}$}};
\node () at (9.85,-8.8) {\tiny{$\gamma_{0^{+}}$}};
\node () at (10.8,-8.1) {\tiny{$\varepsilon_{1}^{\delta_k}$}};
\draw[->] (zqaa) to (zqab);
\draw[-] (zqd) to (zqaa);
\node () at (10.5,-10.4) {(d)};
    \end{tikzpicture}
    \caption{Four important ex-strings}
    \label{ExampleOfExSt}
\end{figure}
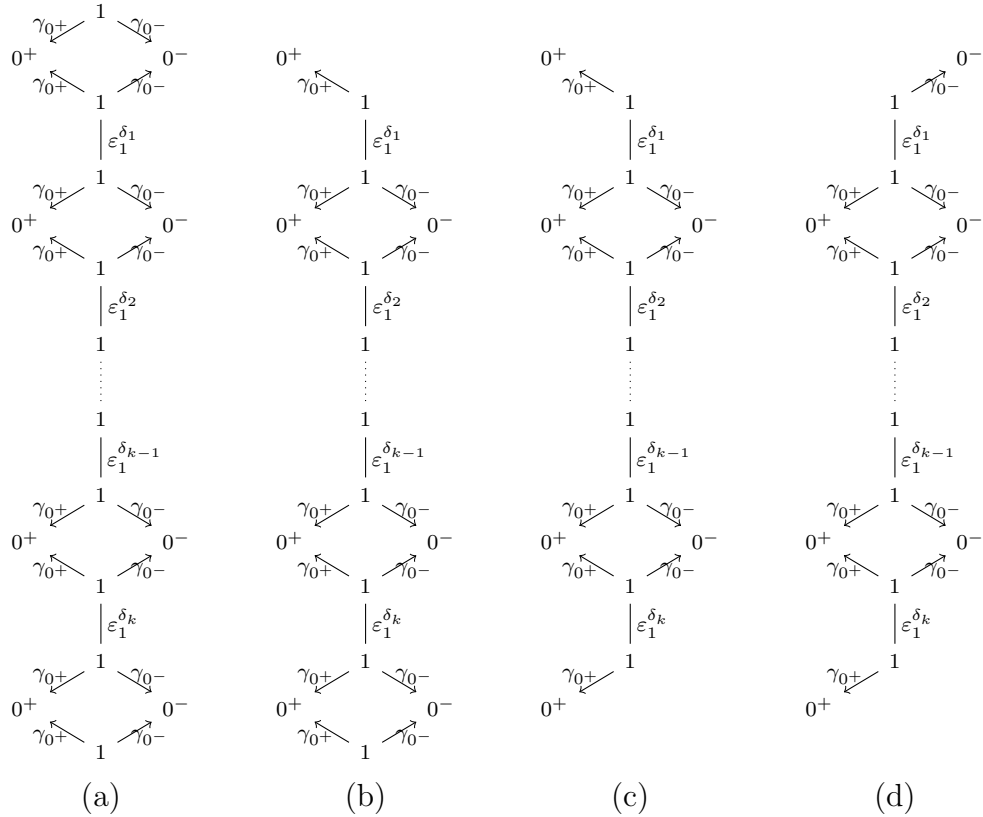

\subsection{Ex-string modules and ex-band modules}
This subsection introduces $H(\widetilde{CD}_{n})$-modules associated with ex-strings and ex-bands. Consider a full ex-string $\gamma_{\textup{max}}\cdot\theta(\delta)\cdot\gamma_{\textup{max}}$ with $\delta=\delta_1\delta_2\cdots\delta_k\in\mathcal{I}$; the associated ``walk'' (called ex-walk below) is depicted in Figure~\ref{WalkForExSt}. More generally, any ex-string $w=\psi^{-1}\cdot\theta(\delta)\cdot\psi'$ with truncations $\psi,\psi'\in\Psi$ corresponds to an ex-walk obtained from that of $\gamma_{\textup{max}}\cdot\theta(\delta)\cdot\gamma_{\textup{max}}$ by deleting the appropriate arrows and vertices. Let $w$ be an ex-string or an ex-band. If $x_i$ is a vertex in the corresponding ex-walk, we simply write $x_i\in w$ for convenience. 
\begin{figure}[h]
    \begin{center}
        \begin{tikzpicture}
\node(1) at(0,0){$x_1$};
\node(2) at(1.5,0){$x_2$};
\node(3) at(3,0){};
\node() at (0.75,0.2){\tiny{$\gamma_1$}};
\node() at (2.25,0.2){\tiny{$\gamma_2$}};
\draw[->] (1)--(2);
\draw[->] (2)--(3);
\node() at (4,0){$\cdots$};
\node(n-3) at(5,0){};
\node(n-2) at(6.6,0){$x_{n-2}$};
\node(n-1) at(8.5,0){$x_{n-1}$};
\node() at (5.6,0.2){\tiny{$\gamma_{n-3}$}};
\node() at (7.45,0.2){\tiny{$\gamma_{n-2}$}};
\draw[->] (n-3)--(n-2);
\draw[->] (n-2)--(n-1);
\node(n) at(7.5,-0.8){$x_{n}$};
\node(n+1) at(9.8,-0.8){$x_{n+1}$};
\node() at (8.3,-0.5){\tiny{$\gamma_{0^{+}}$}};
\node() at (8.3,-1.1){\tiny{$\gamma_{0^{+}}$}};
\node() at (9,-0.5){\tiny{$\gamma_{0^{-}}$}};
\node() at (9,-1.1){\tiny{$\gamma_{0^{-}}$}};
\draw[->] (n-1)--(n);
\draw[->] (n-1)--(n+1);
\node(2n) at(0,-1.6){$x_{2n}$};
\node(2n-1) at(1.5,-1.6){$x_{2n-1}$};
\node(2n-2) at(3,-1.6){};
\node() at (0.65,-1.4){\tiny{$\gamma_1$}};
\node() at (2.45,-1.4){\tiny{$\gamma_2$}};
\draw[->] (2n)--(2n-1);
\draw[->] (2n-1)--(2n-2);
\node() at (4,-1.6){$\cdots$};
\node(n+4) at(5,-1.6){};
\node(n+3) at(6.6,-1.6){$x_{n+3}$};
\node(n+2) at(8.5,-1.6){$x_{n+2}$};
\node() at (5.6,-1.4){\tiny{$\gamma_{n-3}$}};
\node() at (7.45,-1.4){\tiny{$\gamma_{n-2}$}};
\draw[->] (n+3)--(n+2);
\draw[->] (n+4)--(n+3);
\draw[<-] (n)--(n+2);
\draw[<-] (n+1)--(n+2);
\node(2n+1)at(0,-3){$x_{2n+1}$};
\draw[-] (2n)--(2n+1);
\node () at (-0.25,-2.3) {\tiny{$\varepsilon_1^{\delta_{1}}$}};
\node(2n+2)at(1.5,-3){$x_{2n+2}$};
\draw[->] (2n+1)--(2n+2);
\node() at (0.75,-2.8){\tiny{$\gamma_1$}};
\node(2n+3) at(3,-3){};
\node() at (2.45,-2.8){\tiny{$\gamma_2$}};
\draw[->] (2n+2)--(2n+3);
\node() at (4,-3){$\cdots$};
\node(3n-3) at(5,-3){};
\node(3n-2) at(6.6,-3){$x_{3n-2}$};
\node(3n-1) at(8.5,-3){$x_{3n-1}$};
\node() at (5.6,-2.8){\tiny{$\gamma_{n-3}$}};
\node() at (7.45,-2.8){\tiny{$\gamma_{n-2}$}};
\draw[->] (3n-3)--(3n-2);
\draw[->] (3n-2)--(3n-1);
\node() at (8.3,-3.5){\tiny{$\gamma_{0^{+}}$}};
\node() at (8.3,-4.1){\tiny{$\gamma_{0^{+}}$}};
\node() at (9,-3.5){\tiny{$\gamma_{0^{-}}$}};
\node() at (9,-4.1){\tiny{$\gamma_{0^{-}}$}};
\node(3n) at(7.5,-3.8){$x_{3n}$};
\node(3n+1) at(9.8,-3.8){$x_{3n+1}$};
\draw[<-] (3n)--(3n-1);
\draw[<-] (3n+1)--(3n-1);
\draw[->] (3n-1)--(3n-2);
\node() at (4,-4.6){$\cdots$};
\node(3n+4) at(5,-4.6){};
\node(3n+3) at(6.6,-4.6){$x_{3n+3}$};
\node(3n+2) at(8.5,-4.6){$x_{3n+2}$};
\node() at (5.6,-4.4){\tiny{$\gamma_{n-3}$}};
\node() at (7.45,-4.4){\tiny{$\gamma_{n-2}$}};
\draw[->] (3n+3)--(3n+2);
\draw[->] (3n+4)--(3n+3);
\draw[<-] (3n)--(3n+2);
\draw[<-] (3n+1)--(3n+2);
\node(4n) at(0,-4.6){$x_{4n}$};
\node(4n-1) at(1.5,-4.6){$x_{4n-1}$};
\node(4n-2) at(3,-4.6){};
\node() at (0.65,-4.4){\tiny{$\gamma_1$}};
\node() at (2.45,-4.4){\tiny{$\gamma_2$}};
\draw[->] (4n)--(4n-1);
\draw[->] (4n-1)--(4n-2);
\node () at (-0.25,-5.3) {\tiny{$\varepsilon_1^{\delta_{2}}$}};
\node(4n+1)at(0,-6){$x_{4n+1}$};
\draw[-] (4n)--(4n+1);
\node(4n+2)at(1.5,-6){$x_{4n+2}$};
\draw[->] (4n+1)--(4n+2);
\node() at (0.75,-5.8){\tiny{$\gamma_1$}};
\node(4n+3) at(3,-6){};
\node() at (2.45,-5.8){\tiny{$\gamma_2$}};
\draw[->] (4n+2)--(4n+3);
\node()at (6,-6){$\cdots\;\;\;\cdots\;\;\;\cdots$};
 \node() at (8.5,-6.5) {$\vdots$};
 \node() at (8.5,-7) {$\vdots$};
 \node() at (8.5,-7.5) {$\vdots$};
 \node()at (6,-8){$\cdots\;\;\;\cdots\;\;\;\cdots$};
\node(kn-1)at(1.5,-8){$x_{2kn-1}$};
\node() at (0.65,-7.8){\tiny{$\gamma_1$}};
 \node(kn-2) at(3,-8){};
 \node() at (2.45,-7.8){\tiny{$\gamma_2$}};
\node(kn)at(0,-8){$x_{2kn}$};
 \draw[->] (kn)--(kn-1);
 \draw[->] (kn-1)--(kn-2);
 \node () at (-0.25,-8.7) {\tiny{$\varepsilon_1^{\delta_{k}}$}};
\node(kn+1)at(0,-9.4){$x_{2kn+1}$};
 \node(kn+2)at(1.8,-9.4){$x_{2kn+2}$};
 \node(kn+3) at(3,-9.4){};
 \draw[->] (kn+1)--(kn+2);
 \draw[->] (kn+2)--(kn+3);
 \draw[-] (kn)--(kn+1);
 \node() at (0.9,-9.2){\tiny{$\gamma_1$}};
 \node() at (2.7,-9.2){\tiny{$\gamma_2$}};
 \node() at (4,-9.4){$\cdots$};
 \node(kn+n-3) at(5,-9.4){};
 \node(kn+n-2) at(6.6,-9.4){$x_{2kn+n-2}$};
 \node(kn+n-1) at(9,-9.4){$x_{2kn+n-1}$};
 \node() at (5.4,-9.2){\tiny{$\gamma_{n-3}$}};
 \node() at (7.8,-9.2){\tiny{$\gamma_{n-2}$}};
 \draw[->] (kn+n-3)--(kn+n-2);
 \draw[->] (kn+n-2)--(kn+n-1);
 \node(kn+n) at(7.6,-10.2){$x_{2kn+n}$};
 \node(kn+n+1) at(10.2,-10.2){$x_{2kn+n+1}$};
 \draw[->] (kn+n-1)--(kn+n+1);
 \draw[->] (kn+n-1)--(kn+n);
 \node() at (8.6,-9.9){\tiny{$\gamma_{0^{+}}$}};
 \node() at (8.6,-10.5){\tiny{$\gamma_{0^{+}}$}};
 \node() at (9.3,-9.9){\tiny{$\gamma_{0^{-}}$}};
 \node() at (9.3,-10.5){\tiny{$\gamma_{0^{-}}$}};
 \node(kn+n+4) at(5,-11){};
 \node(kn+n+3) at(6.6,-11){$x_{2kn+n+3}$};
\node(kn+n+2) at(9,-11){$x_{2kn+n+2}$};
 \draw[->] (kn+n+2)--(kn+n+1);
 \draw[->] (kn+n+2)--(kn+n);
 \draw[->] (kn+n+4)--(kn+n+3); \draw[->] (kn+n+3)--(kn+n+2);
 \node() at (5.4,-10.8){\tiny{$\gamma_{n-3}$}};
 \node() at (7.8,-10.8){\tiny{$\gamma_{n-2}$}};
 \node(kn+2n-2) at(3.5,-11){};
 \node(kn+2n) at(-0.2,-11)
 {$x_{2kn+2n}$};
  \node(kn+2n-1) at(2,-11)
 {$x_{2kn+2n-1}$};
 \draw[->] (kn+2n)--(kn+2n-1);
  \draw[->] (kn+2n-1)--(kn+2n-2);
 \node() at (0.8,-10.8){\tiny{$\gamma_{1}$}};
 \node() at (4,-11){$\cdots$};
  \node() at (3.2,-10.8){\tiny{$\gamma_2$}};
\end{tikzpicture}
    \end{center}
    \caption{The ex-walk associated to the full ex-string $\gamma_{\textup{max}}\cdot\theta(\delta_1\delta_2\cdots\delta_k)\cdot\gamma_{\textup{max}}$}
    \label{WalkForExSt}
\end{figure}

Let $w=\psi^{-1}\cdot\theta\cdot\psi'$ be an ex-walk with $\psi,\psi'\in \Psi$ and $\theta\in\Theta$. Now we define the $H(\widetilde{CD}_n)$-module $N(w)$, which is called the {\emph{ex-string module}}, via the representation $$\big((N(w)_u)_{u\in Q_0},(N(w)_{\alpha})_{\alpha\in Q_1}\big)$$ as follows. To each point $x_i$ in the ex-walk $w$, we associate the one-dimensional vector space ${{\bf k}}x_i$. The vector spaces are then defined by
$$N(w)_{u}=\left\{\begin{array}{ll}
    (\bigoplus\limits_{\substack{x_i\in w ;\;i\equiv u (\textup{mod}\; 2n)}}{{\bf k}}x_i)\oplus(\bigoplus\limits_{\substack{x_i\in w;\;i\equiv 1-u(\textup{mod}\; 2n)}}{{\bf k}}x_i) & \textup{if}\;\;1\leq u\leq n-1,  \\
    \bigoplus\limits_{\substack{x_i\in w ;\; i\equiv n(\textup{mod}\;2n)}}{{\bf k}}x_i & \textup{if}\;\;u=0^{+},  \\
     \bigoplus\limits_{\substack{x_i\in w ;\; i\equiv n+1(\textup{mod}\;2n)}}{{\bf k}}x_i & \textup{if}\;\;u=0^{-},
\end{array}\right.$$
and the linear maps are given by
$$N(w)_{\alpha}(x_i)=\left\{\begin{array}{ll}x_j & \textup{if there exists an arrow}\;\alpha:\;x_i\rightarrow x_j\;\textup{in}\;w,  \\
0 & \textup{otherwise}.
\end{array}\right.$$
It is clear that $N(w)$ is isomorphic to $N(w^{-1})$ as $H(\widetilde{CD}_{n})$-modules for any ex-string $w$.

Let $w$ be an ex-band and $\varphi:\;{{\bf k}}^s\rightarrow{{\bf k}}^s$ an automorphism. Then $w$ must be equivalent to the ex-band $$\theta(\delta)\cdot\gamma_{\textup{max}}=\varepsilon_1^{\delta_1}\gamma_{\textup{max}}\varepsilon_1^{\delta_2}\gamma_{\textup{max}}\cdots\gamma_{\textup{max}}\varepsilon_1^{\delta_k}\gamma_{\textup{max}}$$
for some index $\delta=\delta_1\delta_2\cdots\delta_k\in\mathcal{I}$, which can be seen as an ex-walk shown in Figure \ref{FigureOfExBa}. 
\begin{figure}[h]
    \centering
\begin{tikzpicture}
\node (1+) at (0,0) {\tiny{$x_0=x_{2kn}$}};
\node (2+) at (2,0) {\tiny{$x_{2kn-1}$}};
\node (3+) at (4,0) {};
\node () at (5,0) {$\cdots$};
\node (4+) at (6,0) {};
\node (5+) at (8,0) {\tiny{$x_{2kn-n+2}$}};
\node() at (7.5,-0.5){\tiny{$\gamma_{0^{+}}$}};
\node() at (7.5,-1.1){\tiny{$\gamma_{0^{+}}$}};
\node() at (8.6,-0.5){\tiny{$\gamma_{0^{-}}$}};
\node() at (8.6,-1.1){\tiny{$\gamma_{0^{-}}$}};
\node (1-) at (0,-1.6) {\tiny{$x_{2kn-2n+1}$}};
\node (2-) at (2,-1.6) {\tiny{$x_{2kn-2n+2}$}};
\node (3-) at (4,-1.6) {};
\node () at (5,-1.6) {$\cdots$};
\node (4-) at (6,-1.6) {};
\node (5-) at (8,-1.6) {\tiny{$x_{2kn-n-1}$}};
\node (6+) at (6.5,-0.8) {\tiny{$x_{2kn-n}$}};
\node (6-) at (9.5,-0.8) {\tiny{$x_{2kn-n+1}$}};
\draw[->] (1+) to (2+);
\draw[->] (2+) to (3+);
\draw[->] (4+) to (5+);
\draw[->] (5+) to (6+);
\draw[->] (5+) to (6-);
\draw[->] (1-) to (2-);
\draw[->] (2-) to (3-);
\draw[->] (4-) to (5-);
\draw[->] (5-) to (6+);
\draw[->] (5-) to (6-);
\node() at (1,0.2) {\tiny{$\gamma_1$}};
\node() at (3,0.2) {\tiny{$\gamma_2$}};
\node() at (7,0.2) {\tiny{$\gamma_{n-2}$}};
\node() at (1,-1.4) {\tiny{$\gamma_1$}};
\node() at (3,-1.4) {\tiny{$\gamma_2$}};
\node() at (6.75,-1.4) {\tiny{$\gamma_{n-2}$}};
\node (21+) at (0,-2.6) {\tiny{$x_{2kn-2n}$}};
\node (22+) at (2,-2.6) {\tiny{$x_{2kn-2n-1}$}};
\node (23+) at (4,-2.6) {};
\node () at (5,-2.6) {$\cdots$};
\node (24+) at (6,-2.6) {};
\node (25+) at (8,-2.6) {\tiny{$x_{2kn-3n+2}$}};
\node (21-) at (0,-4.2) {\tiny{$x_{2kn-4n+1}$}};
\node (22-) at (2,-4.2) {\tiny{$x_{2kn-4n+2}$}};
\node (23-) at (4,-4.2) {};
\node () at (5,-4.2) {$\cdots$};
\node (24-) at (6,-4.2) {};
\node (25-) at (8,-4.2) {\tiny{$x_{2kn-3n-1}$}};
\node (26+) at (6.5,-3.4) {\tiny{$x_{2kn-3n}$}};
\node (26-) at (9.5,-3.4) {\tiny{$x_{2kn-3n+1}$}};
\node() at (7.5,-3.1)
{\tiny{$\gamma_{0^{+}}$}};
\node() at (7.5,-3.7){\tiny{$\gamma_{0^{+}}$}};
\node() at (8.6,-3.1){\tiny{$\gamma_{0^{-}}$}};
\node() at (8.6,-3.7){\tiny{$\gamma_{0^{-}}$}};
\draw[->] (21+) to (22+);
\draw[->] (22+) to (23+);
\draw[->] (24+) to (25+);
\draw[->] (25+) to (26+);
\draw[->] (25+) to (26-);
\draw[->] (21-) to (22-);
\draw[->] (22-) to (23-);
\draw[->] (24-) to (25-);
\draw[->] (25-) to (26+);
\draw[->] (25-) to (26-);
\node() at (1,-2.4) {\tiny{$\gamma_1$}};
\node() at (3,-2.4) {\tiny{$\gamma_2$}};
\node() at (6.75,-2.4) {\tiny{$\gamma_{n-2}$}};
\node() at (1,-4) {\tiny{$\gamma_1$}};
\node() at (3,-4) {\tiny{$\gamma_2$}};
\node() at (6.75,-4) {\tiny{$\gamma_{n-2}$}};
\draw[-] (1-) to (21+);
\node () at (0.3,-2.1) {\tiny{$\varepsilon_1^{\delta_{k}}$}};
\node () at (0.45,-4.8) {\tiny{$\varepsilon_1^{\delta_{k-1}}$}};
\node () at (0.3,-6.6) {\tiny{$\varepsilon_1^{\delta_{2}}$}};
\draw[-] (0,-5.2) to (21-);
\node () at (5,-5.7) {$\cdots\;\;\;\;\;\cdots\;\;\;\;\;\cdots\;\;\;\;\;\cdots\;\;\;\;\;\cdots$};
\node (31+) at (0,-7.2) {\tiny{$x_{2n}$}};
\node (32+) at (2,-7.2) {\tiny{$x_{2n-1}$}};
\node (33+) at (4,-7.2) {};
\node () at (5,-7.2) {$\cdots$};
\node (34+) at (6,-7.2) {};
\node (35+) at (8,-7.2) {\tiny{$x_{n+2}$}};
\node (31-) at (0,-8.8) {\tiny{$x_1$}};
\node (32-) at (2,-8.8) {\tiny{$x_2$}};
\node (33-) at (4,-8.8) {};
\node () at (5,-8.8) {$\cdots$};
\node (34-) at (6,-8.8) {};
\node (35-) at (8,-8.8) {\tiny{$x_{n-1}$}};
\node (36+) at (6.5,-8) {\tiny{$x_{n}$}};
\node (36-) at (9.5,-8) {\tiny{$x_{n+1}$}};
\node() at (7.5,-7.7)
{\tiny{$\gamma_{0^{+}}$}};
\node() at (7.5,-8.3){\tiny{$\gamma_{0^{+}}$}};
\node() at (8.6,-7.7){\tiny{$\gamma_{0^{-}}$}};
\node() at (8.6,-8.3){\tiny{$\gamma_{0^{-}}$}};
\draw[-] (0,-6.2) to (31+);
\draw[->] (31+) to (32+);
\draw[->] (32+) to (33+);
\draw[->] (34+) to (35+);
\draw[->] (35+) to (36+);
\draw[->] (35+) to (36-);
\draw[->] (31-) to (32-);
\draw[->] (32-) to (33-);
\draw[->] (34-) to (35-);
\draw[->] (35-) to (36+);
\draw[->] (35-) to (36-);
\node() at (1,-7) {\tiny{$\gamma_1$}};
\node() at (3,-7) {\tiny{$\gamma_2$}};
\node() at (6.75,-7) {\tiny{$\gamma_{n-2}$}};
\node() at (1,-8.6) {\tiny{$\gamma_1$}};
\node() at (3,-8.6) {\tiny{$\gamma_2$}};
\node() at (6.75,-8.6) {\tiny{$\gamma_{n-2}$}};
\draw[-] (-0.4,-0.3) arc(153:207:9);
\node () at (-1,-4.5) {\tiny{$\varepsilon_1^{\delta_{1}}$}};
\end{tikzpicture}
\caption{The ex-walk associated to the ex-band $\theta(\delta)\cdot\gamma_{\textup{max}}$}
\label{FigureOfExBa}
\end{figure}
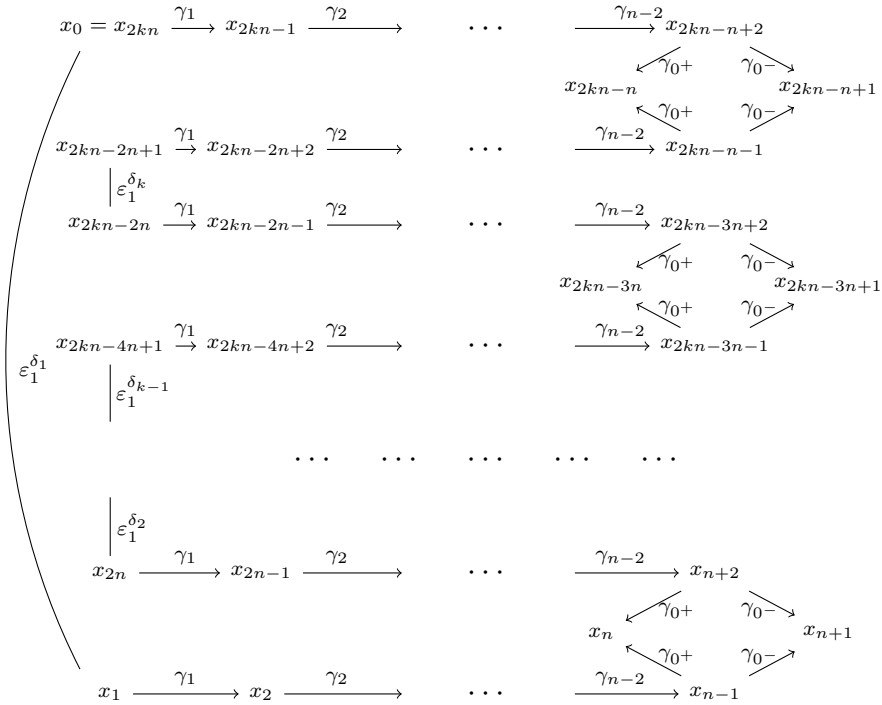

We define an {\emph{ex-band module}} $N(w,\varphi)$ as a representation
$$\big((N(w,\varphi)_u)_{u\in Q_0},(N(w,\varphi)_{\alpha})_{\alpha\in Q_1}\big)$$
of $Q$ as follows. To each point $x_i$ in the ex-walk $w$, where $i=j+2ln$ with $0\leq j<{2n}$, we associate a vector space $U_{j,l}\cong{{\bf k}}^s$. The vector spaces are given by 
$$N(w,\varphi)_u=\left\{\begin{array}{ll}
(\bigoplus_{l=0}^{k-1} U_{u,l})\oplus(\bigoplus_{l=0}^{k-1} U_{2n+1-u,l}) & 1\leq u\leq n-1,  \\
\bigoplus_{l=0}^{k-1} U_{n,l}     & u=0^{+}, \\
\bigoplus_{l=0}^{k-1} U_{n+1,l}     & u=0^{-}, 
\end{array}\right.$$
and the linear maps are defined by
$$N(w,\varphi)_{\varepsilon_1}|_{U_{j,l}}=\left\{
\begin{array}{ll}
\varphi & \textup{if}\;(j,l)=(0,0),\delta_1=1,\\
\varphi^{-1} & \textup{if}\;(j,l)=(1,0),\delta_1=-1,\\
\operatorname{id} & \textup{if there is an arrow}\; \varepsilon_1\;\textup{with source}\; x_{2ln+j}\;\textup{and}\;2ln+j\neq0,1\\
0&\textup{otherwise},
\end{array}\right.$$
and 
$$N(w,\varphi)_{\gamma_i}|_{U_{j,l}}=\left\{
\begin{array}{ll}
\operatorname{id} & \textup{if there is an arrow}\; \gamma_i:\;x_{2ln+j}\rightarrow x_{2ln+j'} \textup{\;with\;} j'-j\neq 2,\\
-\operatorname{id} & \textup{if there is an arrow}\; \gamma_i:\;x_{2ln+j}\rightarrow x_{2ln+j+2},\\
0&\textup{otherwise}.
\end{array}
\right.$$
Obviously, the representation $N(w,\varphi)$ is an $H(\widetilde{CD}_{n})$-module.

\begin{rem}\label{remark3.4}
%\begin{enumerate}
 (1) We caution that not all ex-string modules are indecomposable. For instance,  $N(\gamma\gamma^{-1})\cong N(\gamma_{0^{+}})\oplus N(\gamma^{-1}_{0^{-}})$ by Lemma \ref{A3caseIso} (2). 

 (2) The ex-string modules and ex-band modules play a crucial role in describing the Auslander--Reiten quiver of the category of $H(\widetilde{CD}_{n})$-modules. However, not every $H(\widetilde{CD}_{n})$-module falls into these classes. For instance, let $N$ be the $H(\widetilde{CD}_{2})$-module defined by     
    $$N_1=\textup{span}\{x_1,x_2,x_3,x_4\},\;\;\;N_{0^{+}}=\textup{span}\{y_1,y_2\},\;\;\;N_{0^{-}}=\textup{span}\{z_1,z_2\},$$
    with the action of $H(\widetilde{CD}_{2})$ given by
    $$\varepsilon_1:x_1\mapsto x_2,x_4\mapsto  x_3;\;\;\gamma_{0^+}:x_1\mapsto  y_1,x_2\mapsto y_2,x_3\mapsto y_2;\;\;\gamma_{0^-}:x_3\mapsto  z_1,x_4\mapsto  z_2.$$
    Then $N$ lies on the second layer of a non-homogeneous tube of rank $2$, and is neither an ex-string module nor an ex-band module, see Figure \ref{NStNBdmodule}.
\begin{figure}[h]
   \begin{center}
       \begin{tikzpicture}
        \node(x1) at (0,0) {\tiny{$1$}};
        \node(y1) at (-1.5,0.4) {\tiny{$0^{+}$}};
        \node(x2) at (0,-0.8) {\tiny{$1$}};
        \node(y2) at (-1.5,-1.2) {\tiny{$0^{+}$}};
        \node(z1) at (1.5,-1.2) {\tiny{$0^{-}$}};
        \node(x3) at (0,-1.6) {\tiny{$1$}};
        \node(x4) at (0,-2.4) {\tiny{$1$}};
        \node(z2) at (1.5,-2.8) {\tiny{$0^{-}$}};
        \draw[->](x1)to(x2);
        \draw[->](x4)to(x3);
        \draw[->](x1)to(y1);
        \draw[->](x2)to(y2);
        \draw[->](x3)to(y2);
        \draw[->](x3)to(z1);
        \draw[->](x4)to(z2);
    \end{tikzpicture}
    \caption{The non-ex-string, non-ex-band $H(\widetilde{CD}_{2})$-module $N$}\label{NStNBdmodule}
   \end{center}
\end{figure}
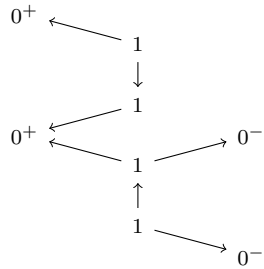
\end{rem}

\subsection{The category of $H(\widetilde{CD}_{n})$-modules}\label{indexSets}
In order to describe the Auslander--Reiten quiver of the category of the GLS algebra $H(\widetilde{CD}_{n})$, we introduce certain index sets together with equivalence relations as follows.

Recall the definition of $\mathcal{I}$ and $\mathcal{I}_e$ from \eqref{Index I} and \eqref{I-even}. For a subset $S$ of $\mathcal{I}$, an index $\delta$ is called $S$-power-free if it is not a non-trivial power of any element in $S$. Set
\begin{itemize}
    \item[-] 
$\mathcal{I}_1=\{\delta\in\mathcal{I}\mid\delta\neq\delta^{-1}\}$;

    \item[-] 
$\mathcal{I}_2=\{\delta\in\mathcal{I}\mid \delta \textup{\;is\;}\mathcal{I}\textup{-power-free}\}$;

\item[-] $\mathcal{I}_{3}=\{\delta\in\mathcal{I}_e\mid \delta \textup{\;is\;}\mathcal{I}_e\textup{-power-free}\}$;

\item[-] $\mathcal{I}_4=\{\delta\in\mathcal{I}\mid\delta\circ\delta^{-1}\textup{\;is\;}\mathcal{I}_e\textup{-power-free}\}$;

\item[-]
$\mathcal{I}_{5}=\{\delta\in\mathcal{I}_{2}\mid \delta[m]=\delta'\circ\delta'^{-1}\textup{\;for some\;} m\in\mathbb{Z} \textup{\;and\;}\delta'\in\mathcal{I}_4\}$.
\end{itemize}

Define $\sim_k \; (1\leq k\leq 3)$ to be the smallest equivalence relations satisfying the following conditions for all $\delta\in\mathcal{I}$:

\begin{itemize}
    \item[-] 
$\delta\sim_1\delta^{-1}$; 

\item[-]
$\delta\sim_2\delta^{-1}$ and $\delta\sim_2\delta[1]$; 
\item[-]
$\delta\sim_3(\delta^{-1})[1]$ and $\delta\sim_3\delta[2]$.
\end{itemize}

Let ${\bf k}^{\#}$ be the set of equivalence classes of ${\bf k}^{*}\backslash\{\pm1\}$ under the relation $x\sim x^{-1}$ for any $x$. Now we can state the following theorem, which is one of the main results of this paper. 

\begin{thm}\label{MainTheorem}
    The Auslander--Reiten quiver $\Gamma_{H(\widetilde{CD}_{n})}$ consists of the following components.
\begin{enumerate}
    \item One component $\mathcal{H}(\widetilde{CD}_{n})_{{PI}}$ containing all indecomposable preprojective modules and all indecomposable preinjective modules, which has the form $\mathbb{Z}D_{\infty}-\big(Q^{\textup{o}}(\widetilde{CD}_{n})\big)^{\textup{op}}$.
    \item Components $\mathcal{H}(\widetilde{CD}_{n})_{(w_1,w_2)}$ of type $\mathbb{Z}D_{\infty}$, where $(w_1,w_2)$ runs through all $1$-ex-string pairs.
    \item Components $\mathcal{H}(\widetilde{CD}_{n})_{w}$ of type $\mathbb{Z}A_{\infty}^{\infty}$, where $w$ runs through the full ex-strings with index in $\mathcal{I}_1/\sim_1$. 
    \item One stable tube $\mathcal{H}(\widetilde{CD}_{n})_{(1,1)}$ of rank $n-1$.
    %with $N\big(\varepsilon_1\cdot(\gamma_{\textup{max}})_{\leq n-1}\big),S_{n-1},\cdots,S_3,S_2$ at the bottom of the tube.
    \item Stable tubes $\mathcal{H}(\widetilde{CD}_{n})_{(w_1,w_2)}$ of rank $2$, where $(w_1,w_2)$ runs through the $2$-ex-string pairs with indices $\delta(w_1)=\delta(w_2)$ in $\mathcal{I}_{4}/{\sim_1}$.
    \item Homogeneous tubes $\mathcal{H}(\widetilde{CD}_{n})_{(w,\lambda)}$, where $w$ is an ex-band and $(w,\lambda)$ runs through the set $$\{(w,\lambda)\mid \delta(w)\in(\mathcal{I}_2-\mathcal{I}_{5})/{\sim_2},\lambda\in{\bf{k}}^{*}  \}\cup\{(w,\lambda)\mid \delta(w)\in\mathcal{I}_{5}/{\sim_{2}},\lambda\in{\bf k}^{\#}\}.$$
\end{enumerate}   
\end{thm}

The proof of this main theorem will be given in Section \ref{Mainresult1}. As an immediate consequence, we obtain the following result.

\begin{cor}\label{tame for H(CD)}
   The GLS algebra $H(\widetilde{CD}_{n})$ is representation-tame.  
\end{cor}

Moreover, we can prove tameness for arbitrary orientations of the quiver $Q(\widetilde{CD}_{n})$ in Figure \ref{quiverForCD}.

\begin{prop}\label{tame for arbitrary orientation}
    The GLS algebra $H(C,D,\Omega)$ is representation-tame if $C$ is of type $\widetilde{CD}_n$ with $n\geq 2$ and $D$ is minimal. 
\end{prop}

\begin{proof}
By \cite[Theorem~3.14]{HLS2023}, the GLS algebra of type $\widetilde{C}_{2n-2}$ with minimal symmetrizer is representation-tame for every orientation. Hence, the equivariant argument used in the proof of Theorem~\ref{MainTheorem} applies when the arrows $\gamma_i\;(1\le i\le n-2)$ are arbitrarily oriented, and $\gamma_{0^{+}}$ and $\gamma_{0^{-}}$ have the same orientation.
 
Moreover, suppose that $\gamma_{0^{+}}$ and $\gamma_{0^{-}}$ have opposite orientations. Since $0^{+}$ and $0^{-}$ are leaf vertices, each incident with a unique arrow, and no defining relation involves either of them, the usual BGP reflection functor can be applied at $0^{+}$ or $0^{-}$. For a sink (respectively, source) $i\in\{0^{+},0^{-}\}$, it induces mutually inverse bijections between the indecomposable modules, except for the simple module at $i$, of the two GLS algebras whose orientations differ by reversing the unique arrow incident with $i$; see \cite[Theorem~1.1]{BGP1973}. In particular, such a reflection preserves representation type. Hence, by reflecting at $0^{+}$ or $0^{-}$ if necessary, we may arrange that $\gamma_{0^{+}}$ and $\gamma_{0^{-}}$ have the same orientation. The preceding argument then applies, and the result follows.
\end{proof}

\section{Equivariantization and the $\widetilde{C}$-$\widetilde{CD}$ correspondence}
\label{sec:equivariantization}

In this section, we recall the fundamental constructions of equivariantization and %investigate the equivariant categories of a tube under various $\mathbb{Z}_2$-actions. Additionally, we
introduce a classification of bands, laying the groundwork for the proof of Theorem \ref{MainTheorem}.

\subsection{Equivariantization}
Let $G=\{e,\sigma\}$ be the cyclic group of order $2$. Let $\mathcal{C}$ be an additive category. A (strict) $G$-action on $\mathcal{C}$ is an automorphism $F_{\sigma}$: $\mathcal{C}\rightarrow\mathcal{C}$ satisfying $F_{\sigma}^2=\operatorname{id}_{\mathcal{C}}$. A {\emph{$G$-equivariant object}} is a pair $(X,\alpha)$, where $X$ is an object in $\mathcal{C}$ and $\alpha:X\rightarrow F_{\sigma}(X)$ is an isomorphism satisfying the condition $$F_{\sigma}(\alpha)\circ\alpha=\operatorname{id}_X.$$ A morphism $f:\;(X,\alpha)\rightarrow(Y,\alpha')$ of $G$-equivariant objects is a morphism $f:\;X\rightarrow Y$ in $\mathcal{C}$ such that $\alpha'\circ f=F_{\sigma}(f)\circ \alpha$. We denote by $\mathcal{C}^G$ the {\emph{equivariant category}} of $\mathcal{C}$ (with respect to the given $G$-action), that is, the category of $G$-equivariant objects and their morphisms. 

The {\emph{induction functor}} $\textup{Ind}$: $\mathcal{C}\rightarrow\mathcal{C}^G$ is defined as follows. For an object $X\in\mathcal{C}$, set $$\textup{Ind}(X)=(X\oplus F_{\sigma}(X),\alpha),$$ 
where $\alpha$: $X\oplus F_{\sigma}(X)\rightarrow F_{\sigma}(X\oplus F_{\sigma}(X))$ is the isomorphism represented by the matrix $$\alpha=\begin{pmatrix}
    0&\operatorname{id}_{F_{\sigma}X}\\
    \operatorname{id}_X&0
\end{pmatrix}.$$ 
For a morphism $f$:\;$X\rightarrow Y$ in $\mathcal{C}$, we set $$\textup{Ind}(f)=f\oplus F_{\sigma}(f):\;\textup{Ind}(X)\rightarrow \textup{Ind}(Y).$$ 
By \cite[Theorem 3.8]{ReitenRiedtmann1985Skew}, the induction functor preserves almost split sequences as well as minimal left and right almost split maps in any Hom-finite abelian ${\bf k}$-category.

\subsection{Skew group algebras}
Let $G$ act on a ${\bf k}$-algebra $A$ by algebra automorphisms. The {\emph{skew group algebra}} $A[G]$ is the ${\bf k}$-algebra with underlying ${\bf k}$-vector space $A\otimes_{\bf k}{\bf k}G$ and multiplication defined on basis elements by $$(a\otimes g)(a'\otimes g')=ag(a')\otimes gg'$$ for all $a,a'\in A$ and $g,g'\in G$, extended linearly (see \cite{ReitenRiedtmann1985Skew}). For simplicity, we write $ag$ in place of $a\otimes g$.

The skew group algebra provides {a connection} between a module category and an equivariant category. Let $G=\{e,\sigma\}$ act on $A$ by automorphisms. The $\sigma$-action on $A$ induces an autoequivalence of $\operatorname{mod}A$, denoted by $F_{\sigma}:\;M\mapsto {^{\sigma}M}$, where $^{\sigma}M$ equals $M$ as an abelian group but with $A$-action twisted by $\sigma$, that is, $$m\cdot a=m\sigma(a)$$ 
for $a\in A$ and $m\in M$. The following lemma establishes the fundamental connection between the equivariant module category $(\operatorname{mod}A)^G$ and the module category $\operatorname{mod}A[G]$, which is essential for our subsequent calculations.

\begin{lem}[{\cite[Example 2.6]{ChenXW2017}}]\label{EqFromA^GToA[G]}
There is an isomorphism of categories
$$\Phi:\;(\operatorname{mod}A)^G\rightarrow\operatorname{mod}A[G],\;(M,\alpha)\mapsto M,$$
where the $A[G]$-module structure on $M$ is defined by $m\cdot (a\sigma)=\alpha(m)\cdot a$. 
\end{lem}

Note that the skew group algebra $A[G]$ need not be basic even when $A$ is basic. We denote by $\textup{bas}(A[G])$ the basic algebra Morita equivalent to $A[G]$ (see \cite[Corollary~I.6.10]{Elements1}). Let $$\Phi:\;(\operatorname{mod}A)^G\rightarrow\operatorname{mod}\big(\textup{bas}(A[G])\big)$$ be the composition of the isomorphism from Lemma \ref{EqFromA^GToA[G]} with the Morita equivalence.

\begin{example}\label{EXAMPLEcd+c}
Let $G$ act on $Q(\widetilde{C}_{2n-2})$ by $\sigma(i)=-i$, $\sigma(\beta_i)=\beta_{-i}$ for each vertex $i$, and by $\sigma(\varepsilon_1)=\varepsilon_{-1}, \sigma(\varepsilon_{-1})=\varepsilon_1$. This induces a $\sigma$-action on $H(\widetilde{C}_{2n-2})$. By the explicit computation in \cite[Section 2]{ReitenRiedtmann1985Skew}, the basic algebra $\textup{bas}\big(H(\widetilde{C}_{2n-2})[G]\big)$ admits a complete set of primitive orthogonal idempotents
$$\{e_1, \cdots, e_{n-1}, e_{0^{+}}, e_{0^{-}}\},$$
where 
$$e_{0^{+}}={\frac{e_0+e_0{\sigma}}{2}}, e_{0^{-}}={\frac{e_0-e_0{\sigma}}{2}}.$$
This set corresponds precisely to a complete set of primitive orthogonal idempotents in $H(\widetilde{CD}_{n})$. Hence, $$\textup{bas}\big(H(\widetilde{C}_{2n-2})[G]\big)\simeq H(\widetilde{CD}_{n}),$$
and under this identification we have
$${{\gamma_i=\beta_i}},\;\gamma_{0^{+}}=\frac{\beta_{n-1}+\beta_{1-n}\sigma}{2},\;\gamma_{0^{-}}=\frac{\beta_{n-1}-\beta_{1-n}\sigma}{2}$$ for each $1\leq i\leq n-2$. By Lemma \ref{EqFromA^GToA[G]}, there is an isomorphism of categories
$$\Phi:\;\big(\operatorname{mod}H(\widetilde{C}_{2n-2})\big)^G\xrightarrow{\simeq}\operatorname{mod}H(\widetilde{CD}_{n}).$$
\end{example}

\subsection{Symmetric strings and bands}
For the remainder of the paper, we adopt the following conventions: $G$ acts on $Q(\widetilde{C}_{2n-2})$ as specified in Example \ref{EXAMPLEcd+c}; all strings and bands are understood to lie in $H(\widetilde{C}_{2n-2})$, while all ex-strings and ex-bands are understood to lie in $H(\widetilde{CD}_{n})$.

Recall the index sets $\mathcal{I}_j$ for $1\leq j\leq 5$ and the equivalence relations $\sim_k$ for $1\leq k\leq 3$ in Subsection
\ref{indexSets}.

Extend $\sigma$ to formal inverses by imposing $\sigma(c^{-1})=\sigma(c)^{-1}$ for every arrow $c$ in $H(\widetilde{C}_{2n-2})$, and then extend it to any word by multiplicativity, namely, $$\sigma(c_1c_2\cdots c_m)=\sigma(c_1)\sigma(c_2)\cdots\sigma(c_m).$$
These assignments induce actions on the sets ${\textup{St}}(H(\widetilde{C}_{2n-2}))$ and ${\textup{Ba}}(H(\widetilde{C}_{2n-2}))$, which we denote again by $\sigma$. One readily verifies that $\sigma$ respects the relevant equivalence relations.

\begin{defn}\label{symstring}
\begin{enumerate}
    \item 
    A string $w$ is called {\emph{symmetric}} if $\sigma(w)$ is equivalent to $w$ under the string equivalence. A band $w$ is called {\emph{symmetric}} if $\sigma(w)$ is equivalent to $w$ under the band equivalence.
\item Given a symmetric band $w$, we say it is:
\begin{itemize}
\item[-] {\emph{reversed}} when $\sigma(w)=w_{(i)}$ for some $i\in\mathbb{Z}$;
    \item[-] {\emph{rotated}} when $\sigma(w)=(w^{-1})_{(i)}$ for some $i\in\mathbb{Z}$.
\end{itemize}
\end{enumerate}
\end{defn}

Since $\sigma$ fixes no arrow, we have $\sigma(w)\neq w$ for every non-trivial string $w$. Thus, a non-trivial string $w$ is symmetric if and only if $\sigma(w)=w^{-1}$. Moreover, under the equivalence relation of $\textup{Ba}\big(H(\widetilde{C}_{2n-2})\big)$, a symmetric band is necessarily either rotated or reversed. The next lemma provides explicit expressions for these two cases.

\begin{lem}\label{ClassifyOfBands}
Let $w=\xi(\delta)\cdot\beta_{\textup{max}}^{-1}$ be a symmetric band with $\delta=\delta_1\delta_2\cdots\delta_{2k}\in\mathcal{I}_{3}$.
\begin{enumerate}
    \item If $w$ is reversed, then $k$ is odd and $\delta=\delta'\circ\delta'$ for some index $\delta'\in\mathcal{I}$;
    \item If $w$ is rotated, then $\delta\in\mathcal{I}_{5}$.
\end{enumerate}
\end{lem}

\begin{proof}
Recall that
$$\sigma(w)=\varepsilon_{-1}^{\delta_{1}}\beta_{\textup{max}}^{-1}\varepsilon_{1}^{\delta_{2}}\beta_{\textup{max}}\cdots\beta_{\textup{max}}^{-1}\varepsilon_1^{\delta_{2k}}\beta_{\textup{max}},$$ and
$$w^{-1}=\beta_{\textup{max}}\varepsilon_{-1}^{-\delta_{2k}}\beta_{\textup{max}}^{-1}\cdots\beta_{\textup{max}}\varepsilon_{-1}^{-\delta_{2}}\beta_{\textup{max}}^{-1}\varepsilon_1^{-\delta_1}.$$

(1) Suppose $w$ is reversed, i.e., {$\sigma(w)$ is a cyclic shift of $w$.} Then there exists an odd $j$ such that
$$\delta_{i}=\delta_{i+j}\;\;\textup{for all}\;i\in\mathbb{Z}_{2k}.$$
Let $d=\textup{gcd}(j,k)$. There exist $u,v\in\mathbb{Z}$ with $uj+v(2k)=d$. This implies
$$\delta_i=\delta_{i+uj}=\delta_{i+d-2kv}=\delta_{i+d}\;\;\textup{for each}\;i.$$
Since $d\mid k$, we obtain $\delta_i=\delta_{i+d}=\delta_{i+k}$ for all $i$. Thus $\delta=\delta'\circ\delta'$ with $\delta'=\delta_1\delta_2\cdots\delta_k$.

Finally, $k$ must be odd. Indeed, if $k$ were even, then $\delta'\in \mathcal{I}_e$, so $\delta=(\delta')^2$ would contradict the assumption that $\delta\in \mathcal{I}_3$ is $\mathcal{I}_e$-power-free.

(2) Suppose $w$ is rotated, so $\sigma(w)$ is a cyclic shift of $w^{-1}$. Then there exists $l\in\mathbb{Z}_{2k}$ such that
$$\delta_i=-\delta_{2l+1-i}\;\;\textup{for all}\;i\in\mathbb{Z}_{2k}.$$
Consequently,
$$\delta_{l+1}\delta_{l+2}\cdots\delta_{l+k}=(-\delta_l)(-\delta_{l-1})\cdots(-\delta_{l-k+2})(-\delta_{l-k+1})=(\delta_{l+k+1}\delta_{l+k+2}\cdots\delta_{l-1}\delta_l)^{-1},$$
where $\delta_{l-k+i}=\delta_{2k+(l-k+i)}=\delta_{l+k+i}$. Thus,
$$\delta[l]=(\delta_{l+1}\delta_{l+2}\cdots\delta_{l+k})\circ(\delta_{l+1}\delta_{l+2}\cdots\delta_{l+k})^{-1}.$$
Since $\delta\in \mathcal{I}_3$, every cyclic shift of $\delta$ is $\mathcal{I}_e$-power-free. Hence $\delta_{l+1}\delta_{l+2}\cdots\delta_{l+k}\in \mathcal{I}_4$. 

It remains to show that $\delta$ is $\mathcal{I}$-power-free. Suppose that $\delta=\rho^m$ with $m\geq2$. If the length of $\rho$ is even, this contradicts the fact that $\delta$ is $I_e$-power-free. Hence the length $q$ of $\rho$ is odd; since $\delta$ has even length, $m$ is even. The relation $\delta_i = -\delta_{2l+1-i}$ descends modulo the odd period $q$. But the involution $i\mapsto 2l+1-i$ on $\mathbb{Z}/q\mathbb{Z}$ has a fixed point, which would give $\delta_i=-\delta_i$, a contradiction. Thus $\delta$ lies in $\mathcal{I}_2$. Hence, $\delta\in \mathcal{I}_5$.
\end{proof}

\begin{example}
Consider $H(\widetilde{C}_{2})$ as in Example \ref{Example2.2}. The bands
$$w_1=\varepsilon_{1}\beta_1\beta_{-1}^{-1}\varepsilon_{-1}\beta_{-1}\beta_1^{-1}\;\;\textup{and}\;\;w_2=\varepsilon_{1}\beta_1\beta_{-1}^{-1}\varepsilon_{-1}^{-1}\beta_{-1}\beta_1^{-1}$$ are reversed and rotated, respectively. Furthermore, $$w_3=\varepsilon_1\beta_1\beta_{-1}^{-1}\varepsilon_{-1}\beta_{-1}\beta_1^{-1}\varepsilon_1\beta_1\beta_{-1}^{-1}\varepsilon_{-1}^{-1}\beta_{-1}\beta_1^{-1}$$ is a band that is not symmetric. These three bands correspond to the diagrams in Figure \ref{SymBands+NonSymBand}, where composition proceeds anticlockwise.
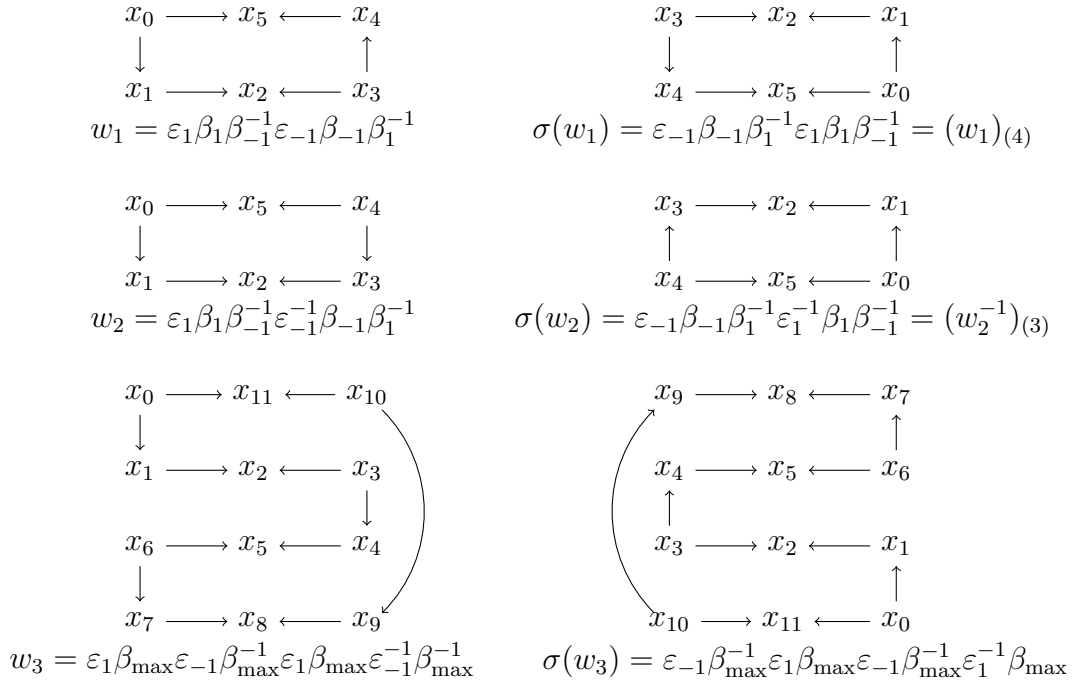
\begin{figure}[h]
    \centering
\begin{tikzpicture}
    \node (1) at (0,5) {$x_0$};
    \node (2) at (0,4) {$x_1$};
    \node (3) at (1.5,4) {$x_2$};
    \node (4) at (3,4) {$x_3$};
    \node (5) at (3,5) {$x_4$};
    \node (6) at (1.5,5) {$x_5$};
    \draw [->] (1) -- (2);
    \draw [->] (2) -- (3);
    \draw [->] (4) -- (3);
    \draw [->] (4) -- (5);
    \draw [->] (5) -- (6);
    \draw [->] (1) -- (6);
    \node () at (1.5,3.5) {$w_1=\varepsilon_{1}\beta_1\beta_{-1}^{-1}\varepsilon_{-1}\beta_{-1}\beta_1^{-1}$};
     \node (11) at (7,5) {$x_3$};
    \node (12) at (7,4) {$x_4$};
    \node (13) at (8.5,4) {$x_5$};
    \node (14) at (10,4) {$x_0$};
    \node (15) at (10,5) {$x_1$};
    \node (16) at (8.5,5) {$x_2$};
    \draw [<-] (12) -- (11);
    \draw [->] (12) -- (13);
    \draw [->] (14) -- (13);
    \draw [<-] (15) -- (14);
    \draw [->] (15) -- (16);
    \draw [->] (11) -- (16);
    \node () at (8.5,3.5) {$\sigma(w_1)=\varepsilon_{-1}\beta_{-1}\beta_{1}^{-1}\varepsilon_{1}\beta_{1}\beta_{-1}^{-1}=(w_1)_{(4)}$};
    \node (21) at (0,2.5) {$x_0$};
    \node (22) at (0,1.5) {$x_1$};
    \node (23) at (1.5,1.5) {$x_2$};
    \node (24) at (3,1.5) {$x_3$};
    \node (25) at (3,2.5) {$x_4$};
    \node (26) at (1.5,2.5) {$x_5$};
    \draw [->] (21) -- (22);
    \draw [->] (22) -- (23);
    \draw [->] (24) -- (23);
    \draw [->] (25) -- (24);
    \draw [->] (25) -- (26);
    \draw [->] (21) -- (26);
    \node () at (1.5,1) {$w_2=\varepsilon_{1}\beta_1\beta_{-1}^{-1}\varepsilon_{-1}^{-1}\beta_{-1}\beta_1^{-1}$};
    \node (31) at (7,2.5) {$x_3$};
    \node (32) at (7,1.5) {$x_4$};
    \node (33) at (8.5,1.5) {$x_5$};
    \node (34) at (10,1.5) {$x_0$};
    \node (35) at (10,2.5) {$x_1$};
    \node (36) at (8.5,2.5) {$x_2$};
    \draw [<-] (31) -- (32);
    \draw [->] (32) -- (33);
    \draw [->] (34) -- (33);
    \draw [<-] (35) -- (34);
    \draw [->] (35) -- (36);
    \draw [->] (31) -- (36);
    \node () at (8.5,1) {$\sigma(w_2)=\varepsilon_{-1}\beta_{-1}\beta_{1}^{-1}\varepsilon_{1}^{-1}\beta_{1}\beta_{-1}^{-1}=(w_2^{-1})_{(3)}$};
    \node (41) at (0,0) {$x_0$};
    \node (42) at (0,-1) {$x_{1}$};
    \node (43) at (1.5,-1) {$x_{2}$};
    \node (44) at (3,-1) {$x_3$};
    \node (45) at (3,0) {$x_{10}$};
    \node (46) at (1.5,0) {$x_{11}$};
    \node (51) at (0,-2) {$x_6$};
    \node (52) at (0,-3) {$x_7$};
    \node (53) at (1.5,-3) {$x_8$};
    \node (54) at (3,-3) {$x_9$};
    \node (55) at (3,-2) {$x_4$};
    \node (56) at (1.5,-2) {$x_5$};
    \draw [->] (41) -- (42);
    \draw [->] (42) -- (43);
    \draw [->] (44) -- (43);
    \draw [->] (44) -- (55);
    \draw [->] (55) -- (56);
    \draw [->] (51) -- (56);
    \draw [->] (51) -- (52);
    \draw [->] (52) -- (53);
    \draw [->] (54) -- (53);
    \draw [->] (3.2,-0.2) arc(45:-45:1.9);
    \draw [->] (45) -- (46);
    \draw [->] (41) -- (46);
    \node () at (1.35,-3.5) {$w_3=\varepsilon_1\beta_{\textup{max}}\varepsilon_{-1}\beta_{\textup{max}}^{-1}\varepsilon_1\beta_{\textup{max}}\varepsilon_{-1}^{-1}\beta_{\textup{max}}^{-1}$};
    \node (61) at (7,0) {$x_{9}$};
    \node (62) at (7,-1) {$x_{4}$};
    \node (63) at (8.5,-1) {$x_{5}$};
    \node (64) at (10,-1) {$x_6$};
    \node (65) at (10,0) {$x_7$};
    \node (66) at (8.5,0) {$x_{8}$};
    \node (71) at (7,-2) {$x_3$};
    \node (72) at (7,-3) {$x_{10}$};
    \node (73) at (8.5,-3) {$x_{11}$};
    \node (74) at (10,-3) {$x_0$};
    \node (75) at (10,-2) {$x_1$};
    \node (76) at (8.5,-2) {$x_2$};
    \draw [<-] (65) -- (64);
    \draw [->] (62) -- (63);
    \draw [->] (64) -- (63);
    \draw [<-] (62) -- (71);
    \draw [->] (75) -- (76);
    \draw [->] (71) -- (76);
    \draw [<-] (75) -- (74);
    \draw [->] (72) -- (73);
    \draw [->] (74) -- (73);
    \draw [<-] (6.8,-0.2) arc(135:225:1.9);
    \draw [->] (65) -- (66);
    \draw [->] (61) -- (66);
    \node () at (8.8,-3.5) {$\sigma(w_3)=\varepsilon_{-1}\beta_{\textup{max}}^{-1}\varepsilon_{1}\beta_{\textup{max}}\varepsilon_{-1}\beta_{\textup{max}}^{-1}\varepsilon_1^{-1}\beta_{\textup{max}}$};
\end{tikzpicture}
\caption{Two symmetric bands and a non-symmetric band}
\label{SymBands+NonSymBand}
\end{figure}
\end{example}

The next proposition determines the conditions under which a band module remains invariant under the $\sigma$-action.

\begin{prop}\label{StringEquai}
    Let $w$ be a band of $H(\widetilde{C}_{2n-2})$ and $\varphi$ an automorphism of a finite-dimensional ${\bf{k}}$-vector space. Then $$F_{\sigma}\big(M(w,\varphi)\big)\cong M\big(\sigma(w),\varphi\big).$$ 
Moreover, $F_{\sigma}\big(M(w,\varphi)\big)\cong M(w,\varphi)$ if and only if one of the following holds:
    \begin{enumerate}
        \item $w$ is a reversed band;
        \item $w$ is a rotated band, and $\varphi$ is similar to $\varphi^{-1}$.
    \end{enumerate}
\end{prop}

\begin{proof}
Assume that $w=\xi(\delta)\cdot\beta_{\textup{max}}^{-1}$. Since $\sigma$ permutes $\beta_{i}$ and $\beta_{-i}$ for each $1\leq i\leq n-2$, and swaps $\varepsilon_1$ with $\varepsilon_{-1}$, we obtain
$$F_{\sigma}(M(w,\varphi))=M(\varepsilon_{-1}^{\delta_{1}}\beta_{\textup{max}}^{-1}\varepsilon_{1}^{\delta_{2}}\beta_{\textup{max}}\cdots\beta_{\textup{max}}^{-1}\varepsilon_1^{\delta_{2k}}\beta_{\textup{max}},\varphi)=M(\sigma(w),\varphi).$$
Thus $F_{\sigma}(M(w,\varphi))\cong M(w,\varphi)$ only when $\sigma(w)\sim w$, i.e., only when $w$ is symmetric.

(1) If $w$ is reversed, by Lemma \ref{ClassifyOfBands}, $k$ is odd and $\delta_1\delta_2\cdots\delta_{k}=\delta_{k+1}\delta_{k+2}\cdots\delta_{2k}$. Consequently, $$M\big(\sigma(w),\varphi\big)=M(w_{(2kn-k+1)},\varphi)\cong M(w,\varphi).$$  

(2) If $w$ is rotated, by Lemma \ref{ClassifyOfBands}, we may assume that $\delta=\delta'\circ\delta'^{-1}$. Then
$$M\big(\sigma(w),\varphi\big)=M(w^{-1}_{(2k-1)},\varphi)\cong M(w,\varphi^{-1}),$$
where the last isomorphism follows from $\delta_1=-\delta_{2k}$. In this case, $F_{\sigma}\big(M(w,\varphi)\big)\cong M(w,\varphi)$ holds if and only if $\varphi$ is similar to $\varphi^{-1}$.
\end{proof}

\section{The Auslander--Reiten quiver of $H(\widetilde{CD}_{n})$}\label{Mainresult1}
In this section, we prove one of the main results of this paper---Theorem \ref{MainTheorem}---via equivariant methods. Throughout this section, we use the $G$-action introduced in Example \ref{EXAMPLEcd+c} and the equivalence
$$\Phi:\;\big(\operatorname{mod}H(\widetilde{C}_{2n-2})\big)^G\xrightarrow{\simeq}\operatorname{mod}H(\widetilde{CD}_{n}).$$

\subsection{Subcategories arising from string modules}
In this subsection, we study the equivariant category of the subcategories arising from string modules. 

\begin{defn}[{\cite[Definition 2.8]{HLS2023}}]
\begin{enumerate}
    \item A string $w$ is {\emph{right directly extendable}} ({\emph{RDE}}) if there is an arrow $\alpha$ such that $w\alpha$ is a string. In this case, we write $w_c:\;=w\cdot\alpha \cdot\alpha_{-}$, where $\alpha_{-}$ is the inverse string of maximal length such that $\alpha\cdot\alpha_{-}$ is a string.
    \item A string $w$ is {\emph{right inversely extendable}} ({\emph{RIE}}) if there is an arrow $\alpha$ such that $w\alpha^{-1}$ is a string. In this case, we write $w_h:\;=w\cdot\alpha^{-1} \cdot (\alpha^{-1})_{+}$, where $(\alpha^{-1})_{+}$ is the direct string of maximal length such that $\alpha^{-1}\cdot(\alpha^{-1})_{+}$ is a string.
    \item A string $w$ is {\emph{left directly extendable}} ({\emph{LDE}}) if there is an arrow $\alpha$ such that $\alpha w$ is a string. In this case, we write $_hw:\;=_{-}\alpha\cdot\alpha\cdot w$, where $_{-}\alpha$ is the inverse string of maximal length such that $_{-}\alpha\cdot\alpha$ is a string.
    \item A string $w$ is {\emph{left inversely extendable}} ({\emph{LIE}}) if there is an arrow $\alpha$ such that $\alpha^{-1}w$ is a string. In this case, we write $_cw:\;=_{+}(\alpha^{-1})\cdot\alpha^{-1}\cdot w$, where $_{+}(\alpha^{-1})$ is the direct string of maximal length such that $_{+}(\alpha^{-1})\cdot\alpha^{-1}$ is a string.
\end{enumerate}
\end{defn}

\begin{lem}\label{extendString}
    Let $w$ be a string of $H(\widetilde{C}_{2n-2})$. If the string $_{i}w_j$ exists with $i,j\in\{c,h\}$, then
    $$F_{\sigma}\big(M(_{i}w_j)\big)\cong M\big(_{i}\sigma(w)_j\big).$$
\end{lem}
\begin{proof}
Note that the quiver $Q(\widetilde{C}_{2n-2})$ is symmetric with respect to the vertex $0$. Consequently, a string $w$ of $H(\widetilde{C}_{2n-2})$ is RDE (resp. RIE, LDE, LIE) if and only if so is $\sigma(w)$. Moreover,
$$_i\sigma(w)_j=\sigma(_iw_j).$$
This implies that
$$F_{\sigma}(M(_iw_j))\cong M(\sigma(_iw_j))=M(_i\sigma(w)_j).$$
\end{proof}

Recall from Proposition \ref{ThmCn} the structure of the Auslander–Reiten quiver $\Gamma_{H(\widetilde{C}_{2n-2})}$ of the GLS algebra $H(\widetilde{C}_{2n-2})$. We now determine the equivariant categories of the full subcategories consisting of string modules.
 We denote by $(\mathcal{T}_p,\tau)$, or simply $\mathcal{T}_p$, the standard tube of rank $p$, that is, the Auslander--Reiten quiver of $\mathcal{T}_p$ is of the form $\mathbb{Z}A_{\infty}/(\tau^p)$ (see \cite{Elements2}). 
 % For a non-projective indecomposable object $X$ in a standard tube $\mathcal{T}_p$, there is a canonical bijection $\zeta$ from the set $X^{-}$ of arrows ending at $X$ to the set $(\tau X)^{+}$ of arrows starting at $\tau X$ (see \cite{Covering15}). 
 
\begin{prop}\label{forthm002} The following statements hold.
   \begin{enumerate}
       \item $\big(\mathcal{H}(\widetilde{C}_{2n-2})_{PI}\big)^G\simeq \mathcal{H}(\widetilde{CD}_{n})_{PI}$, which has the form $\mathbb{Z}D_{\infty}-\big(Q^{\textup{o}}(\widetilde{CD}_{n})\big)^{\textup{op}}$.
       \item $\big(\mathcal{H}(\widetilde{C}_{2n-2})_{(1,1)}\big)^G\simeq \mathcal{T}_{n-1}$.
       \item Let $w$ be a minimal string of type $(2,2)$.
        \begin{itemize}
           \item[(a)] If $w$ is symmetric, then $$\big(\mathcal{H}(\widetilde{C}_{2n-2})_w\big)^G\simeq \mathbb{Z}D_{\infty}.$$
           \item[(b)] If $w$ is non-symmetric, then $$\big(\mathcal{H}(\widetilde{C}_{2n-2})_w\times\mathcal{H}(\widetilde{C}_{2n-2})_{\sigma(w)}\big)^G\simeq \mathbb{Z}A_{\infty}^{\infty}.$$
        \end{itemize}
   \end{enumerate} 
\end{prop}

\begin{proof}
(1) Let $P_i$ denote the projective $H(\widetilde{C}_{2n-2})$-module (resp. $H(\widetilde{CD}_{n})$-module) associated to the vertex $i$. By direct calculation, we obtain 
$$\Phi\circ\textup{Ind}(P_0)=P_{0^{+}}\oplus P_{0^{-}}\;\;\textup{and}\;\;\;\Phi\circ\textup{Ind}(P_{\pm i})=P_i\;\;\textup{for each}\;1\leq i\leq n-1.$$
Moreover, the action of $F_{\sigma}$ reverses the component $\mathcal{H}(\widetilde{C}_{2n-2})_{PI}$ upside-down along the $\tau$-orbit of $P_0$. Hence the result follows directly from Proposition \ref{ThmCn} (1).

(2) Let $S_i$ denote the simple $H(\widetilde{C}_{2n-2})$-module (resp. $H(\widetilde{CD}_{n})$-module) associated to the vertex $i$. By \cite[Proposition 3.7]{HLS2023}, the minimal string modules of type $(1,1)$ are $M\big(_{-}(\beta_i)\big)$ for each $i$. More precisely, these are
$$M(\beta_{n-1}^{-1}\cdots\beta_2^{-1}\beta_1^{-1}\varepsilon_1^{-1}),S_2,S_3,\cdots,S_{n-1},S_{1-n},\cdots,S_{-3},S_{-2},M(\beta_{1-n}^{-1}\cdots\beta_{-2}^{-1}\beta_{-1}^{-1}\varepsilon_{-1}^{-1}).$$
By \cite[Proposition 3.8]{HLS2023}, we have $$\tau^{-1}S_{\pm2}=M(\beta_{\pm(n-1)}^{-1}\cdots\beta_{\pm2}^{-1}\beta_{\pm1}^{-1}\varepsilon_{\pm1}^{-1}),\;\tau S_{\pm(n-1)}=M(\beta_{\mp(n-1)}^{-1}\cdots\beta_{\mp2}^{-1}\beta_{\mp1}^{-1}\varepsilon_{\mp1}^{-1}),$$ and 
$$\tau S_i=S_{i+1},\;\;\tau S_j=S_{j-1}\textup{\;\;for\;}1< i,-j< n-1.$$ 
By Lemma \ref{extendString}, we obtain $$F_{\sigma}|_{\mathcal{H}(\widetilde{C}_{2n-2})_{(1,1)}}=\tau^{n-1}.$$ 
Consequently, \cite[Proposition 2.1]{DRZ} yields $\big(\mathcal{H}(\widetilde{C}_{2n-2})_{(1,1)}\big)^G\simeq \mathcal{T}_{n-1}$.

(3) By \cite[Proposition 3.7]{HLS2023}, both the source and the target of a minimal string $w$ of type $(2,2)$ belong to $\{1,-1\}$. Hence $\sigma(w)$ is also minimal of type $(2,2)$, and by Proposition \ref{ThmCn}\;(4), both $\mathcal{H}(\widetilde{C}_{2n-2})_w$ and $\mathcal{H}(\widetilde{C}_{2n-2})_{\sigma(w)}$ are of type $\mathbb{Z}A_{\infty}^{\infty}$. 

Using the almost split sequences constructed in \cite{BR1987} and Lemma \ref{extendString}, the successors $M\big(_h\sigma(w)\big)$ and $M\big(\sigma(w)_h\big)$ of $F_{\sigma}\big(M(w)\big)$ coincide with $F_{\sigma}\big(M(_hw)\big)$ and $F_{\sigma}\big(M(w_h)\big)$. A similar result holds for predecessors.

If $w$ is symmetric, then $\sigma(w)=w^{-1}$. So $$F_{\sigma}\big(M(w)_h\big)=F_{\sigma}\big(M(\sigma(w)^{-1})_h\big)\cong F_{\sigma}\big(M(_h\sigma(w))\big)\cong M(_hw).$$ By induction, $F_{\sigma}$ reverses the component $\mathcal{H}(\widetilde{C}_{2n-2})_w$ upside-down along the $\tau$-orbit of $M(w)$. An equivariant calculation then shows that $\big(\mathcal{H}(\widetilde{C}_{2n-2})_w\big)^G\simeq \mathbb{Z}D_{\infty}$. 

If $w$ is non-symmetric, then $F_{\sigma}$ permutes $\mathcal{H}(\widetilde{C}_{2n-2})_w$ and $\mathcal{H}(\widetilde{C}_{2n-2})_{\sigma(w)}$. An equivariant calculation then shows that $\big(\mathcal{H}(\widetilde{C}_{2n-2})_w\times\mathcal{H}(\widetilde{C}_{2n-2})_{\sigma(w)}\big)^G\simeq \mathbb{Z}A_{\infty}^{\infty}$. 
\end{proof}

\subsection{Subcategories arising from band modules}

This subsection studies the equivariant categories associated
with subcategories of band modules. Some of these equivariant
categories are tubes of rank $2$, which play an important role
in our reformulation of the GLS conjecture.

\begin{prop}\label{forthm003}
    Let $w$ be a band of $H(\widetilde{C}_{2n-2})$ and $\lambda\in{\bf{k}}^{*}$.
    \begin{enumerate}
        \item If $w$ is non-symmetric, then $$\big(\mathcal{H}(\widetilde{C}_{2n-2})_{w,\lambda}\times\mathcal{H}(\widetilde{C}_{2n-2})_{\sigma(w),\lambda}\big)^G\simeq \mathcal{T}_1.$$
        \item If $w$ is rotated, then
        \begin{itemize}
            \item[(a)] $\big(\mathcal{H}(\widetilde{C}_{2n-2})_{w,\lambda}\times\mathcal{H}(\widetilde{C}_{2n-2})_{w,\lambda^{-1}}\big)^G\simeq \mathcal{T}_1$, when $\lambda\neq\pm1$;
            \item[(b)] $\big(\mathcal{H}(\widetilde{C}_{2n-2})_{w,\lambda}\big)^G\simeq \mathcal{T}_2$, when $\lambda=\pm1$.
        \end{itemize}
        \item If $w$ is reversed, then $$\big(\mathcal{H}(\widetilde{C}_{2n-2})_{w,\lambda}\big)^G\simeq \mathcal{T}_1\times\mathcal{T}_1.$$
    \end{enumerate}
\end{prop}

\begin{proof}
(1) If $w$ is not symmetric, then by Proposition \ref{StringEquai}, the functor $F_{\sigma}$ exchanges the two homogeneous tubes $\mathcal{H}(\widetilde{C}_{2n-2})_{w,\lambda}$ and $\mathcal{H}(\widetilde{C}_{2n-2})_{\sigma(w),\lambda}$. The desired result now follows immediately from Proposition \ref{ImportantEquivar} (1).

(2) If $w$ is rotated and $\lambda\neq\pm1$, then $F_{\sigma}$ permutes the two distinguished tubes 
$$\mathcal{H}(\widetilde{C}_{2n-2})_{w,\lambda} \textup{\;\;\;and\;\;\;} \mathcal{H}(\widetilde{C}_{2n-2})_{\sigma(w),\lambda}=\mathcal{H}(\widetilde{C}_{2n-2})_{w,\lambda^{-1}}.$$
Hence, assertion (a) follows from Proposition \ref{ImportantEquivar} (1).

Now we assume $w$ is rotated and $\lambda=\pm1$. For $r\geq1$, write $M_{\lambda,r}:=M(w,J(\lambda,r))$, and set $M_{\lambda,0}=0$. Choose the standard bases in which $\varphi$ is represented by the Jordan block $J(\lambda,r)$. We have the almost split sequence
$$0\rightarrow M_{\lambda,r}\xrightarrow{[f_r,f'_{r-1}]^t}M_{\lambda,r+1}\oplus M_{\lambda,r-1}\xrightarrow{[f'_{r},f_{r-1}]}M_{\lambda,r}\rightarrow0,$$
where, on every vertex space $V_{x_i}$,
$$f_r|_{V_{x_i}}=\begin{pmatrix}
    I_r\\0_{1\times r}
\end{pmatrix},\;f'_r|_{V_{x_i}}=\begin{pmatrix}
    0_{r\times 1}&(-1)^rI_r
\end{pmatrix}.$$
For each $r$, let $\alpha_r:\;F_{\sigma}(M_{\lambda,r})\xrightarrow{\sim}M_{\lambda,r}$ be the isomorphism whose components are given by $$\alpha_r|_{V_{i}}=L(\lambda,r)$$
when $i\neq l(2n-1)+n-1$ (where $l$ is the integer introduced in the proof of Lemma \ref{ClassifyOfBands}), and by
$$\alpha_r|_{V_{x_{l(2n-1)+n-1}}}=L(\lambda,r)J(\lambda,r)^{-1},$$
where $L(\lambda,r)$ is the involutive matrix from Lemma \ref{MatrixLemma}. By Lemma \ref{MatrixLemma}, $(M_{\lambda,r},\alpha_r^{-1})$ is a $G$-equivariant object.

For a morphism $h:\;M_{\lambda,s}\to M_{\lambda,t}$, we have $F_{\sigma}(h)=\alpha_t^{-1}\circ h\circ\alpha_s$. Because 
$$J(\lambda,r)^{-1}\begin{pmatrix}
    0_{r\times 1}&(-1)^rI_r
\end{pmatrix}=\begin{pmatrix}
    0_{r\times 1}&(-1)^rI_r
\end{pmatrix}J(\lambda,r+1)^{-1},$$
it follows that $$F_{\sigma}(f'_r)|_{V_{x_i}}=\alpha_r^{-1}\circ f'_r\circ\alpha_{r+1}=L(\lambda,r)\begin{pmatrix}
    0_{r\times 1}&(-1)^rI_r
\end{pmatrix}L(\lambda,r+1)$$
holds at each vertex $x_i$.

Define $f''_{r}=\frac{1}{2}(f'_r-F_{\sigma}(f'_r))=\frac{1}{2}(f'_r-\alpha_r\cdot f'_r\cdot\alpha_{r+1})$. Then on each vertex space $V_{x_i}$,
$$\begin{aligned}
    (f''_r-f'_r)|_{V_{x_i}}&=
   -\frac{1}{2}(
    0_{r\times 1}\;(-1)^rI_r)J(0,r+1)J(\lambda,r+1)^{-1}\\
&=\frac{(-1)^r}{2}\begin{pmatrix}
    0&0&-\lambda&(-\lambda)^2&\cdots&(-\lambda)^{r-1}\\
    0&0&0&-\lambda&\cdots&(-\lambda)^{r-2}\\
    \vdots&\vdots&\ddots&\ddots&\ddots&\vdots\\
    0&0&\cdots&0&0&-\lambda\\
    0&0&\cdots&0&0&0
\end{pmatrix}.
\end{aligned}$$
Let $\eta_{r+1}\in\textup{End}(M_{\lambda,r+1})$ be the endomorphism represented on each vertex space by $\eta_{r+1}|_{V_{x_i}}=-\frac{1}{2}J(0,r+1)J(\lambda,r+1)^{-1}$. Since $M_{\lambda,r+1}$ is indecomposable and $\eta_{r+1}$ is nilpotent, we have $\eta_{r+1}\in\textup{rad}\;\textup{End}(M_{\lambda,r+1})$. Therefore,
$$f''_r-f'_r=f'_r\;\eta_{r+1}\in\textup{rad}^2(M_{\lambda,r+1},M_{\lambda,r}).$$
Thus, $f''_r$ is an irreducible morphism and represents the same arrow in the Auslander--Reiten quiver as $f'_r$.

The identities 
$$L(\lambda,r+1)\begin{pmatrix}
    I_r\\0_{1\times r}
\end{pmatrix}=\begin{pmatrix}
    I_r\\0_{1\times r}
\end{pmatrix}L(\lambda,r),$$
and
$$L(\lambda,r+1)J(\lambda,r+1)^{-1}\begin{pmatrix}
    I_r\\0_{1\times r}
\end{pmatrix}=\begin{pmatrix}
    I_r\\0_{1\times r}
\end{pmatrix}L(\lambda,r)J(\lambda,r)^{-1},$$ show that $F_{\sigma}(f_r)=f_r$. On the other hand, $$F_{\sigma}(f_r'')=\frac{1}{2}\big(F_{\sigma}(f_r')-F^2_{\sigma}(f_r')\big)=\frac{1}{2}\big(F_{\sigma}(f_r')-f_r'\big)=-f''_r.$$

We may therefore choose the two types of irreducible generators $f_r$ and $f_r''$ so
that
$$F_\sigma(f_r)=f_r,\;\;\;\textup{and}\;\;\;\;F_\sigma(f''_r)=-f''_r.$$
Since the canonical bijection $\zeta$ pairs an arrow of one type with
an arrow of the other type, we obtain $l_{\zeta f}l_f=-1$ for each irreducible arrow $f$. Proposition
\ref{ImportantEquivar}(2)(ii) now yields
$$\big(\mathcal{H}(\widetilde{C}_{2n-2})_{w,\lambda}\big)^G\simeq \mathcal{T}_2.$$

(3) By a similar calculation, in the reversed case, the irreducible generators may be chosen such that
$$F_\sigma(f_r)=f_r,\;\;\;\textup{and}\;\;\;\;F_\sigma(f'_r)=f'_r.$$
Thus, $l_{\zeta f}l_f=1$ for each irreducible arrow $f$. The statement then follows directly from Proposition \ref{ImportantEquivar} (2)(i).
\end{proof}

\subsection{Proof of Theorem \ref{MainTheorem}}
We now proceed to the proof of Theorem \ref{MainTheorem}. Let $\widehat{G}$ denote the character group of $G$. By \cite[Theorem 4.6]{CCR}, there is a dual $\widehat{G}$-action on $\big(\operatorname{mod}H(\widetilde{C}_{2n-2})\big)^G$ satisfying
$$\Big(\big(\operatorname{mod}H(\widetilde{C}_{2n-2})\big)^G\Big)^{\widehat{G}}\simeq \operatorname{mod}H(\widetilde{C}_{2n-2}).$$
By equivariantization duality, every indecomposable object in $$\operatorname{mod}H(\widetilde{CD}_{n})\simeq\big(\operatorname{mod}H(\widetilde{C}_{2n-2})\big)^G$$ occurs as a direct summand of $\Phi\circ\textup{Ind}(X)$ for some indecomposable $X\in\operatorname{mod}H(\widetilde{C}_{2n-2})$.

The components labelled (1)--(4) in Theorem \ref{MainTheorem} arise from the string components of $\operatorname{mod}H(\widetilde{C}_{2n-2})$ that consist of string modules. Their shapes follow from Proposition \ref{forthm002}, while the indexing sets are determined by the equivariant correspondences established in Propositions \ref{StToExst} and \ref{ClassifyStToExst}.

The components (5) and (6) arise from those components of  $\operatorname{mod}H(\widetilde{C}_{2n-2})$ that consist of band modules. Their shapes are described by Proposition \ref{forthm003}, and the relevant equivariant correspondences are provided in Propositions \ref{FROMBdTOExBd} and \ref{ClassifyBaToExba}. Because the computation of these equivariant correspondences is somewhat lengthy yet essential for describing the index sets, we defer the details to the appendix.

This completes the proof.

\section{Galois coverings of GLS algebras}\label{sec:galois}

By Drozd's trichotomy, every finite-dimensional algebra is either
representation-finite, tame, or wild. In this section, we develop the
covering-theoretic tools needed to establish the wildness of GLS algebras.
More precisely, we first construct locally bounded Galois coverings of GLS
algebras and then identify hypercritical convex subcategories in these coverings. These subcategories will serve as the wildness witnesses used in
the classification in Section \ref{sec:tame-class}. We use the standard terminology concerning tame and wild representation type; see \cite{Dro1980}.

% We use the standard definitions of tame and wild representation type. Let $A$ be a finite-dimensional basic connected algebra. It is {\emph{representation-finite}} (or {\emph{of finite type}}) if there are only finitely many isomorphism classes of finitely generated indecomposable $A$-modules. Otherwise, $A$ is either tame or wild. Roughly speaking, $A$ is tame if, in each fixed dimension, all but finitely many indecomposable $A$-modules occur in finitely many one-parameter families; otherwise, $A$ is wild. We refer to \cite{Dro1980} for the precise definitions.

\subsection{The covering construction}
Galois coverings provide an effective tool for studying representation type; we refer to \cite{Gar1981} for the relevant definitions. We now construct explicitly a locally bounded Galois covering of a GLS algebra $H=H(C,D,\Omega)$.

Write $$D=\operatorname{diag}(d_1,d_2,\cdots,d_n),
\qquad
d:=\textup{lcm}(d_1,d_2,\cdots,d_n),
\qquad
b_i=\frac{d}{d_i}$$
for $1\leq i\leq n$. Let $Q=Q(C,\Omega)$, and assume the ordinary quiver $Q^{\textup{o}}$ (obtained from $Q$ by deleting all loops) is a tree. In this case, we can construct a precise Galois covering via quiver with relations as follows.

Let $(Q^{\mathbb{Z},D})_0:=Q_0\times\mathbb{Z}$ and
$$\begin{array}{c}
    (Q^{\mathbb{Z},D})_1:=\{(\alpha,j):(s(\alpha),j)\rightarrow(t(\alpha),j)|\alpha\in Q^{\textup{o}}_1,j\in\mathbb{Z}\}\\
    \;\;\;\;\;\;\;\;\;\;\;\;\;\;\;\;\;\cup\{(\varepsilon_x,j):(x,j)\rightarrow(x,j+b_x)|x\in Q_0,j\in\mathbb{Z}\}.
\end{array}$$
The relation set $I^{\mathbb{Z},D}$ is generated by the zero relations
$$(\varepsilon_x,j)(\varepsilon_x,j+b_x)\cdots(\varepsilon_x,j+d-b_x)$$
for each $x\in Q_0,j\in\mathbb{Z}$, and the commutativity relations
$$\begin{array}{c}
    (\alpha,j)(\varepsilon_{t(\alpha)},j)(\varepsilon_{t(\alpha)},j+b_{t(\alpha)})\cdots(\varepsilon_{t(\alpha)},j+\frac{d}{\textup{gcd}(d_{s(\alpha)},d_{t(\alpha)})}-b_{t(\alpha)})\\-(\varepsilon_{s(\alpha)},j)(\varepsilon_{s(\alpha)},j+b_{s(\alpha)})\cdots
(\varepsilon_{s(\alpha)},j+\frac{d}{\textup{gcd}(d_{s(\alpha)},d_{t(\alpha)})}-b_{s(\alpha)})(\alpha,j+\frac{d}{\textup{gcd}(d_{s(\alpha)},d_{t(\alpha)})})
\end{array}$$
for each $\alpha\in Q_1^{\textup{o}},j\in\mathbb{Z}$. 

\begin{prop}\label{GaloisCovering}
    With the notation above, $H=H(C,D,\Omega)$ admits a Galois
    covering
    $$\pi:\textup{Gal}(H)\longrightarrow H$$
    of locally bounded $\mathbf{k}$-categories, where $\textup{Gal}(H):\;=\mathbf{k}Q^{\mathbb{Z},D}/I^{\mathbb Z,D}$. 
\end{prop}

\begin{proof}
Give $H$ the $\mathbb Z_{\geq0}$-grading
\[
\deg(\alpha)=0\qquad(\alpha\in Q_1^\circ),
\qquad
\deg(\varepsilon_i)=b_i.
\]
Since $DC$ is symmetric, all defining relations of $H$ are
homogeneous with respect to this grading.

By the construction of $Q^{\mathbb Z,D}$ and $I^{\mathbb Z,D}$,
every homogeneous path in $Q$ of degree $j'-j$ has a unique lift
from $(x,j)$ to $(y,j')$, and the relations in
$I^{\mathbb Z,D}$ are precisely the lifts of the homogeneous
relations of $H$. Hence
\begin{equation}\label{eq:graded-cover}
e_{x,j}\operatorname{Gal}(H)e_{y,j'}
\cong
e_xH_{j'-j}e_y
\end{equation}
for all $x,y\in Q_0$ and $j,j'\in\mathbb Z$.

The shift
$$h(x,j)=(x,j+1)\textup{\;\;and\;\;}h(\alpha,j)=(\alpha,j+1)\;\;\;\;(x\in Q_0,\alpha\in Q_1,j\in\mathbb{Z})$$
% \[h(x,j)=(x,j+1)\]
defines a free $\mathbb Z$-action on $\operatorname{Gal}(H)$, and
the projection
$$\pi(x,j)=x\textup{\;\;and\;\;}\pi(\alpha,j)=\alpha\;\;\;\;(x\in Q_0,\alpha\in Q_1,j\in\mathbb{Z})$$
is constant on $\mathbb Z$-orbits. Moreover, by
\eqref{eq:graded-cover},
\[
\bigoplus_{m\in\mathbb Z}
e_{x,j}\operatorname{Gal}(H)e_{y,m}
\cong e_xHe_y,
\qquad
\bigoplus_{m\in\mathbb Z}
e_{x,m}\operatorname{Gal}(H)e_{y,j'}
\cong e_xHe_y.
\]
Thus $\pi$ is a Galois covering with group $\mathbb Z$.

Finally, since $H$ is finite-dimensional, only finitely many
homogeneous components $H_m$ are nonzero. Hence
\eqref{eq:graded-cover} implies that all morphism spaces in
$\operatorname{Gal}(H)$ are finite-dimensional and that every
object has nonzero morphisms to or from only finitely many objects.
Furthermore,
\[
e_{x,j}\operatorname{Gal}(H)e_{x,j}
\cong e_xH_0e_x=ke_x,
\]
since $Q^\circ$ is acyclic. Therefore,
$\operatorname{Gal}(H)$ is locally bounded.
\end{proof}

We next explain how the covering constructed above can be used to establish
wildness. The main idea is to detect a bounded wild subcategory, and in
particular a convex hypercritical subcategory, inside the Galois covering.

\subsection{Wildness criteria}
We first recall the notion of a hypercritical algebra, which arises from the
classification of minimal wild concealed algebras.

Let $A={\bf{k}}\Delta$ be the path algebra of an acyclic quiver $\Delta$, and let $T$ be a preprojective tilting $A$-module. The endomorphism algebra $\operatorname{End}_A(T)$ is called a {\emph{concealed algebra}} of type $\overline{\Delta}$, where $\overline{\Delta}$ is the underlying graph of $\Delta$. 

For the extended Euclidean types $\widetilde{\widetilde{A}}_m,T_5,\widetilde{\widetilde{D}}_n,\widetilde{\widetilde{E}}_6,\widetilde{\widetilde{E}}_7,\widetilde{\widetilde{E}}_8$ (see Figure \ref{ExEuclidean}), concealed algebras (which are minimal wild) have been classified by quivers and relations \cite{Ler1987,Unger1990,Wittman1990}. The ones arising from minimal wild hereditary tree algebras of types $T_5,\widetilde{\widetilde{D}}_n,\widetilde{\widetilde{E}}_6,\widetilde{\widetilde{E}}_7,\widetilde{\widetilde{E}}_8$ are termed {\emph{hypercritical algebras}}, a key benchmark for distinguishing tame from wild type.
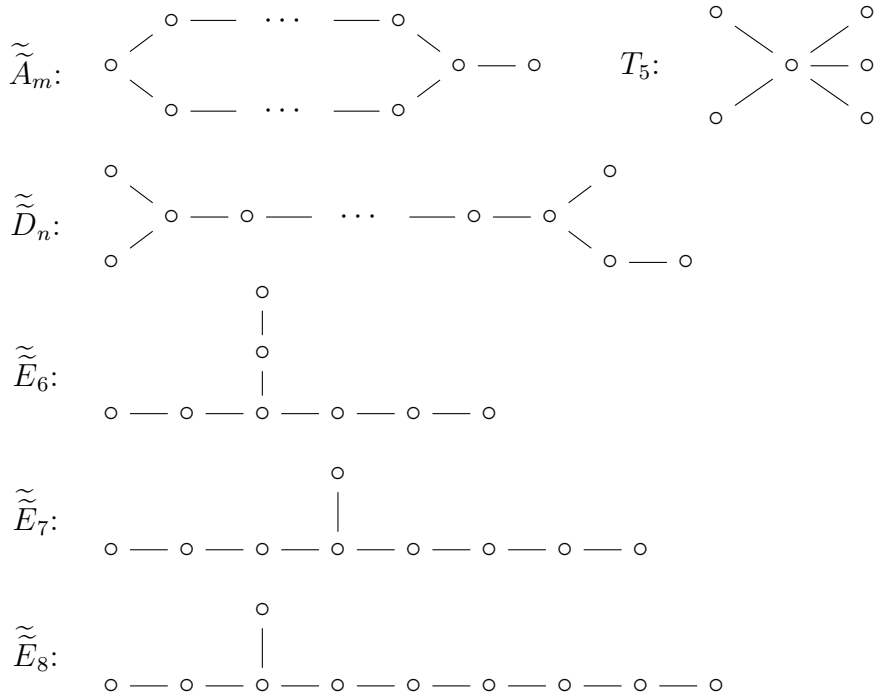
\begin{figure}[h]
    \centering
    \begin{tikzpicture}
    \node () at (-1,0) {$\widetilde{\widetilde{A}}_m$:};
        \node (00) at (0,0) {$\circ$};
        \node (01) at (0.8,0.6) {$\circ$};
        \node (02) at (0.8,-0.6) {$\circ$};
        \node (03) at (1.8,0.6) {};
        \node (04) at (1.8,-0.6) {};
        \node (05) at (2.8,0.6) {};
        \node (06) at (2.8,-0.6) {};
        \node (07) at (3.8,0.6) {$\circ$};
        \node (08) at (3.8,-0.6) {$\circ$};
        \node (09) at (4.6,0) {$\circ$};
        \node (0a) at (5.6,0) {$\circ$};
        \draw [-] (00) -- (01);
        \draw [-] (00) -- (02);
        \draw [-] (03) -- (01);
        \draw [-] (04) -- (02);
        \node () at (2.3,0.6) {$\cdots$};
        \node () at (2.3,-0.6) {$\cdots$};
        \draw [-] (07) -- (05);
        \draw [-] (08) -- (06);
        \draw [-] (07) -- (09);
        \draw [-] (08) -- (09);
        \draw [-] (0a) -- (09);
         \node () at (7,0) {$T_5$:};
         \node (10) at (8,0.7) {$\circ$};
         \node (11) at (8,-0.7) {$\circ$};
         \node (12) at (9,0) {$\circ$};
         \node (13) at (10,0.7) {$\circ$};
         \node (14) at (10,-0.7) {$\circ$};
         \node (15) at (10,0) {$\circ$};
         \draw [-] (10) -- (12);
         \draw [-] (11) -- (12);
         \draw [-] (13) -- (12);
         \draw [-] (14) -- (12);
         \draw [-] (15) -- (12);
         \node () at (-1,-2) {$\widetilde{\widetilde{D}}_n$:};
         \node (21) at (0,-1.4) {$\circ$};
        \node (22) at (0,-2.6) {$\circ$};
        \node (23) at (0.8,-2) {$\circ$};
        \node (24) at (1.8,-2) {$\circ$};
        \node (25) at (2.8,-2) {};
        \node () at (3.3,-2) {$\cdots$};
        \node (26) at (3.8,-2) {};
        \node (27) at (4.8,-2) {$\circ$};
        \node (28) at (5.8,-2) {$\circ$};
         \node (29) at (6.6,-1.4) {$\circ$};
        \node (2a) at (6.6,-2.6) {$\circ$};
        \node (2b) at (7.6,-2.6) {$\circ$};
        \draw [-] (21) -- (23);
        \draw [-] (22) -- (23);
        \draw [-] (24) -- (23);
        \draw [-] (24) -- (25);
        \draw [-] (26) -- (27);
        \draw [-] (27) -- (28);
        \draw [-] (29) -- (28);
        \draw [-] (2a) -- (28);
        \draw [-] (2a) -- (2b);
        \node () at (-1,-4) {$\widetilde{\widetilde{E}}_6$:};
        \node (30) at (0,-4.6) {$\circ$};
        \node (31) at (1,-4.6) {$\circ$};
        \node (32) at (2,-3) {$\circ$};
        \node (33) at (2,-3.8) {$\circ$};
        \node (34) at (2,-4.6) {$\circ$};
         \node (35) at (3,-4.6) {$\circ$};
         \node (36) at (4,-4.6) {$\circ$};
         \node (37) at (5,-4.6) {$\circ$};
         \draw [-] (31) -- (30);
        \draw [-] (31) -- (34);
        \draw [-] (32) -- (33);
        \draw [-] (33) -- (34);
        \draw [-] (35) -- (34);
        \draw [-] (35) -- (36);
        \draw [-] (36) -- (37);
        \node () at (-1,-5.9) 
        {$\widetilde{\widetilde{E}}_7$:};
        \node (40) at (0,-6.4) {$\circ$};
        \node (41) at (1,-6.4) {$\circ$};
        \node (42) at (2,-6.4) {$\circ$};
        \node (43) at (3,-6.4) {$\circ$};
        \node (44) at (3,-5.4) {$\circ$};
        \node (45) at (4,-6.4) {$\circ$};
        \node (46) at (5,-6.4) {$\circ$};
        \node (47) at (6,-6.4) {$\circ$};
        \node (48) at (7,-6.4) {$\circ$};
        \draw [-] (41) -- (40);
        \draw [-] (41) -- (42);
        \draw [-] (42) -- (43);
        \draw [-] (44) -- (43);
        \draw [-] (45) -- (43);
        \draw [-] (45) -- (46);
        \draw [-] (46) -- (47);
        \draw [-] (47) -- (48);
        \node () at (-1,-7.7) 
        {$\widetilde{\widetilde{E}}_8$:};
        \node (50) at (0,-8.2) {$\circ$};
        \node (51) at (1,-8.2) {$\circ$};
        \node (52) at (2,-8.2) {$\circ$};
        \node (53) at (2,-7.2) {$\circ$};
        \node (54) at (3,-8.2) {$\circ$};
        \node (55) at (4,-8.2) {$\circ$};
        \node (56) at (5,-8.2) {$\circ$};
        \node (57) at (6,-8.2) {$\circ$};
        \node (58) at (7,-8.2) {$\circ$};
        \node (59) at (8,-8.2) {$\circ$};
        \draw [-] (51) -- (50);
        \draw [-] (51) -- (52);
        \draw [-] (52) -- (53);
        \draw [-] (52) -- (54);
        \draw [-] (54) -- (55);
        \draw [-] (59) -- (58);
        \draw [-] (57) -- (56);
        \draw [-] (57) -- (58);
        \draw [-] (56) -- (55);
    \end{tikzpicture}
    \caption{The extended Euclidean graphs}
    \label{ExEuclidean}
\end{figure}

The following lemma allows us to establish wildness by detecting a bounded
wild subcategory in a locally bounded Galois covering.

\begin{lem}\label{CriterionForWild}
    Let $A$ be a representation-infinite finite-dimensional algebra which admits a Galois covering $\textup{Gal}(A)\to A$ of locally bounded ${\bf k}$-categories. If $\textup{Gal}(A)$ contains a bounded full subcategory of wild representation type, then $A$ is wild. In particular, this holds if $\textup{Gal}(A)$ contains a convex hypercritical subcategory.
\end{lem}

\begin{proof}
    Suppose that $A$ is tame. By \cite[Proposition 2.1]{LS}, the locally bounded category $\textup{Gal}(A)$ is tame. By definition, every bounded full subcategory of a tame locally bounded category is tame. This contradicts the existence of a bounded full wild subcategory of $\textup{Gal}(A)$. Hence $A$ is not tame. Since $A$ is representation-infinite, Drozd's trichotomy implies that $A$ is wild.
\end{proof}

The following lemma allows us to reduce the wildness analysis to the smallest relevant scalar multiples of the minimal symmetrizer.

\begin{lem}\label{prop:symmetrizer-quotient}
Let $H=H(C,D,\Omega)$ be a GLS algebra of type $C$ with a symmetrizer $D=\operatorname{diag}(d_1,d_2,\dots,d_n)$. For any integers $s\geq t\geq1$, there is a natural surjective algebra homomorphism
$$H(C,sD,\Omega)\twoheadrightarrow H(C,tD,\Omega).$$
Consequently, if $H(C,tD,\Omega)$ is wild, then
$H(C,sD,\Omega)$ is wild.
\end{lem}

\begin{proof}
The two algebras have the same quiver and commutativity relations, and
$$H(C,tD,\Omega)\cong
H(C,sD,\Omega)/
\langle\varepsilon_i^{td_i}\mid 1\le i\le n\rangle.$$
The last assertion follows since an algebra with a wild quotient is wild.
\end{proof}

\subsection{Typical examples}
We next present the explicit wildness witnesses needed in the proof of
Theorem \ref{Tame}. In each example, we exhibit a convex hypercritical
subcategory in the Galois covering of the corresponding GLS algebra.

\begin{example}\label{D4wild}
    Let $H=H(C,D,\Omega)$ be the GLS algebra of type $D_4$ with $D=\operatorname{diag}(2,2,2,2)$. Thus, the underlying graph $Q^{\textup{o}}$ is given by
    \begin{center}
\begin{tikzpicture}[scale=0.7]
\node (1) at (0,0.8) {$1$};
\node (2) at (0,-0.8) {$2$};
\node (3) at (1.5,0) {$3$};
\node (4) at (3,0) {$4$.};
\draw [-] (1) -- (3);
\node () at (0.8,0.65) {$\alpha$};
\node () at (0.7,-0.2) {$\beta$};
\node () at (2.2,0.25) {$\gamma$};
\draw [-] (2) -- (3);
\draw [-] (4) -- (3);
\end{tikzpicture}
\end{center}
By Proposition \ref{GaloisCovering}, $H$ admits a locally bounded Galois covering ${\bf{k}}Q^{\mathbb{Z},D}/I^{\mathbb{Z},D}$, where 
$(Q^{\mathbb{Z},D})_0=\{(i,j)\mid i=1,2,3,4,\;j\in\mathbb{Z}\}$ and
{\small{$$\begin{aligned}
     (Q^{\mathbb{Z},D})_1&=\{(\varepsilon_i,j):\;(i,j)\to(i,j+1)\mid i=1,2,3,4,\;j\in\mathbb{Z}\}\cup\{(\alpha,j):\;(1,j)-(3,j)\mid j\in\mathbb{Z}\}  \\
     & \cup\{(\beta,j):\;(2,j)-(3,j)\mid j\in\mathbb{Z}\}\cup\{(\gamma,j):\;(3,j)-(4,j)\mid j\in\mathbb{Z}\},
\end{aligned}$$}}
For each $j\in\mathbb{Z}$, the arrows $(\alpha,j)$, $(\beta,j)$, and $(\gamma,j)$ have the same orientations as $\alpha$, $\beta$, and $\gamma$, respectively. Denote by $t$ the number of arrows $\alpha$, $\beta$ and $\gamma$ whose source is $3$.

(1) If $t=0$, the convex subcategory indicated by the rectangles in the left-hand diagram below is hypercritical of type $\widetilde{\widetilde{D}}_4$.

(2) If $t=1$, the convex subcategory indicated by the rectangles in the right-hand diagram below is hypercritical of type $\widetilde{\widetilde{E}}_6$.

\begin{figure}[h]
    \begin{center}
\begin{tikzpicture}
    \node (10) at (-1.5,0.8) {$(1,0)$};
    \node (20) at (-2.5,-0.4) {$(2,0)$};
    {\node[draw] (11) at (-1.5,-1.2) {$(1,1)$};}
    {\node[draw] (21) at (-2.5,-2.4) {$(2,1)$};}
    \node (12) at (-1.5,-3.2) {$(1,2)$};
    \node (22) at (-2.5,-4.4) {$(2,2)$};
    {\node[draw] (30) at (0,0) {$(3,0)$};}
    {\node[draw] (31) at (0,-2) {$(3,1)$};}
   \node (32) at (0,-4) {$(3,2)$};
   \node (40) at (2,0) {$(4,0)$};
    {\node[draw] (41) at (2,-2) {$(4,1)$};}
   {\node[draw] (42) at (2,-4) {$(4,2)$};}
    \draw[->](10) to (30);
    \draw[->](20) to (30);
    {\draw[->](11) to (31);}
    {\draw[->](21) to (31);}
    \draw[->](12) to (32);
    \draw[->](22) to (32);
    \draw[->](10) to (11);
    \draw[->](11) to (12);
    \draw[->](20) to (21);
    \draw[->](21) to (22);
    {\draw[->](30) to (31);}
    \draw[->](31) to (32);
    \draw[->](40) to (41);
    {\draw[->](41) to (42);}
    \draw[<-](30) to (40);
    {\draw[<-](31) to (41);}
    \draw[<-](32) to (42);

    \node (50) at (6,1) {$(1,0)$};
    \node (60) at (5,0.2) {$(2,0)$};
    \node (51) at (6,-0.5) {$(1,1)$};
    \node (61) at (5,-1.3) {$(2,1)$};
    {\node[draw] (52) at (6,-2) {$(1,2)$};}
    {\node[draw] (62) at (5,-2.8) {$(2,2)$};}
    {\node[draw] (53) at (6,-3.5) {$(1,3)$};}
    {\node[draw] (63) at (5,-4.4) {$(2,3)$};}
    \node (70) at (7.7,0.5) {$(3,0)$};
    {\node[draw] (71) at (7.7,-1) {$(3,1)$};}
    {\node[draw] (72) at (7.7,-2.5) {$(3,2)$};}
    \node (73) at (7.7,-4) {$(3,3)$};
    {\node[draw] (80) at (9.7,0.5) {$(4,0)$};}
    {\node[draw] (81) at (9.7,-1) {$(4,1)$};}
    \node (82) at (9.7,-2.5) {$(4,2)$};
    \node (83) at (9.7,-4) {$(4,3)$};
    \draw[->] (50)to(70);
    \draw[->] (51)to(71);
    {\draw[->] (52)to(72);}
    \draw[->] (53)to(73);
    \draw[->] (60)to(70);
    \draw[->] (61)to(71);
    {\draw[->] (62)to(72);}
    \draw[->] (63)to(73);
    \draw[->] (50)to(51);
    \draw[->] (51)to(52);
    {\draw[->] (52)to(53);}
    \draw[->] (60)to(61);
    \draw[->] (61)to(62);
    {\draw[->] (62)to(63);}
    \draw[->] (70)to(71);
    {\draw[->] (71)to(72);}
    \draw[->] (72)to(73);
     {\draw[->] (80)to(81);}
    \draw[->] (81)to(82);
    \draw[->] (82)to(83);
     \draw[->] (70)to(80);
    {\draw[->] (71)to(81);}
    \draw[->] (72)to(82);
    \draw[->] (73)to(83);
\end{tikzpicture}
\end{center}
\end{figure}

(3) If $t=2$, there is a convex subcategory that is the dual of the one described in (2).

(4) If $t=3$, there is a convex subcategory that is the dual of the one described in (1).

Therefore, this GLS algebra $H$ is wild by Lemma \ref{CriterionForWild}.
\end{example}

\begin{example}\label{B2wild}
    Let $H=H(C,D,\Omega)$ be the GLS algebra of type $B_2$ with symmetrizer $D=\operatorname{diag}(4,2)$, that is, $H$ is given by the quiver
    \begin{center}
\begin{tikzpicture}[scale=1.1]
\node (-20) at (0,0) {$1$};
\node (-20) at (1.5,0) {$2$};
\node (-20) at (0.75,-0.2) {$\alpha$};
\draw [->] (0.25,0) -- (1.25,0);
\draw[-latex] (-0.2,0.1) .. controls (-0.5,0.7) and (0.5,0.7) .. (0.2,0.05);
\draw[-latex] (1.3,0.1) .. controls (1,0.7) and (2,0.7) .. (1.7,0.05);
\node (-20) at (0,0.8) {$\varepsilon_1$};
\node (-20) at (1.5,0.8) {$\varepsilon_2$};
\end{tikzpicture}
\end{center}
with relations $\varepsilon_1^4,\varepsilon_2^2,\varepsilon_1^2\alpha-\alpha\varepsilon_2$. Then according to Proposition \ref{GaloisCovering}, $H$ admits a locally bounded Galois covering ${\bf{k}}Q^{\mathbb{Z},D}/I^{\mathbb{Z},D}$, where 
$(Q^{\mathbb{Z},D})_0=\{(1,i),(2,j)\mid i,j\in\mathbb{Z}\}$ and
$$(Q^{\mathbb{Z},D})_1=\{(\alpha,i):\;(1,i)\to(2,i)\mid i\in\mathbb{Z}\}\cup\{(\varepsilon_i,j):\;(i,j)\to(i,j+i)\mid i=1,2,\;j\in\mathbb{Z}\}.$$
In this Galois covering, there exists a convex subcategory marked by rectangles in the figure below, which is 
hypercritical of type $\widetilde{\widetilde{E}}_6$.
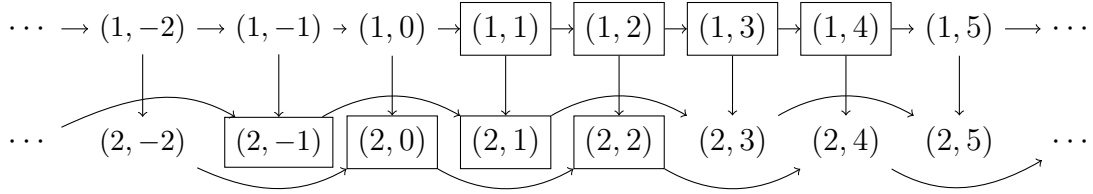
\begin{center}
\begin{tikzpicture}
    \node (10) at (0,1.5) {$(1,0)$};
    \node[draw] (11) at (1.5,1.5) {{$(1,1)$}};
    \node[draw] (12) at (3,1.5) {{$(1,2)$}};
    \node (1-1) at (-1.5,1.5) {\small{$(1,-1)$}};
    \node[draw] (13) at (4.5,1.5) {{$(1,3)$}};
    \node[draw] (14) at (6,1.5) {{$(1,4)$}};
    \node (23) at (4.5,0) {$(2,3)$};
    \node (24) at (6,0) {$(2,4)$};
    \node[draw] (20) at (0,0) {{$(2,0)$}};
    \node[draw] (21) at (1.5,0) {{$(2,1)$}};
    \node[draw] (22) at (3,0) {{$(2,2)$}};
    \node[draw] (2-1) at (-1.5,0) {{\small{$(2,-1)$}}};
    \node (15) at (7.5,1.5) {$(1,5)$};
    \node (1-2) at (-3.3,1.5) {\small{$(1,-2)$}};
    \node (25) at (7.5,0) {$(2,5)$};
    \node (2-2) at (-3.3,0) {\small{$(2,-2)$}};
    \draw [->]  (10) -- (11); 
    \node (16) at (9,1.5) {$\cdots$};
    \node (26) at (9,0) {$\cdots$};
    \node (1-3) at (-4.8,1.5) {$\cdots$};
    \node (2-3) at (-4.8,0) {$\cdots$};
    {\draw [->]  (11) -- (12);} 
    {\draw [->]  (12) -- (13);} 
    {\draw [->]  (13) -- (14);} 
    \draw [->]  (13) -- (23);
    \draw [->]  (14) -- (24);
    \draw [->]  (14) -- (15);
    \draw [->]  (15) -- (16);
    \draw [->]  (1-1) -- (10); 
    \draw [->]  (1-2) -- (1-1); 
    \draw [->]  (10) -- (20); 
    \draw [->]  (15) -- (25); 
    \draw [->]  (1-2) -- (2-2);
    \draw [->]  (1-3) -- (1-2); 
    {\draw [->]  (11) -- (21);}
    {\draw [->]  (12) -- (22);}
    \draw [->]  (1-1) -- (2-1); 
    {\draw [->] (2-1) .. controls (-0.3,0.7) and (0.3,0.7) .. (21);}
    {\draw [->] (20) .. controls (1.2,-0.7) and (1.8,-0.7) .. (22);}
    \draw [->] (22) .. controls (4.2,-0.7) and (4.8,-0.7) .. (24);
    \draw [->] (24) .. controls (7.2,-0.7) and (7.8,-0.7) .. (26);
    \draw [->] (2-2) .. controls (-1.8,-0.7) and (-1.2,-0.7) .. (20);
    \draw [->] (21) .. controls (2.7,0.7) and (3.3,0.7) .. (23);
    \draw [->] (23) .. controls (5.7,0.7) and (6.3,0.7) .. (25);
    \draw [->] (2-3) .. controls (-3.3,0.7) and (-2.7,0.7) .. (2-1);
\end{tikzpicture}
\end{center}
Therefore, this GLS algebra is wild by Lemma \ref{CriterionForWild}.
\end{example}

\begin{rem}\label{dualConst}
    Let $\Omega'$ be the orientation obtained from $\Omega$ by reversing the arrow $\alpha$ in Example \ref{B2wild}. The Galois covering of $H(C,D,\Omega')$ is the opposite of the Galois covering of $H(C,D,\Omega)$ and therefore contains a convex hypercritical subcategory of type $\widetilde{\widetilde{E}}_6$, dual to the one given in Example \ref{B2wild}. Hence, whenever only the existence of a convex hypercritical subcategory is relevant, we consider one orientation and omit its dual.
\end{rem}

\begin{example}\label{B4wild}
Let $H=H(C,D,\Omega)$ be the GLS algebra of type $B_4$ with minimal symmetrizer $D=\operatorname{diag}(2,2,2,1)$, that is, $H$ is given by the quiver
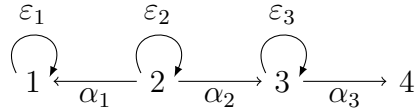
\begin{center}
\begin{tikzpicture}[scale=1.1]
\node (-20) at (0,0) {$1$};
\node (-20) at (1.5,0) {$2$};
\node () at (3,0) {$3$};
\node () at (4.5,0) {$4$};
\node (-20) at (0.75,-0.2) {$\alpha_1$};
\node (-20) at (2.25,-0.2) {$\alpha_2$};
\node (-20) at (3.75,-0.2) {$\alpha_3$};
\draw [<-] (0.25,0) -- (1.25,0);
\draw [->] (1.75,0) -- (2.75,0);
\draw [->] (3.25,0) -- (4.25,0);
\draw[-latex] (-0.2,0.1) .. controls (-0.5,0.7) and (0.5,0.7) .. (0.2,0.05);
\draw[-latex] (1.3,0.1) .. controls (1,0.7) and (2,0.7) .. (1.7,0.05);
\draw[-latex] (2.8,0.1) .. controls (2.5,0.7) and (3.5,0.7) .. (3.2,0.05);
\node (-20) at (0,0.8) {$\varepsilon_1$};
\node (-20) at (1.5,0.8) {$\varepsilon_2$};
\node (-20) at (3,0.8) {$\varepsilon_3$};
\end{tikzpicture}
\end{center}
with relations $\varepsilon_1^2,\varepsilon_2^2,\varepsilon_3^2,\varepsilon_2\alpha_1-\alpha_1\varepsilon_1,\varepsilon_2\alpha_2-\alpha_2\varepsilon_3$. Then $H$ admits a locally bounded Galois covering ${\bf{k}}Q^{\mathbb{Z},D}/I^{\mathbb{Z},D}$, where 
$(Q^{\mathbb{Z},D})_0=\{(i,j)\mid i=1,2,3,4,\;j\in\mathbb{Z}\}$ and
$$\begin{array}{c}
     (Q^{\mathbb{Z},D})_1=\{(\alpha_i,j):\;(s(\alpha_i),j)\to(t(\alpha_i),j)\mid i=1,2,3,\;j\in\mathbb{Z}\}  \\
     \cup\{(\varepsilon_i,j):\;(i,j)\to(i,j+1)\mid i=1,2,3,\;j\in\mathbb{Z}\}.
\end{array}$$
In this Galois covering, there exists a convex subcategory marked by rectangles in the figure below, which is 
hypercritical of type $\widetilde{\widetilde{D}}_5$. It implies that this GLS algebra is wild by Lemma \ref{CriterionForWild}.
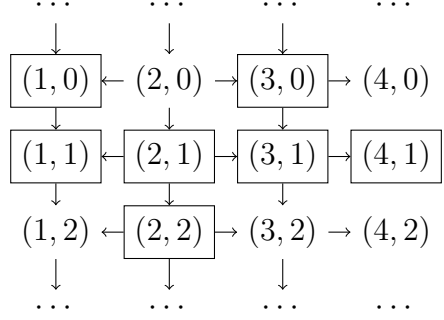
\begin{center}
\begin{tikzpicture}
    \node (1-1) at (0,1) {$\cdots$};
    \node (2-1) at (1.5,1) {$\cdots$};
    \node (3-1) at (3,1) {$\cdots$};
    \node (4-1) at (4.5,1) {$\cdots$};
    \node (13) at (0,-3) {$\cdots$};
    \node (23) at (1.5,-3) {$\cdots$};
    \node (33) at (3,-3) {$\cdots$};
    \node (43) at (4.5,-3) {$\cdots$};
    {\node[draw] (10) at (0,0) {$(1,0)$};}
    {\node[draw] (11) at (0,-1) {$(1,1)$};}
    \node (12) at (0,-2) {$(1,2)$};
    \node (20) at (1.5,0) {$(2,0)$};
    {\node[draw] (21) at (1.5,-1) {$(2,1)$};}
    {\node[draw] (22) at (1.5,-2) {$(2,2)$};}
    {\node[draw] (30) at (3,0) {$(3,0)$};}
    {\node[draw] (31) at (3,-1) {$(3,1)$};}
    \node (32) at (3,-2) {$(3,2)$};
    \node (40) at (4.5,0) {$(4,0)$};
    {\node[draw] (41) at (4.5,-1) {$(4,1)$};}
    \node (42) at (4.5,-2) {$(4,2)$};
    \draw[->] (20)--(10);
    \draw[->] (20)--(30);
    {\draw[->] (21)--(11);}
    {\draw[->] (21)--(31);}
    \draw[->] (22)--(12);
    \draw[->] (22)--(32);
    \draw[->] (30)--(40);
    \draw[->] (32)--(42);
    {\draw[->] (31)--(41);}
    \draw[->] (11)--(12);
    \draw[->] (31)--(32);
    \draw[->] (20)--(21);
    {\draw[->] (10)--(11);}
    {\draw[->] (30)--(31);}
    {\draw[->] (21)--(22);}
    \draw[->] (1-1)--(10);
    \draw[->] (2-1)--(20);
    \draw[->] (3-1)--(30);
    \draw[->] (12)--(13);
    \draw[->] (22)--(23);
    \draw[->] (32)--(33);
\end{tikzpicture}
\end{center}
\end{example}

\begin{rem}\label{RemarkB4}
For arbitrary orientation, we claim the GLS algebra of type $B_4$ is wild. 

In fact, reversing the arrow $\alpha_3$ in the above example yields another convex hypercritical subcategory, again of type $\widetilde{\widetilde{D}}_5$, by reversing the arrow $(3,1)\to (4,1)$ in the original one. 
Since the orientation of $\alpha_3$ does not affect the existence of such a subcategory in the Galois covering,  we may represent $\alpha_3$ by an undirected line in the sequel. 

  If we reverse both $\alpha_1$ and $\alpha_2$, then the Galois covering  contains a convex hypercritical subcategory obtained via the dual construction in Remark \ref{dualConst}.
  
  If we reverse $\alpha_1$ or $\alpha_2$, then the Galois covering contains a convex hypercritical subcategory as shown below, both of type $\widetilde{\widetilde{E}}_6$ according to \cite{Unger1990}. This proves the claim. 
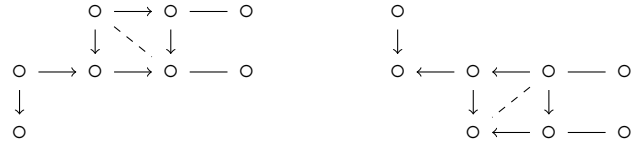
\begin{center}
    \begin{tikzpicture}
        \node (1) at (0,0.8) {$\circ$};
        \node (2) at (1,0.8) {$\circ$};
        \node (3) at (0,0) {$\circ$};
        \node (4) at (1,0) {$\circ$};
        \node (5) at (-1,0) {$\circ$};
        \node (6) at (2,0) {$\circ$};
        \node (7) at (2,0.8) {$\circ$};
        \node (8) at (-1,-0.8) {$\circ$};
        \draw[->] (1) to (2);
        \draw[->] (1) to (3);
        \draw[->] (2) to (4);
        \draw[->] (3) to (4);
        \draw[->] (5) to (3);
        \draw[->] (5) to (8);
        \draw[-] (6) to (4);
        \draw[-] (2) to (7);
        \draw[dashed] (1) to (4);
         \node (11) at (5,-0.8) {$\circ$};
        \node (12) at (6,-0.8) {$\circ$};
        \node (13) at (5,0) {$\circ$};
        \node (14) at (6,0) {$\circ$};
        \node (15) at (4,0) {$\circ$};
        \node (16) at (7,0) {$\circ$};
        \node (17) at (7,-0.8) {$\circ$};
        \node (18) at (4,0.8) {$\circ$};
        \draw[<-] (11) to (12);
        \draw[<-] (11) to (13);
        \draw[<-] (12) to (14);
        \draw[<-] (13) to (14);
        \draw[<-] (15) to (13);
        \draw[<-] (15) to (18);
        \draw[-] (16) to (14);
        \draw[-] (12) to (17);
        \draw[dashed] (14) to (11);
    \end{tikzpicture}
\end{center}
\end{rem}

\begin{example}\label{G2Wild}
    Let $H=H(C,D,\Omega)$ be the GLS algebra of type $G_2$ with symmetrizer $D=\operatorname{diag}(6,2)$, that is, $H$ is given by the same quiver as in Example \ref{B2wild}, but 
with relations $\varepsilon_1^6,\varepsilon_2^2,\varepsilon_1^3\alpha-\alpha\varepsilon_2$. Then $H$ admits a locally bounded Galois covering ${\bf{k}}Q^{\mathbb{Z},D}/I^{\mathbb{Z},D}$, where 
$(Q^{\mathbb{Z},D})_0=\{(i,j)\mid i=1,2,\;j\in\mathbb{Z}\}$ and
$$\begin{array}{c}
     (Q^{\mathbb{Z},D})_1=\{(\alpha,i):\;(1,i)\to(2,i)\mid i\in\mathbb{Z}\}  \\
     \cup\{(\varepsilon_1,j):\;(1,j)\to(1,j+1)\mid j\in\mathbb{Z}\}\cup\{(\varepsilon_2,j):\;(2,j)\to(2,j+3)\mid j\in\mathbb{Z}\}.
\end{array}$$
In this Galois covering, there exists a convex subcategory marked by rectangles in the figure below, which is 
hypercritical of type $\widetilde{\widetilde{D}}_5$. Hence $H$ is wild by Lemma \ref{CriterionForWild}.
\begin{center}
\begin{tikzpicture}
    \node (10) at (0,0) {\tiny{$(1,0)$}};
    \node[draw] (11) at (1.5,0) {\tiny{$(1,1)$}};
    \node[draw] (12) at (3,0) {\tiny{$(1,2)$}};
    \node[draw] (13) at (4.5,0) {\tiny{$(1,3)$}};
    \node[draw] (14) at (6,0) {\tiny{$(1,4)$}};
    \node (15) at (7.5,0) {\tiny{$(1,5)$}};
    \node (20) at (0,-1) {\tiny{$(2,0)$}};
    \node[draw] (21) at (1.5,-1) {\tiny{$(2,1)$}};
    \node[draw] (22) at (3,-1) {\tiny{$(2,2)$}};
    \node[draw] (23) at (4.5,-1) {\tiny{$(2,3)$}};
    \node (24) at (6,-1) {\tiny{$(2,4)$}};
    \node (25) at (7.5,-1) {\tiny{$(2,5)$}};
    \node (16) at (9,0) {\tiny{$\cdots$}};
    \node (26) at (9,-1) {\tiny{$\cdots$}};
    \node (27) at (10,-1) {};
    \node (1-1) at (-1.5,0) {\tiny{$\cdots$}};
    \node (2-1) at (-1.5,-1) {\tiny{$\cdots$}};
    \node (2-2) at (-2.5,-1) {};
    \draw[->](1-1)--(10);
    \draw[->](15)--(16);
    \draw[->](10)--(11);
    \draw[->](11)--(12);
    \draw[->](12)--(13);
    \draw[->](13)--(14);
    \draw[->](14)--(15);
    \draw[->](10)--(20);
    \draw[->](11)--(21);
    \draw[->](12)--(22);
    \draw[->](13)--(23);
    \draw[->](14)--(24);
    \draw[->](15)--(25);
    \draw [->] (2-2) .. controls (-1.7,-0.4) and (0.2,-0.4) .. (21);
    \draw [->] (20) .. controls (1.3,-0.4) and (3.2,-0.4) .. (23);
    \draw [->] (21) .. controls (2.8,-1.6) and (4.7,-1.6) .. (24);
    \draw [->] (22) .. controls (4.3,-0.4) and (6.2,-0.4) .. (25);
    \draw [->] (23) .. controls (5.8,-1.6) and (7.7,-1.6) .. (26);
    \draw [->] (2-1) .. controls (-0.2,-1.6) and (1.7,-1.6) .. (22);
    \draw [->] (24) .. controls (7.3,-0.4) and (9.2,-0.4) .. (27);
\end{tikzpicture}
\end{center}
\end{example}

\begin{example}\label{OtherWild}
    Let $H=H(C,D,\Omega)$ be the GLS algebra of one of the following types: $$F_4,\;\widetilde{A}_{11},\;\widetilde{BD}_3,\;\widetilde{G}_{21},\;\widetilde{G}_{22}$$ with minimal symmetrizer $D$. We claim that
    $H$ is wild in each case. 
    Omitting the analogous procedure presented above, we turn instead to exhibiting the convex hypercritical subcategory in the Galois covering $\text{Gal}(H)$ of 
$H$.

%     (1) The GLS algebra of type $F_4$ is given by the quiver
%     \begin{center}
% \begin{tikzpicture}[scale=1.1]
% \node (-20) at (0,0) {$\circ$};
% \node (-20) at (1.5,0) {$\circ$};
% \node () at (3,0) {$\circ$};
% \node () at (4.5,0) {$\circ$};
% \node (-20) at (0.75,-0.2) {$\alpha_1$};
% \node (-20) at (2.25,-0.2) {$\alpha_2$};
% \node (-20) at (3.75,-0.2) {$\alpha_3$};
% \draw [-] (0.25,0) -- (1.25,0);
% \draw [-] (1.75,0) -- (2.75,0);
% \draw [-] (3.25,0) -- (4.25,0);
% \draw[-latex] (-0.2,0.1) .. controls (-0.5,0.7) and (0.5,0.7) .. (0.2,0.05);
% \draw[-latex] (1.3,0.1) .. controls (1,0.7) and (2,0.7) .. (1.7,0.05);
% \node (-20) at (0,0.8) {$\varepsilon_1$};
% \node (-20) at (1.5,0.8) {$\varepsilon_2$};
% \end{tikzpicture}
% \end{center}
% with relations $\varepsilon_1^2,\varepsilon_2^2,\alpha_1\varepsilon_2-\varepsilon_1\alpha_1$ (resp. $\varepsilon_1^2,\varepsilon_2^2,\varepsilon_2\alpha_1-\alpha_1\varepsilon_1$). 
(1) For the case of type $F_4$, $\text{Gal}(H)$ contains a convex hypercritical subcategory of type $\widetilde{\widetilde{E}}_6$ as illustrated below. 
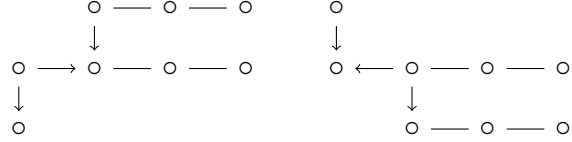
\begin{center}
    \begin{tikzpicture}
        \node (111) at (2,3) {$\circ$};
    \node (112) at (2,3.8) {$\circ$};
    \node (113) at (3,3.8) {$\circ$};
    \node (114) at (3,4.6) {$\circ$};
    \node (115) at (4,4.6) {$\circ$};
    \node (116) at (5,4.6) {$\circ$};
    \node (117) at (4,3.8) {$\circ$};
    \node (118) at (5,3.8) {$\circ$};
    \draw [->] (112) -- (111);
    \draw [->] (112) -- (113);
    \draw [->] (114) -- (113);
    \draw [-] (115) -- (114);
    \draw [-] (115) -- (116);
    \draw [-] (113) -- (117);
    \draw [-] (117) -- (118);
    \node (211) at (6.2,4.6) {$\circ$};
    \node (212) at (6.2,3.8) {$\circ$};
    \node (213) at (7.2,3.8) {$\circ$};
    \node (214) at (7.2,3) {$\circ$};
    \node (215) at (8.2,3) {$\circ$};
    \node (216) at (9.2,3) {$\circ$};
    \node (217) at (8.2,3.8) {$\circ$};
    \node (218) at (9.2,3.8) {$\circ$};
    \draw [<-] (212) -- (211);
    \draw [<-] (212) -- (213);
    \draw [<-] (214) -- (213);
    \draw [-] (215) -- (214);
    \draw [-] (215) -- (216);
    \draw [-] (213) -- (217);
    \draw [-] (217) -- (218);
    \end{tikzpicture}
\end{center}

(2) For the case of type $\widetilde{BD}_3$,  $\text{Gal}(H)$ contains a convex hypercritical subcategory of type $\widetilde{\widetilde{D}}_5$ or $\widetilde{\widetilde{D}}_4$ according to \cite{Unger1990}, as illustrated below. 
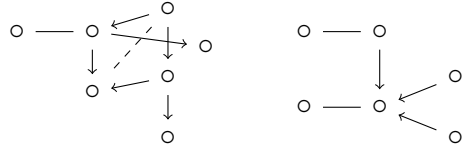
\begin{center}
    \begin{tikzpicture}
        \node (61) at (5.2,-1) {$\circ$};
    \node (62) at (6.2,-1) {$\circ$};
    \node (63) at (6.2,-1.8) {$\circ$};
    \node (64) at (7.2,-0.7) {$\circ$};
    \node (65) at (7.2,-1.6) {$\circ$};
    \node (66) at (7.2,-2.4) {$\circ$};
    \node (67) at (7.7,-1.2) {$\circ$};
    \draw [-] (61) -- (62);
    \draw [->] (62) -- (63);
    \draw [->] (64) -- (62);
    \draw [->] (64) -- (65);
    \draw [->] (65) -- (63);
    \draw [->] (65) -- (66);
    \draw [->] (62) -- (67);
    \draw [dashed] (64) -- (63);
    \node (71) at (9,-1) {$\circ$};
    \node (72) at (9,-2) {$\circ$};
    \node (73) at (10,-1) {$\circ$};
    \node (74) at (10,-2) {$\circ$};
    \node (75) at (11,-1.6) {$\circ$};
    \node (76) at (11,-2.4) {$\circ$};
    \draw [-] (71) -- (73);
    \draw [-] (72) -- (74);
    \draw [->] (73) -- (74);
    \draw [->] (75) -- (74);
    \draw [->] (76) -- (74);
    \end{tikzpicture}
\end{center}

(3) For the case of types $\widetilde{A}_{11}$, $\widetilde{G}_{21}$ or $\widetilde{G}_{22}$, the Galois covering $\text{Gal}(H)$ contains a convex hypercritical subcategory of type $\widetilde{\widetilde{D}}_5$, $\widetilde{\widetilde{E}}_6$ or $\widetilde{\widetilde{D}}_5$ according to \cite{Unger1990}, respectively, as illustrated below.
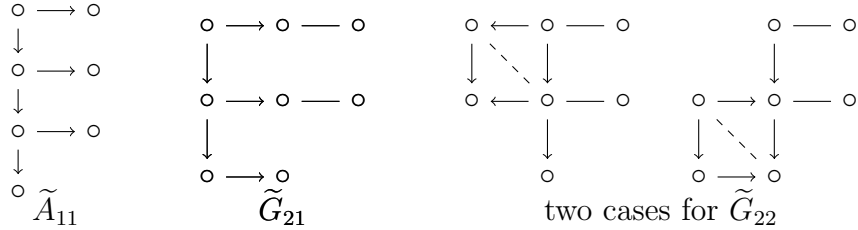
\begin{center}
    \begin{tikzpicture}
         \node (1) at (0,2.4) {$\circ$};
    \node (2) at (0,1.6) {$\circ$};
    \node (3) at (0,0.8) {$\circ$};
    \node (4) at (0,0) {$\circ$};
    \node (5) at (1,2.4) {$\circ$};
    \node (6) at (1,1.6) {$\circ$};
    \node (7) at (1,0.8) {$\circ$};
    \draw [->] (1) -- (5);
    \draw [->] (2) -- (6);
    \draw [->] (3) -- (7);
    \draw [->] (1) -- (2);
    \draw [->] (3) -- (4);
    \draw [->] (2) -- (3);
    \node () at (0.5,-0.2) {$\widetilde{A}_{11}$};
    \node (11) at (2.5,2.2) {$\circ$};
    \node (12) at (2.5,1.2) {$\circ$};
    \node (13) at (2.5,0.2) {$\circ$};
    \node (15) at (3.5,2.2) {$\circ$};
    \node (16) at (3.5,1.2) {$\circ$};
    \node (17) at (3.5,0.2) {$\circ$};
    \node (18) at (4.5,2.2) {$\circ$};
    \node (19) at (4.5,1.2) {$\circ$};
    \draw [->] (11) -- (15);
    \draw [->] (12) -- (16);
    \draw [->] (13) -- (17);
    \draw [->] (11) -- (12);
    \draw [->] (12) -- (13);
    \draw [-] (15) -- (18);
    \draw [-] (16) -- (19);
    \node () at (3.5,-0.2) {$\widetilde{G}_{21}$};
    \node (21) at (2.5,2.2) {$\circ$};
    \node (22) at (2.5,1.2) {$\circ$};
    \node (23) at (2.5,0.2) {$\circ$};
    \node (25) at (3.5,2.2) {$\circ$};
    \node (26) at (3.5,1.2) {$\circ$};
    \node (27) at (3.5,0.2) {$\circ$};
    \node (28) at (4.5,2.2) {$\circ$};
    \node (29) at (4.5,1.2) {$\circ$};
    \draw [->] (21) -- (25);
    \draw [->] (22) -- (26);
    \draw [->] (23) -- (27);
    \draw [->] (21) -- (22);
    \draw [->] (22) -- (23);
    \draw [-] (25) -- (28);
    \draw [-] (26) -- (29);
    \node () at (3.5,-0.2) {$\widetilde{G}_{21}$};
    \node (31) at (6,2.2) {$\circ$};
    \node (32) at (6,1.2) {$\circ$};
    \node (33) at (7,2.2) {$\circ$};
    \node (34) at (7,1.2) {$\circ$};
    \node (35) at (7,0.2) {$\circ$};
    \node (36) at (8,2.2) {$\circ$};
    \node (37) at (8,1.2) {$\circ$};
    \draw [->] (31) -- (32);
    \draw [<-] (31) -- (33);
    \draw [->] (33) -- (34);
    \draw [<-] (32) -- (34);
    \draw [->] (34) -- (35);
    \draw [-] (34) -- (37);
    \draw [-] (33) -- (36);
    \draw [dashed] (31) -- (34);
    \node (41) at (9,1.2) {$\circ$};
    \node (42) at (9,0.2) {$\circ$};
    \node (43) at (10,2.2) {$\circ$};
    \node (44) at (10,1.2) {$\circ$};
    \node (45) at (10,0.2) {$\circ$};
    \node (46) at (11,2.2) {$\circ$};
    \node (47) at (11,1.2) {$\circ$};
    \draw [->] (41) -- (42);
    \draw [->] (41) -- (44);
    \draw [->] (43) -- (44);
    \draw [->] (42) -- (45);
    \draw [->] (44) -- (45);
    \draw [-] (44) -- (47);
    \draw [-] (43) -- (46);
    \draw [dashed] (41) -- (45);
    \node () at (8.5,-0.2) {two cases for $\widetilde{G}_{22}$};
    \end{tikzpicture}
\end{center}
\end{example}

\section{Classification of representation-tame GLS algebras}\label{sec:tame-class}
We now combine the covering-theoretic criteria and wildness arguments established in Section~\ref{sec:galois} with the classification of representation-finite GLS algebras due to Gei\ss-Leclerc-Schr\"oer \cite{GLSI} to obtain the classification of all representation-tame GLS algebras. We first recall the relevant terminology and the known representation-finite cases.

Let $C=(c_{ij})\in M_n(\mathbb Z)$ be a connected symmetrizable generalized
Cartan matrix. Its {\emph{valued graph}} $\Gamma(C)$ has vertex set $\{1,2,\dots,n\}$. Two vertices $i$ and $j$ are joined whenever
$c_{ij}<0$, and the corresponding edge is assigned the valuation
$(|c_{ji}|,|c_{ij}|)$.

% Let $C=(c_{ij})\in M_n(\mathbb{Z})$ be a Cartan matrix. We associate to $C$ a {\emph{valued graph}} $\Gamma(C)$ defined as follows.
% \begin{itemize}
%     \item[-] The vertex set is $\{1,2,\dots,n\}$.
%     \item[-] There is an unoriented edge $\xymatrix{i\ar@{-}[r]&j}$ if and only if $c_{ij}<0$.
%     \item[-] Such an edge is assigned the value $(|c_{ji}|,|c_{ij}|)$.
% \end{itemize}

We say that $C$ is {\emph{of Dynkin type}} if its quadratic form is positive definite, and {\emph{of affine type}} if the form is positive semi-definite but not positive definite. Equivalently, $\Gamma(C)$ is a connected Dynkin diagram or a connected
Euclidean diagram, respectively, in the sense of valued graphs; see \cite{ChenWang2019}. Throughout, we always assume $C$ is connected, i.e., that $\Gamma(C)$ is a connected graph.

The representation-finite GLS algebras are classified in the following proposition.

\begin{prop}[{\cite[Proposition 13.1]{GLSI}}]\label{FiniteClassification}
The GLS algebra $H=H(C,D,\Omega)$ is representation-finite if and only if one of the following holds:  
\begin{enumerate}
    \item $C$ is of Dynkin type $A_n,C_n,D_n,E_6,E_7,E_8,B_2,B_3$ or $G_2$, and $D$ is minimal;
    \item $C$ is of Dynkin type $A_1$;
    \item $C$ is of Dynkin type $A_2$, and $D=\operatorname{diag}(2,2)$ or $D=\operatorname{diag}(3,3)$;
    \item $C$ is of Dynkin type $A_3$, and $D=\operatorname{diag}(2,2,2)$.
\end{enumerate}
\end{prop}

The next lemma excludes all Cartan matrices of indefinite type.

\begin{lem}\label{IndWild}
Let $C=(c_{ij})$ be a connected symmetrizable generalized Cartan matrix. If $C$ is neither of Dynkin nor of affine type, then the GLS algebra $H(C,D,\Omega)$ is representation-wild for every symmetrizer
$D$ and every orientation $\Omega$.
\end{lem}

\begin{proof}
Write
\[
D=\operatorname{diag}(d_1,\ldots,d_n).
\]
Since $C$ is connected and is neither of Dynkin nor of affine type,
it is of indefinite type. By the standard trichotomy for
indecomposable generalized Cartan matrices, there exists
\[
\mathbf r=(r_1,\ldots,r_n)\in\mathbb Z_{>0}^n
\]
such that
\[
C\mathbf r<0
\]
componentwise. Hence
\[
q_{C,D}(\mathbf r)
:=
\frac12\mathbf r^{\mathsf T}DC\mathbf r
=
\frac12\sum_{i=1}^n d_i r_i(C\mathbf r)_i
<0.
\]

Let
$\operatorname{Rep}_{\mathrm{lf}}(H,\mathbf r)$
denote the variety of locally free $H$-modules of rank vector
$\mathbf r$.
For a dimension vector $\mathbf{d}$, we set
$\mathrm{GL}(\mathbf{d})=\prod_{1\leq i\leq n} \mathrm{GL}(d_i)$. 
Denote 
$D\mathbf{r}=(d_1r_1,\dots,d_nr_n)$.
By
\cite[Proposition~3.1(i)]{GLSII},
\begin{align}
\dim\operatorname{Rep}_{\mathrm{lf}}(H,\mathbf r)
=
\dim\operatorname{GL}(D\mathbf r)
-q_{C,D}(\mathbf r).
\end{align}
Therefore,
\[
\dim\operatorname{Rep}_{\mathrm{lf}}(H,\mathbf r)
>
\dim\operatorname{GL}(D\mathbf r).
\]
Since
\[
\operatorname{Rep}_{\mathrm{lf}}(H,\mathbf r)
\subseteq
\operatorname{Rep}(H,D\mathbf r),
\]
we obtain
\begin{align}
\label{eq:wild}
\dim\operatorname{Rep}(H,D\mathbf r)
>
\dim\operatorname{GL}(D\mathbf r).
\end{align}

On the other hand, for every tame (or representation-finite)
finite-dimensional algebra $A$ and every dimension vector
$\mathbf d$, we have 
\[
\dim\operatorname{Rep}(A,\mathbf d)
\leq
\dim\operatorname{GL}(\mathbf d);
\]
see \cite{delaPena1991}. Thus, \eqref{eq:wild}  implies that $H$ is not tame (also not representation-finite). By Drozd's trichotomy, $H$ is representation-wild.
\end{proof}

It remains to treat the representation-infinite Dynkin and affine cases. We can now state the classification theorem, which is one of the main results of this paper.

\begin{thm}\label{Tame}
   Let $C$ be a connected symmetrizable generalized Cartan matrix. Then the GLS algebra $H=H(C,D,\Omega)$ is representation-tame if and only if one of the following conditions is satisfied:
    \begin{enumerate}
        \item $C$ is of Dynkin type $A_2$, and  $D=\operatorname{diag}(4,4)$;
        \item $C$ is of Dynkin type $A_4$, $D=\operatorname{diag}(2,2,2,2)$, and $\Omega$ is a linear orientation;
        \item $C$ is of affine type $\widetilde{A}_{n},\widetilde{C}_n,\widetilde{D}_n,\widetilde{CD}_n,\widetilde{E}_6,\widetilde{E}_7,\widetilde{E}_8$ or $\widetilde{B}_2$, and $D$ is minimal.
    \end{enumerate}
\end{thm}

\begin{proof}
By Proposition~\ref{FiniteClassification} and Lemma \ref{IndWild}, it remains to classify the
representation-tame algebras among the representation-infinite
Dynkin and affine cases. Write $D=\operatorname{diag}(d_1,d_2,\cdots,d_n)$. We distinguish two cases.

\medskip
\noindent{\bf Case I}: $d_1=d_2=\cdots=d_n:=l$.

If $l=1$, the GLS algebra $H=H(C,D,\Omega)$ is the path algebra of $Q^{\textup{o}}(C,\Omega)$. By \cite{[Rin]}, $H$ is tame precisely when $C$ is of affine type $\widetilde{A}_{n},\widetilde{D}_n,\widetilde{E}_6,\widetilde{E}_7,\widetilde{E}_8$.

If $l\geq2$, then $H\cong {\bf k}Q^{\textup{o}}(C,\Omega)\otimes_{\bf k}{\bf{k}}[x]/(x^l)$. According to \cite[Proposition 2]{Skow1987}, $H$ is tame exactly when
\begin{itemize}
    \item[-] $Q^{\textup{o}}(C,\Omega)$ is the Dynkin quiver of type $A_2$ and $l=4$, or
    \item[-]
    $Q^{\textup{o}}(C,\Omega)$ is the Dynkin quiver of type $A_4$ with a linear orientation and $l=2$.
\end{itemize}

\medskip
\noindent{\bf Case II}: the integers $d_1,d_2,\dots,d_n$ are not all equal.

If $C$ is of affine type $\widetilde{C}_n$ or $\widetilde{CD}_n$, with the convention
$\widetilde{B}_2=\widetilde{CD}_2$, and $D$ is minimal, then $H(C,D,\Omega)$ is tame
 by \cite[Theorem A]{HLS2023} and Proposition \ref{tame for arbitrary orientation} respectively. Together with Case I, this proves that all GLS algebras listed in (1)--(3) are tame. 

It remains to show that all other representation-infinite GLS algebras
are wild. Since $C$ is connected, each symmetrizer is of the form $mD_{\textup{min}}$ for some integer $m\geq1$, where $D_{\textup{min}}$ is the minimal symmetrizer of $C$. By Lemma \ref{prop:symmetrizer-quotient}, if $H(C,m_0D_{\textup{min}},\Omega)$ is wild, then so is $H(C,mD_{\textup{min}},\Omega)$ for every $m\geq m_0$.

Therefore, for each Dynkin or affine type, it suffices to consider the smallest remaining scalar multiple not covered by Proposition \ref{FiniteClassification} or by
the tame cases established above. These critical cases are listed in Table \ref{tab:remaining-wild-cases}.

\begin{table}[htbp]
\centering 
\tiny
\begin{tabular}{p{2.4cm}cp{8cm}p{2cm}}
\toprule
Type of $C$ & Symmetrizer $D$& Convex subcategory or wildness criterion&Reference\\
\midrule
$B_n$, $C_n$, $F_4$, $\widetilde{B}_n$, $\widetilde{C}_n$, $\widetilde{BD}_n$, $\widetilde{CD}_n$, $\widetilde{F}_{41}$, $\widetilde{F}_{42}$ & Non-minimal& The Galois covering contains, as a full convex subcategory, the Galois covering of a GLS algebra of type $B_2$, which
contains a convex hypercritical subcategory of type
$\widetilde{\widetilde{E}}_6$.&Example~\ref{B2wild}\\
\midrule
$\widetilde{BC}_n$ & \textup{Minimal}&The same reduction to the Galois covering of an appropriate GLS
algebra of type $B_2$ applies.&Example~\ref{B2wild}\\ 
\midrule
$B_n\;(n\geq4)$ & \textup{Minimal}&The Galois covering contains, as a full convex subcategory, the Galois covering of a GLS algebra of type $B_4$, which
contains a convex hypercritical subcategory of type
$\widetilde{\widetilde{D}}_5$ or of type $\widetilde{\widetilde{E}}_6$.&Example~\ref{B4wild}, Remark \ref{RemarkB4}\\ 
\midrule
$\widetilde{B}_n\;(n\geq3)$ & \textup{Minimal}&The Galois covering contains a convex hypercritical subcategory of type $\widetilde{\widetilde{D}}_{n+2}$.&Explicit construction below\\ 
\midrule
$\widetilde{BD}_3$ & \textup{Minimal}&The corresponding GLS algebra is wild.&Example~\ref{OtherWild}(2)\\ 
\midrule
$\widetilde{BD}_n\;(n\geq4)$ & \textup{Minimal}&The Galois covering contains, as a full convex subcategory, the Galois covering of a GLS algebra of Dynkin type $D_4$ with a
non-minimal symmetrizer; the latter contains a convex hypercritical
subcategory of type
$\widetilde{\widetilde{D}}_4$ or of type $\widetilde{\widetilde{E}}_6$.&Example~\ref{D4wild}\\ 
\midrule
$F_4$, $\widetilde{F}_{41}$,
$\widetilde{F}_{42}$&Minimal&The Galois covering contains, as a full convex subcategory, the
Galois covering of a GLS algebra of type $F_4$, which contains a
convex hypercritical subcategory of type
$\widetilde{\widetilde{E}}_6$.
&Example~\ref{OtherWild}(1)\\
\midrule
$G_2$&Non-minimal&The corresponding GLS algebra is wild.&Example \ref{G2Wild}\\
\midrule
$\widetilde{A}_{11},\widetilde{G}_{21}$, or $\widetilde{G}_{22}$&Minimal&The corresponding GLS algebra is wild.&Example~\ref{OtherWild}(3)\\
\bottomrule
\end{tabular}
\vspace{10pt}
\caption{The remaining representation-infinite cases.}
\label{tab:remaining-wild-cases}
\end{table}

For $C$ of type $\widetilde{B}_n$ with $n\geq3$ and $D$ minimal, the convex hypercritical subcategory of type
$\widetilde{\widetilde{D}}_{n+2}$ occurring in its Galois covering has the following form:
\begin{center}
    \begin{tikzpicture}
    \node (52) at (0,-1.6) {$\circ$};
    \node (53) at (1,-0.8) {$\circ$};
    \node (54) at (1,-1.6) {$\circ$};
    \node (55) at (2,-1.6) {$\circ$};
    \node (57) at (3.3,-1.6) {$\cdots$};
    \draw[-] (2.3,-1.6)--(2.8,-1.6);
    \draw[-] (3.8,-1.6)--(4.3,-1.6);
    \node (58) at (4.6,-1.6) {$\circ$};
    \node (59) at (5.6,-1.6) {$\circ$};
    \node (60) at (5.6,-2.4) {$\circ$};
    \node (61) at (6.6,-1.6) {$\circ$};
    \node (62) at (6.6,-2.4) {$\circ$};
    \node () at (1.3,-1.2) {$\alpha_1$};
    \node () at (1.5,-1.8) {$\alpha_2$};
    \node () at (5.1,-1.4) {$\alpha_3$};
    \node () at (5.9,-2) {$\alpha_4$};
    \draw [-] (52) -- (54);
    \draw [-] (53) -- (54);
    \draw [-] (54) -- (55);
    \draw [-] (58) -- (59);
    \draw [-] (59) -- (60);
    \draw [-] (59) -- (61);
    \draw [-] (60) -- (62);
    \end{tikzpicture}
\end{center}
Here $\alpha_1$ and $\alpha_2$ have a common source or a common target,
and the same holds for $\alpha_3$ and $\alpha_4$.

In each case listed in Table \ref{tab:remaining-wild-cases}, either the corresponding GLS algebra is already known to be wild or its Galois covering contains a convex hypercritical subcategory. Lemmas \ref{CriterionForWild} and \ref{prop:symmetrizer-quotient} together with Proposition \ref{GaloisCovering} therefore imply that every remaining
representation-infinite GLS algebra is wild. This completes the proof.
\end{proof}

\begin{rem}
The case $\widetilde A_{12}$ is omitted because its associated GLS algebra
is isomorphic to the GLS algebra of type $\widetilde A_1$.
\end{rem}

\section{$\tau$-locally free modules and a revised GLS conjecture}\label{sec:tau-free}

In this section, we study indecomposable
$\tau$-locally free modules over representation-tame
GLS algebras of affine type and establish a revised
version of the GLS conjecture.

%In this section, we classify the indecomposable $\tau$-locally free modules over representation-tame GLS algebras and formulate a revised version of the GLS conjecture. Since the conjecture is known for GLS algebras of Dynkin type, it remains to consider the affine case.

\subsection{Classification of $\tau$-locally free modules}

Let $e_i$ be the primitive idempotent corresponding to the vertex $i$ of $Q$ in the GLS algebra $H=H(C,D,\Omega)$. Denote by $H_i=e_iHe_i$ for each $i$. 
%Then 
%$$H(CD,n)_i=\left\{\begin{array}{ll}
%    {\bf k}[\varepsilon_i]/(\varepsilon_i^2) & %\textup{if}\;i=1  \\
%    {\bf k} & \textup{otherwise}
%\end{array}\right.$$

\begin{defn} [{\cite[Definition 1.1]{GLSI}}]
A right $H$-module $M$ is called {\emph{locally free}} if $M_i=Me_i$ is a free $H_i$-module for each $i\in Q_0$. An indecomposable locally free $H$-module $M$ is called {\emph{$\tau$-locally free}} if $\tau^k(M)$ is locally free for all $k\in\mathbb{Z}$.   
\end{defn}

\begin{lem} \label{lem8.2}
Let $H=H(\widetilde{CD}_n)$ with $n\geq2$, and let $N(w)$ be an ex-string module. Then the
following statements hold.
\begin{enumerate}
    \item If $w$ is a full ex-string or $1$-ex-string, then $N(w)$ is not  $\tau$-locally free.

    \item If $w$ is a $2$-ex-string, then $N(w)$ is  $\tau$-locally free.
\end{enumerate}
\end{lem}

\begin{proof}
Suppose first that $w$ is a full ex-string or a $1$-ex-string.  By the
construction of the corresponding ex-string module, $N(w)$ is not
locally free. Since every $\tau$-locally free module is locally free,
statement~(1) follows.

Now suppose that $w$ is a $2$-ex-string. Let $(w,w_2)$ be the
$2$-ex-string pair containing $w$. By the construction of
$2$-ex-string modules, both $N(w)$ and $N(w_2)$ are locally free.
By Theorem~\ref{MainTheorem}(5),
\[
    \tau N(w)\cong N(w_2)
    \qquad\text{and}\qquad
    \tau N(w_2)\cong N(w).
\]
Therefore, $\tau^kN(w)$ is locally free for every $k\in\mathbb Z$,
and thus $N(w)$ is $\tau$-locally free.
\end{proof}
% If $w$ belongs to a $2$-ex-string pair defining a stable tube of rank two, then $N(w)$ occurs at the mouth of this tube. Hence all Auslander--Reiten translates of $N(w)$ are locally free by \cite[Proposition~11.14]{GLS}. Therefore $N(w)$ is $\tau$-locally free.

\begin{thm} \label{Thm2} Let $C$ be a Cartan matrix of affine type and $H=H(C,D,\Omega)$ be a representation-tame GLS algebra. 
 Let $M$ be an indecomposable $H$-module. Then $M$ is $\tau$-locally free if and only if one
of the following is satisfied.
\begin{enumerate}
\item $M$ is preprojective.

\item $M$ is preinjective.

\item  $M$ lies in a stable tube.
%$M$ is a regular module occurring in any tube.
\end{enumerate}
 \end{thm}

\begin{proof}
By Theorem \ref{Tame}, the symmetrizer $D$ is minimal.
The simply-laced affine cases are immediate, since
$H$ is the path algebra of a Euclidean quiver.
The case of type $\widetilde C_n$ follows from
\cite[Theorem~3.19]{HLS2023}.
It remains to consider type $\widetilde{CD}_n$
with $n\geq2$.

By \cite[Proposition~11.6]{GLSI}, every indecomposable
preprojective or preinjective module is
$\tau$-locally free.

First assume that $H=H(\widetilde{CD}_n)$ has the
orientation in Figure~\ref{quiverForCD}.
The mouth modules of the stable tubes in
Theorem~\ref{MainTheorem}(4)--(6) are locally free by the
mouth-module calculation in Proposition~\ref{forthm002}(2),
Lemma~\ref{lem8.2}(2), and the construction of ex-band modules,
respectively.
Since $\tau$ permutes the mouth modules of each tube,
these modules are $\tau$-locally free.
It follows from \cite[Proposition~11.14]{GLSI}
that every module in these tubes is $\tau$-locally free.

The components in Theorem~\ref{MainTheorem}(2)--(3) contain
full ex-string or $1$-ex-string modules that are
not $\tau$-locally free.
Hence these components contain no regular
$\tau$-locally free modules by
\cite[Proposition~11.14]{GLSI}.
Finally, the component
$\mathcal H(\widetilde{CD}_n)_{PI}$ contains the
regular simple module $S_1$, which is not locally free.
The same proposition therefore excludes all regular
$\tau$-locally free modules from this component.
This proves the assertion for the orientation
in Figure~\ref{quiverForCD}.

% Now let $\Omega$ be arbitrary.
% Any two orientations of a finite tree can be connected
% by sink and source reflections avoiding a prescribed
% vertex. Hence we may connect $\Omega$ to the orientation
% in Figure~\ref{quiverForCD} by such reflections at vertices other than~$1$.
% As in the proof of Proposition~\ref{tame for arbitrary orientation}, the corresponding
% BGP reflection functors preserve stable tubes.
% They also preserve local freeness, since they leave
% the component at vertex~$1$ unchanged and
% $H_j\simeq\mathbf k$ for $j\neq1$.
% Thus every stable tube for the new orientation
% consists of locally free modules and hence of
% $\tau$-locally free modules.

% Conversely, by
% \cite[Propositions~3.8--3.9]{LS2025}, the inverse
% reflection sequence sends every regular
% $\tau$-locally free module to one for the orientation
% in Figure~\ref{quiverForCD}.
% The latter lies in a stable tube by the case already
% proved, so the original module also lies in a stable
% tube.
% The assertion follows.

Now let $\Omega$ be arbitrary.
Any two orientations of a finite tree can be connected
by sink and source reflections avoiding a prescribed
vertex. Hence we may connect $\Omega$ to the orientation
in Figure~\ref{quiverForCD} by such reflections at
vertices other than~$1$.

As in the proof of
Proposition~\ref{tame for arbitrary orientation},
the corresponding BGP reflection functors restrict
to mutually quasi-inverse equivalences between
the subcategories without the exceptional simple
summands. They carry almost split sequences inside
stable tubes to almost split sequences and therefore
preserve stable tubes in both directions.
Indeed, for a sink reflection, exactness follows
because the first term has no simple projective
summand at the reflected vertex; the right almost
split property extends from the subcategory to the
whole module category because the excluded simple
is injective and has no nonzero maps to the
indecomposable end term. The source case is dual.

These reflections also preserve local freeness,
since they leave the $H_1$-module structure at
vertex~$1$ unchanged and $H_j\simeq\mathbf k$
for $j\neq1$.
Thus every stable tube for $\Omega$ corresponds to
a stable tube for the orientation in
Figure~\ref{quiverForCD}, whose modules are locally free.
Applying the inverse reflections shows that every
module in the original tube is locally free and
hence $\tau$-locally free.

Conversely, by
\cite[Propositions~3.8--3.9]{LS2025}, the chosen
reflection sequence sends every regular
$\tau$-locally free module to one for the orientation
in Figure~\ref{quiverForCD}.
The latter lies in a stable tube by the case already
proved, so the original module also lies in a stable
tube.
The assertion follows.
\end{proof}

% \begin{proof}
% The simply-laced affine cases are immediate, since $H$ is the path
% algebra of a Euclidean quiver. The type $\widetilde C_n$ case follows
% from \cite[Theorem~3.19]{HLS2023}. Thus it remains to consider
% $H=H(\widetilde{CD}_n)$ with $n\geq2$.

% By \cite[Proposition~11.6]{GLSI}, every indecomposable
% preprojective or preinjective module is $\tau$-locally free.

% Consider now the regular components in Theorem~3.6.
% The mouth modules of the stable tubes listed in items~(4), (5), and~(6)
% are locally free by Proposition~\ref{forthm002}(2),
% by Lemma~\ref{lem8.2}(2), and by definition of ex-band modules, respectively.
% Since $\tau$ permutes the mouth modules of each tube,
% these mouth modules are $\tau$-locally free.
% It follows from \cite[Proposition~11.14]{GLSI} that every module
% in these tubes is $\tau$-locally free.

% On the other hand, the components in Theorem~3.6(2)--(3) contain
% full ex-string or $1$-ex-string modules, which are not locally free by
% Lemma~\ref{lem8.2}(1). Hence these components contain no $\tau$-locally free
% regular modules by \cite[Proposition~11.14]{GLSI}.
% Finally, the component $H(\widetilde{CD}_n)_{PI}$ contains the
% regular simple module $S_1$, which is not locally free; the same
% proposition therefore excludes all regular modules in this
% component.

% The assertion follows.
% \end{proof}

\subsection{A revised GLS conjecture}
As established by Gei\ss, Leclerc and Schr\"{o}er \cite[Theorem 1.3]{GLSI}, if $C$ is the Cartan matrix of Dynkin type, the set of  isomorphism classes of indecomposable $\tau$-locally free $H$-modules is in  bijection with the positive roots  of the complex simple Lie algebra $\mathfrak{g}(C)$ associated with $C$. For arbitrary Cartan matrices, the same authors formulated the so-called GLS conjecture, stated as follows.

\begin{conj}[{\cite[Conjecture 5.3]{GLSII}}] There exists a bijection between the positive roots of the Kac--Moody Lie algebra $\mathfrak{g}(C)$ attached to $C$  and the rank vectors of indecomposable $\tau$-locally
free $H$-modules.
\end{conj}

For the remainder of this section,
we suppose the Cartan matrix $C$ is of affine type.  
By \cite[Theorem 2]{LS2025}, every positive root of $\mathfrak{g}(C)$ occurs as the rank vector of some $\tau$-locally free $H$-module. However, the converse fails if $C$ is of type $\widetilde{B}_n$, $\widetilde{CD}_n$, $\widetilde{F}_{41}$, or $\widetilde{G}_{21}$; see \cite[Theorem 3]{LS2025}. Consequently, the GLS conjecture fails in full generality. 

We now reformulate the GLS conjecture for representation-tame
GLS algebras of affine type. The conjecture has already been
established for the simply-laced affine types and type
$\widetilde{C}_n$, so it remains to consider type
$\widetilde{CD}_n$ with $n\geq 2$.
We first recall the standard decomposition of the root system
of the Kac--Moody algebra $\mathfrak{g}(C)$:
\[
\Delta=\Delta_{\mathrm{re}}\sqcup\Delta_{\mathrm{im}},
\]
where $\Delta_{\mathrm{re}}$ denotes the set of real roots,
$\Delta_{\mathrm{im}}=\{m\delta\mid m\in\mathbb{Z}\setminus\{0\}\}$
is the set of imaginary roots, and $\delta$ is the unique
minimal positive imaginary root associated with $C$.
For type $\widetilde{CD}_n$, the set of real roots decomposes
as
\[
\Delta_{\mathrm{re}}
=\Delta_{\mathrm{re}}^{s}\sqcup\Delta_{\mathrm{re}}^{\ell},
\]
where $\Delta_{\mathrm{re}}^{s}$ and
$\Delta_{\mathrm{re}}^{\ell}$ denote the sets of short and
long real roots, respectively.

By removing the affine vertex from the Dynkin diagram of $C$, we obtain a finite-type Cartan submatrix associated to a simple finite-dimensional Lie algebra $\mathfrak{g}_0$, with finite root system decomposed disjointly as $$\Delta_0 = \Delta_0^{\mathrm{s}} \sqcup \Delta_0^{\ell}.$$ The distinct translation rules for short and long finite roots along $\delta$ are recorded in the following lemma, taken from \cite[Theorem 17.17]{Car}.

\begin{lem}\label{lem:root-translation}
Let $C$ be a Cartan matrix of type $\widetilde{CD}_n$ with $n\geq2$. Then
\[
\begin{aligned}
\Delta_{\mathrm{re}}^{\mathrm{s}} &= \big\{\alpha + r\delta \mid \alpha \in \Delta_{0}^{\mathrm{s}},\; r \in \mathbb{Z}\big\},\\
\Delta_{\mathrm{re}}^{\ell} &= \big\{\alpha + 2r\delta \mid \alpha \in \Delta_{0}^{\ell},\; r \in \mathbb{Z}\big\}.
\end{aligned}
\]
\end{lem}

%Since the rank vector of any indecomposable  preprojective or preinjective module is a positive root \cite[Proposition 11.6]{GLSI}, we say that a root is preprojective, preinjective, and regular (resp. regular simple) if it is the rank vector of a preprojective, preinjective, and regular (resp. regular simple) module, respectively. By Theorem \ref{Thm2} we know that each regular root occurs as a rank vector of certain $\tau$-locally free module in a tube. Denote by $\Delta^+_{\textup{reg},l}$ (resp. $\Delta^+_{\textup{reg-sim},l}$) the subset of $\Delta_{\textup{re},l}$ consisting of regular (resp. regular simple) and long roots. We have $$\Delta^+_{\textup{reg},l}=\{\alpha+2r\delta|\alpha\in\Delta^+_{\textup{reg-sim},l},r\in\mathbb{Z}_{\geq 0}\}.$$
%The proof of \cite[Theorem 3.18]{LS2025} implies that 
%$$\Delta^+_{\textup{reg-sim},l}=
%{\varphi=\alpha_1+2\alpha_2+\dots+2\alpha_{n-%1}+2\alpha_n, %\phi=\alpha_1+2\alpha_2+\dots+2\alpha_{n-%1}+2\alpha_{n+1}\}.$$

Let $\Delta^+$ denote the set of positive roots of
$\mathfrak g(C)$. We identify the root lattice with
the lattice of rank vectors by identifying each
simple root $\alpha_i$ with the standard basis
vector corresponding to the vertex $i$.

For type $\widetilde{CD}_n$ with $n\geq2$, where
$\widetilde{CD}_2=\widetilde B_2$, we use the vertex
labels $\{1,\ldots,n-1,0^+,0^-\}$ from Figure~\ref{quiverForCD}.
The following lemma constructs, for an arbitrary
orientation, a stable tube of rank $2$ whose mouth
modules have positive long real rank vectors.
These vectors will be used in Theorem~\ref{Bijection}
to describe all additional non-root rank vectors.

% Let $C$ be a Cartan matrix of type $\widetilde{B}_2$ or $\widetilde{CD}_n$. 
% For convenience, we identify the root lattice with $\mathbb{Z}^{n+1}$
% by identifying its simple roots with the standard basis vectors
% $\varepsilon_1,\ldots,\varepsilon_{n+1}$ as follows:
% \[
%     \alpha_i \longleftrightarrow \varepsilon_i
%     \quad (1\leq i\leq n-1),\qquad
%     \alpha_{0^+} \longleftrightarrow \varepsilon_n,\qquad
%     \alpha_{0^-} \longleftrightarrow \varepsilon_{n+1}.
% \]
% With this identification, we state the following revised version of
% the GLS conjecture.

\begin{lem}\label{lem:long-mouth-pair}
Let $C$ be of type $\widetilde{CD}_n$ with $n\geq2$,
let $D$ be its minimal symmetrizer, and let $\Omega$
be any orientation. Set $H=H(C,D,\Omega)$.
Then there exists a stable tube of rank $2$ consisting
of $\tau$-locally free modules whose mouth modules
$X_1,X_2$ have positive long real rank vectors
\[
  \beta_i:=\underline{\operatorname{rank}}(X_i),
  \qquad i=1,2,
\]
satisfying $\beta_1+\beta_2=2\delta$.
\end{lem}

\begin{proof}
Let $H_0=H(C,D,\Omega_0)$, where $\Omega_0$ is the
orientation in Figure~\ref{quiverForCD}. Consider the $2$-ex-string pair
\[
  w_1=\gamma_{\mathrm{semi}+}^{-1}
      \varepsilon_1\gamma_{\mathrm{semi}+},
  \qquad
  w_2=\gamma_{\mathrm{semi}-}^{-1}
      \varepsilon_1\gamma_{\mathrm{semi}-}.
\]
By Theorem~\ref{MainTheorem}(5) and Lemma~\ref{lem8.2}(2), the modules
$X_i^0=N(w_i)$ are $\tau$-locally free mouth modules
of a stable tube of rank $2$. All modules in this
tube are $\tau$-locally free by Theorem~\ref{Thm2}.
Their rank vectors are
\[
  \beta_1^0
  =\alpha_1+\sum_{j=2}^{n-1}2\alpha_j+2\alpha_{0^+},
  \qquad
  \beta_2^0
  =\alpha_1+\sum_{j=2}^{n-1}2\alpha_j+2\alpha_{0^-}.
\]
These are positive long real roots and satisfy
$\beta_1^0+\beta_2^0=2\delta$.

Since the underlying graph of $C$ is a tree, there
exists a sequence of reflections at sinks or sources
$i_1,\ldots,i_t$ transforming $\Omega_0$ into $\Omega$.
Let $F$ be the corresponding composition of reflection
functors, and set
\[
  X_i=F(X_i^0),\qquad
  w=s_{i_t}\cdots s_{i_1}.
\]
By \cite[Propositions~3.8--3.9]{LS2025}, these functors
preserve regular $\tau$-locally free modules, stable
tubes and their mouth modules. Hence $X_1,X_2$ are
the mouth modules of a stable rank-$2$ tube consisting
of $\tau$-locally free $H$-modules.
The rank-vector formula for reflection functors gives
\[
  \beta_i:=\underline{\operatorname{rank}}(X_i)
  =w\beta_i^0,\qquad i=1,2.
\]
Since the Weyl group preserves real roots and their
lengths and fixes $\delta$, the vectors $\beta_1,\beta_2$
are long real roots satisfying
\[
  \beta_1+\beta_2
  =w(\beta_1^0+\beta_2^0)=2\delta.
\]
They are positive because they are rank vectors
of nonzero modules.
\end{proof}

Using the pair $\beta_1,\beta_2$ constructed above,
we formulate a revised version of the GLS conjecture
for representation-tame GLS algebras of affine type,
explicitly describing all rank vectors that occur
in addition to the positive roots.

\begin{thm}\label{Bijection}
Let $C$ be a Cartan matrix of affine type, and let
$H=H(C,D,\Omega)$ be a representation-tame GLS algebra.
\begin{enumerate}
  \item[(1)]
  If $C$ is simply laced or of type $\widetilde C_n$, then
  \[
    \bigl\{
      \underline{\operatorname{rank}}(M)
      \mid M \text{ is indecomposable and }
      \tau\text{-locally free}
    \bigr\}
    =\Delta^+.
  \]

  \item[(2)]
  If $C$ is of type $\widetilde{CD}_n$ with $n\geq2$,
  where $\widetilde{CD}_2=\widetilde B_2$, then
  \[
    \begin{aligned}
      &\bigl\{
        \underline{\operatorname{rank}}(M)
        \mid M \text{ is indecomposable and }
        \tau\text{-locally free}
      \bigr\}\\
      &\qquad
      =\Delta^+\cup
      \bigl\{
        \beta_i+(2r+1)\delta
        \mid i=1,2,\ r\in\mathbb Z_{\geq0}
      \bigr\}.
    \end{aligned}
  \]
  %where $\beta_1, \beta_2$ are constructed in the proof of  Lemma~\ref{lem:long-mouth-pair}.
\end{enumerate}
\end{thm}

\begin{proof}
By Theorem \ref{Tame}, the symmetrizer $D$ is minimal.
If $C$ is simply laced, then $H$ is the path algebra
of a Euclidean quiver, and assertion~(1) follows from
the classical root--dimension vector correspondence.
The case of type $\widetilde C_n$ follows from
\cite[Corollary~4.10]{HLS2023}.

We prove assertion~(2).
By \cite[Theorem~2]{LS2025}, every positive root occurs
as the rank vector of an indecomposable $\tau$-locally
free module. Thus it remains to determine all non-root
rank vectors.

First consider
$H_0=H(C,D,\Omega_0)$, where $\Omega_0$ is the
orientation in Figure~\ref{quiverForCD}. Let $\beta_1^0,\beta_2^0$
be the vectors in the proof of Lemma~\ref{lem:long-mouth-pair}.
In particular, they are positive long real roots and
\[
  \beta_1^0+\beta_2^0=2\delta.
\]

By Theorem~\ref{Thm2}, every indecomposable $\tau$-locally
free $H_0$-module is preprojective, preinjective,
or contained in a stable tube.
Preprojective and preinjective modules have positive
real rank vectors.

The tubes in Theorem~\ref{MainTheorem}(4) and~(6) contribute
only positive roots. Indeed, for $n\geq3$, the mouth
vectors of the tube in Theorem~\ref{MainTheorem}(4) are
\[
  \alpha_2,\ldots,\alpha_{n-1},
  \qquad
  \rho=\delta-\sum_{j=2}^{n-1}\alpha_j.
\]
They sum to $\delta$, and their nonempty proper
consecutive sums in the cyclic order are short real
roots. Rank additivity and Lemma~\ref{lem:root-translation} therefore show
that every module in this tube has a positive root
as its rank vector. For $n=2$, its mouth vector is
$\delta$, so all its rank vectors are positive
multiples of $\delta$.
The construction of ex-band modules likewise shows
that the tubes in Theorem~\ref{MainTheorem}(6) contribute only
positive multiples of $\delta$.

It remains to consider the rank-$2$ tubes in
Theorem~\ref{MainTheorem}(5). Let
$(\widetilde w_1,\widetilde w_2)$ be a
$2$-ex-string pair, and set
\[
  \mathbf r_i
  =\underline{\operatorname{rank}}N(\widetilde w_i),
  \qquad i=1,2.
\]
By Definition~\ref{2-ex-string} and the construction of ex-string
modules, there are two cases.

\smallskip
\noindent
\textit{Case 1.}
For some $k\geq1$,
\[
 \widetilde w_1=
 \gamma_{\mathrm{semi}+}^{-1}
 \theta(\delta_1\cdots\delta_k)\gamma_{\mathrm{semi}+},
 \qquad
 \widetilde w_2=
 \gamma_{\mathrm{semi}-}^{-1}
 \theta(\delta_1\cdots\delta_k)\gamma_{\mathrm{semi}-}.
\]
Then
\[
  \mathbf r_i=\beta_i^0+(k-1)\delta,
  \qquad
  \mathbf r_1+\mathbf r_2=2k\delta.
\]
An indecomposable module with quasi-socle
$N(\widetilde w_i)$ and quasi-length $m$ therefore
has rank vector
\[
 \begin{cases}
   2sk\delta,
     & m=2s,\quad s\geq1,\\[2pt]
   \beta_i^0+(k-1+2sk)\delta,
     & m=2s+1,\quad s\geq0.
 \end{cases}
\]
The even quasi-lengths give positive imaginary roots.
Since $\beta_i^0$ is a long real root,
Lemma~\ref{lem:root-translation} implies that
\[
  \beta_i^0+t\delta\in\Delta^+
  \quad\Longleftrightarrow\quad
  t\equiv0\pmod2
  \qquad (t\in\mathbb Z_{\geq0}).
\]
Thus the odd quasi-lengths give roots when $k$ is odd,
and non-root vectors of the form
$\beta_i^0+(2r+1)\delta$ when $k$ is even.
Conversely, every such non-root vector occurs as a
mouth vector by taking $k=2r+2$.

\smallskip
\noindent
\textit{Case 2.}
For some $k\geq1$,
\[
 \widetilde w_1=
 \gamma_{\mathrm{semi}-}^{-1}
 \theta(\delta_1\cdots\delta_k)\gamma_{\mathrm{semi}+},
 \qquad
 \widetilde w_2=
 \gamma_{\mathrm{semi}+}^{-1}
 \theta(\delta_1\cdots\delta_k)\gamma_{\mathrm{semi}-}.
\]
Here $\mathbf r_1=\mathbf r_2=k\delta$.
Hence every module of quasi-length $m$ has rank
vector $mk\delta$, a positive imaginary root.

Consequently, the non-root rank vectors for $H_0$
are precisely
\[
  \left\{
    \beta_i^0+(2r+1)\delta
    \;\middle|\;
    i=1,2,\ r\in\mathbb Z_{\geq0}
  \right\}.
\]

Now let $\Omega$ be arbitrary. 
Choose the reflection sequence and the Weyl group
element $w$ used in the proof of
Lemma~\ref{lem:long-mouth-pair}, so that
$\beta_i=w\beta_i^0$ for $i=1,2$.
By \cite[Propositions~3.8--3.9]{LS2025}, the associated
reflection functors give mutually inverse
correspondences on indecomposable regular
$\tau$-locally free modules. Their rank vectors
are transformed by $w$.

Preprojective and preinjective rank vectors are real
roots, so every non-root rank vector comes from a
regular module. Since $w$ preserves the root system
and fixes $\delta$, the non-root rank vectors for
$H(C,D,\Omega)$ are exactly
\[
  w\bigl(\beta_i^0+(2r+1)\delta\bigr)
  =\beta_i+(2r+1)\delta,
  \qquad i=1,2,\quad r\in\mathbb Z_{\geq0}.
\]
Together with the realization of every positive root,
this proves assertion~(2).
\end{proof}

\begin{rem}
Let $\mathcal L_H$ be the set of positive long real
rank vectors of $\tau$-locally free mouth modules
of stable tubes, and put
\[
  \mathcal P_H
  :=\{\alpha\in\mathcal L_H
       \mid \alpha-2\delta\notin\mathcal L_H\}.
\]
Both sets are defined for the fixed algebra $H$ and may depend
on its orientation.
In the simply-laced cases, set
$\mathcal L_H=\mathcal P_H=\varnothing$ by convention.

The mouth-vector calculations give
$\mathcal L_H=\varnothing$ in type $\widetilde C_n$ and
\[
  \mathcal L_H
  =\{\beta_i+2r\delta
      \mid i=1,2,\ r\in\mathbb Z_{\geq0}\},
  \qquad
  \mathcal P_H=\{\beta_1,\beta_2\}
\]
in type $\widetilde{CD}_n$.
Hence Theorem~\ref{Bijection} can be written uniformly as
\[
 \begin{aligned}
 &\bigl\{
   \underline{\operatorname{rank}}(M)
   \mid M\text{ is indecomposable and }
   \tau\text{-locally free}
 \bigr\}\\
 &\qquad
 =\Delta^+\cup
 \bigl\{
   \alpha+(2r+1)\delta
   \mid \alpha\in\mathcal P_H,\ r\in\mathbb Z_{\geq0}
 \bigr\}.
 \end{aligned}
\]
\end{rem}

% \begin{rem}
% Theorem~8.6 also admits the following uniform formulation.
% In the nonsimply-laced cases, let
% $\Delta^{\ell,+}_{\mathrm{re,reg\text{-}sim}}$
% denote the set of minimal elements, with respect to the
% standard root order, among the positive long real roots
% occurring as rank vectors of $\tau$-locally free mouth
% modules of stable tubes. In the simply-laced cases, set
% $\Delta^{\ell,+}_{\mathrm{re,reg\text{-}sim}}=\varnothing$
% by convention.

% This set is empty in type $\widetilde C_n$.
% For type $\widetilde{CD}_n$ with $n\geq2$,
% the mouth-vector calculation in the proof gives
% \[
%  \Delta^{\ell,+}_{\mathrm{re,reg\text{-}sim}}
%  =\{\beta_1,\beta_2\}.
% \]
% Indeed, the long real mouth vectors are precisely of the
% form $\beta_i+2r\delta$, with $i=1,2$ and
% $r\in\mathbb Z_{\geq0}$, and $\beta_1$ and $\beta_2$
% are incomparable in the standard root order.

% Consequently, in the cases covered by Theorem~8.6,
% \[
%  \left\{
%  \underline{\operatorname{rank}}M
%  \;\middle|\;
%  M\text{ is indecomposable and $\tau$-locally free}
%  \right\}
%  =\widetilde\Delta^+,
% \]
% where
% \[
%  \widetilde\Delta^+
%  :=
%  \Delta^+\cup
%  \left\{
%  \alpha+(2r+1)\delta
%  \;\middle|\;
%  \alpha\in\Delta^{\ell,+}_{\mathrm{re,reg\text{-}sim}},
%  \ r\in\mathbb Z_{\geq0}
%  \right\}.
% \]
% \end{rem}

\begin{rem}
In type $\widetilde{CD}_n$ with $n\geq2$,  where $\widetilde{CD}_2=\widetilde B_2$, the non-root rank vectors in
\cite[Propositions~6.5 and  6.8]{LS2025}
are initial instances of the family
\[
 \{\beta_i+(2r+1)\delta
   \mid i=1,2,\ r\in\mathbb Z_{\geq0}\}.
\]
Theorem~\ref{Bijection} shows that this family exhausts the non-root
rank vectors for $H$.
They occur precisely at odd quasi-length in the tubes
of Case~1 with even $k$.
\end{rem}

\begin{rem}
Theorem~\ref{Bijection} concerns rank vectors rather than isomorphism
classes. In particular, a positive real root need not
determine a unique indecomposable $\tau$-locally free module.
For example, with the orientation shown in Figure~\ref{quiverForCD}, the vector $\beta^0_i+2\delta$ occurs both as a mouth vector
in Case~1 with $k=3$ and as the rank vector of a module
of quasi-length $3$ in Case~1 with $k=1$.
These modules lie in distinct tubes and are therefore
nonisomorphic.
\end{rem}

\begin{rem}
For the orientation in Figure~\ref{quiverForCD},
$\underline{\operatorname{rank}}I_1=\alpha_1$.
Although $\alpha_1$ is a positive long real root, none of
the vectors $\alpha_1+(2r+1)\delta$, with
$r\in\mathbb Z_{\geq0}$, occurs as the rank vector of an
indecomposable $\tau$-locally free module.
Thus the additional vectors in Theorem~\ref{Bijection} cannot be
obtained by allowing arbitrary positive long real roots
in place of $\beta_1$ and $\beta_2$.
\end{rem}

\appendix\label{appendix}

\section{Equivariantization of stable tubes}

In this appendix, we investigate the equivariant categories of a tube under various $\mathbb{Z}_2$-actions. Recall that we denote by $(\mathcal{T}_p,\tau)$, or simply $\mathcal{T}_p$, the standard tube of rank $p$. %, that is, the Auslander--Reiten quiver of $\mathcal{T}_p$ is of the form $\mathbb{Z}A_{\infty}/(\tau^p)$ (see \cite{Elements2}).  
For a non-projective indecomposable object $X$ in a standard tube $\mathcal{T}_p$, there is a canonical bijection $\zeta$ from the set $X^{-}$ of arrows ending at $X$ to the set $(\tau X)^{+}$ of arrows starting at $\tau X$ (see \cite{Covering15}). 

The following proposition describes the equivariant categories
for three types of $\mathbb Z_2$-actions used in Section~\ref{Mainresult1}.

\begin{prop}\label{ImportantEquivar}
Let $G=\{e,\sigma\}$ be the cyclic group of order two.
\begin{enumerate}
\item If $G$ acts on $\mathcal{T}_p\times\mathcal{T}_p$ via the involution swapping the two factors, then $$(\mathcal{T}_p\times\mathcal{T}_p)^G\cong \mathcal{T}_p.$$
\item Assume that $G$ acts on $\mathcal{T}_p$ and that $F_{\sigma}$ fixes every object. For each arrow $f$ in the Auslander--Reiten quiver of $\mathcal{T}_p$, choose an irreducible representative, still denoted by $f$, such that $F_\sigma(f)=l_f f$, where $l_f\in\{1,-1\}$. Then the following statements hold.
\begin{itemize}
\item[(i)]
If $l_{\zeta f}l_f=1$ for each irreducible morphism $f$, then
$$(\mathcal{T}_p)^G\cong \mathcal{T}_p\times \mathcal{T}_p.$$
\item[(ii)] If $p=1$ and $l_{\zeta f}l_f=-1$ for each irreducible morphism $f$, then  $$(\mathcal{T}_1)^G\cong \mathcal{T}_2.$$
\end{itemize}
\end{enumerate}
\end{prop}

\begin{proof}
(1) For each indecomposable object $X$ in either copy of $\mathcal{T}_p$, the object $\textup{Ind}(X)$ is indecomposable, since 
$F_{\sigma}(X)\ncong X$. Let $\eta_X$ be the almost split sequence 
ending at $X$. Then $F_{\sigma}(\eta_X)$ is the almost split sequence 
ending at $F_{\sigma}(X)$. By \cite[Theorem~3.8]{ReitenRiedtmann1985Skew}, 
$\eta_X$ and $F_{\sigma}(\eta_X)$ are identified under 
the induction functor and yield, up to isomorphism, the same almost split sequence in $(\mathcal{T}_p\times\mathcal{T}_p)^G$. Thus the two tubes are identified along the $G$-orbits, and hence
$$(\mathcal{T}_p\times\mathcal{T}_p)^G\cong \mathcal{T}_p.$$

(2) (i) Since $F_\sigma$ fixes every object, we regard $F_\sigma(Y)=Y$ for every object $Y$ of $\mathcal{T}_p$. By assumption, the chosen irreducible representatives satisfy
\begin{equation}\label{lflf}
    l_{\zeta f}\cdot l_f=1.
\end{equation}

Fix an object $X$ at the mouth of $\mathcal{T}_p$. For any object $Y$, choose a walk $w=f_1f_2\cdots f_t$ from $X$ to $Y$ in the Auslander--Reiten quiver, where each $f_i$ is either an irreducible arrow or the formal inverse of one. Define
$$c_Y=\prod_{i=1}^tl_{f_i}^{\,\epsilon_i},\;\;\;\;\textup{where}\;\epsilon_i=\left\{\begin{array}{ll}
     -1&  f_i\;\textup{is an irreducible arrow}\\
     1&  f_i\;\textup{is the formal inverse of an irreducible arrow}
\end{array}\right..$$

We claim that $c_Y$ is independent of the chosen walk. Since $\mathcal T_p$ is a standard tube, every closed walk in its Auslander--Reiten quiver is, up to cancellation and the mesh relations, generated by mesh boundaries and a fundamental circuit around the tube. The scalar weight of every two-arrow mesh path
associated with $f$ is $(l_{\zeta f}l_f)^{-1}=1$. Hence every mesh boundary has weight $1$. A fundamental circuit is obtained by concatenating these two-arrow paths along a complete $\tau$-orbit, and therefore also has weight $1$. Thus every closed walk has scalar weight $1$, proving the claim.

Let $\{F'_e,F'_{\sigma}\}$ be the trivial $G$-action on $\mathcal{T}_p$. We construct a natural isomorphism
$$\eta:\;\operatorname{id}_{\mathcal{T}_p}\circ F_{\sigma}\rightarrow F'_{\sigma}\circ\operatorname{id}_{\mathcal{T}_p}$$ 
satisfying the condition $\operatorname{id}=F'_{\sigma}\eta\circ \eta F_{\sigma}$. 

Set $\eta_Y=c_Y\operatorname{id}_Y$. If $f:\;Y\rightarrow Z$ is irreducible, then
$c_Z=l_{f}^{-1}c_Y$. Consequently,
$$\eta_Z\circ F_{\sigma}(f)=c_Zl_f f=c_Y f=f\circ\eta_Y.$$
Since $\mathcal{T}_p$ is standard and hence generated by its irreducible morphisms subject to the mesh relations, $\eta$ is a
natural isomorphism from $F_{\sigma}$ to $\operatorname{id}_{\mathcal{T}_p}$.

Moreover, $c_Y\in\{1,-1\}$ for every $Y$, and hence $\operatorname{id}_{\mathcal{T}_p}(\eta_Y)\circ\eta_{F_{\sigma}(Y)}=\operatorname{id}_Y$. Thus, $(\operatorname{id}_{\mathcal{T}_p},\eta):\;\mathcal{T}_p\rightarrow\mathcal{T}_p$ is a $G$-equivariant functor. By \cite[Lemma 2.1]{CCR}, the equivariant category $(\mathcal{T}_p)^G$ with respect to the given $G$-action is equivalent to that with respect to the trivial $G$-action. By \cite[Proposition 2.1]{DRZ}, the latter is equivalent to $\mathcal{T}_p\times\mathcal{T}_p$. Hence, $(\mathcal{T}_p)^G\cong \mathcal{T}_p\times \mathcal{T}_p$.

(ii) Assume that $p=1$. Fix a ray
$$X_1\xrightarrow{f_1} X_2\xrightarrow{f_2} X_3\xrightarrow{f_3}\cdots$$ in $\mathcal{T}_1$. Define an autoequivalence $F'$ by setting $F'(f_i)=f_i$, $F'(\zeta f_i)=-\zeta f_i$, for every $i\geq1$. Similarly to the proof in part (i), we obtain an equivalence between the two $G$-actions $\{F_e,F_\sigma\}$ and $\{\operatorname{id}_{\mathcal T_1},F'\}$.

We next consider the $G$-action $\{\operatorname{id}_{\mathcal{T}_1},F'\}$ on $\operatorname{mod}^{\textup{nil}}{\bf k}[X]$, the category of nilpotent ${\bf k}[X]$-modules, which is equivalent to $\mathcal{T}_1$. Put $M_i={\bf{k}}[X]/(X^i)$ and $M_0=0$. The almost split sequences in $\operatorname{mod}^{\textup{nil}}{\bf k}[X]$ are of the form
$$0\rightarrow M_i\xrightarrow{\begin{pmatrix}
    X\\(-1)^i
\end{pmatrix}}M_{i+1}\oplus M_{i-1}\xrightarrow{\begin{pmatrix}
    (-1)^{i+1}&X
\end{pmatrix}}M_{i}\rightarrow0.$$
Consider the $\sigma$-action on ${\bf k}[X]$ defined by ${\sigma}(X)=-X$, and let $F$ denote the induced action on $\operatorname{mod}^{\textup{nil}}{\bf k}[X]$. Each $F(M_i)$ is isomorphic to $M_i$ via
$$F(M_i)\to M_i\;\;\;\;\overline{h(X)}\mapsto\overline{h(-X)}.$$
With respect to these identifications, $F$ sends every irreducible morphism labelled by $X$ to $-X$, and fixes every irreducible morphism labelled by a nonzero scalar. Thus, under the standard equivalence $$\mathcal{T}_1\cong \operatorname{mod}^{\textup{nil}}{\bf k}[X],$$
the action $F$ corresponds to the object-fixing action $F'$ on $\mathcal{T}_1$. By \cite[Lemma 2.1]{CCR}, we have $$(\mathcal{T}_1)^G\cong (\operatorname{mod}^{\textup{nil}}{\bf k}[X])^G.$$

Moreover, the skew group algebra $({\bf k}[X])[G]$ has two primitive orthogonal idempotents, namely $(e+\sigma)/2$ and $(e-\sigma)/2$. The bound quiver associated to $({\bf k}[X])[G]$ is
$$\xymatrix{\frac{e+\sigma}{2}\ar@<0.4ex>[rr]^{X-X\sigma}&&\frac{e-\sigma}{2}\ar@<0.4ex>[ll]^{X+X\sigma}.}$$
By Lemma \ref{EqFromA^GToA[G]}, $\Phi$ induces an equivalence between $(\operatorname{mod}^{\textup{nil}}{\bf k}[X])^G$ and $\operatorname{mod}^{\textup{nil}}({\bf k}[X])[G]$. Consequently,
$$(\mathcal{T}_1)^G\cong (\operatorname{mod}^{\textup{nil}}{\bf k}[X])^G\xrightarrow[\Phi]{\simeq}\operatorname{mod}^{\textup{nil}}({\bf k}[X])[G]\cong\mathcal{T}_2.$$
\end{proof}

\section{Equivariantization of string modules}

Recall that $H(\widetilde{C}_{2n-2})$ and $H(\widetilde{CD}_n)$ denote the GLS algebras of types $\widetilde{C}_{2n-2}$ and $\widetilde{CD}_n$, respectively; see Figures \ref{quiverForC} and \ref{quiverForCD} for their quivers. By Example \ref{EXAMPLEcd+c}, the $G$-action on $H(\widetilde{C}_{2n-2})$ induces an equivalence $$\Phi:\;\big(\operatorname{mod}H(\widetilde{C}_{2n-2})\big)^G\xrightarrow{\simeq}\operatorname{mod}H(\widetilde{CD}_{n}).$$ 
In this appendix, we make explicit the induced correspondence between string modules and certain ex-string modules.

First, we describe the structure of the ex-string modules $N(\gamma\gamma^{-1})$, where $\gamma=(\gamma_{0^+},\gamma_{0^-})$ is defined in Subsection \ref{exstring}. Note that the $H(\widetilde{CD}_n)$-module $N(\gamma\gamma^{-1})$ can be regarded as the module
$$\mathbf{k}\xleftarrow{(1,1)}\mathbf{k}^2\xrightarrow{(1,-1)}\mathbf{k}$$
over the quiver $0^{+}\xleftarrow{\gamma_{0^+}}n-1\xrightarrow{\gamma_{0^-}}0^{-}$.

\begin{lem}\label{A3caseIso}
\begin{enumerate}
    \item There exists an isomorphism between 
 $${{\bf k}}\xleftarrow{(1,1)}{{\bf k}}^2\xrightarrow{(1,-1)}{{\bf k }}\;\;\;\textup{and}\;\;\;{{\bf k}}\xleftarrow{(1,1)}{{\bf k}}^2\xrightarrow{(-1,1)}{{\bf k }}$$
    that preserves the basis of the vector space at vertex $n-1$.
    \item We have
    $$N(\gamma\gamma^{-1})\cong N(\gamma_{0^{+}})\oplus N(\gamma^{-1}_{0^{-}}).$$
\end{enumerate}
\end{lem}

\begin{proof}
   Both assertions follow from the commutative diagrams below.
   $$\xymatrix@R=6ex@C=10ex{{\bf k}\ar[d]_{1}&{{\bf k}}^2\ar[d]_*+!<4pt,2pt>{\tiny{\begin{pmatrix}
       1&0\\0&1
\end{pmatrix}}}\ar[l]_{(1,1)}\ar[r]^{(1,-1)}&{\bf k}\ar[d]_{-1}
&{\bf k}\ar[d]_{1}&{{\bf k}}^2\ar[d]_*+!<4pt,2pt>{\tiny{\begin{pmatrix}
       1&1\\1&-1
\end{pmatrix}}}\ar[l]_{(1,1)}\ar[r]^{(1,-1)}&{\bf k}\ar[d]_{1}\\
  {\bf k}&{{\bf k}}^2\ar[l]_{(1,1)}\ar[r]^{(-1,1)}&{\bf k}
  &{\bf k}&{{\bf k}}^2\ar[l]_{(1,0)}\ar[r]^{(0,1)}&{\bf k} }$$
\end{proof}

Define a map $\phi$ from the strings in $H(\widetilde{C}_{2n-2})$ to the ex-strings in $H(\widetilde{CD}_{n})$ by 
$$\phi(\varepsilon^{\pm1}_{\pm1})=\varepsilon^{\pm1}_{1},\;\;\phi(\beta^{\pm1}_{\pm i})=\gamma_{i}^{\pm1},\;\;\textup{and}\;\;\phi(\beta^{\pm1}_{\pm(n-1)})=\gamma^{\pm1}$$
for each $1\leq i\leq n-2$, extended multiplicatively. 

\begin{prop}\label{StToExst}
Let $w$ be a non-trivial string in $H(\widetilde{C}_{2n-2})$. Then there exists an isomorphism of $H(\widetilde{CD}_{n})$-modules
$$\Phi\circ\textup{Ind}\big(M(w)\big)\cong N\big(\phi(w)\big).$$
\end{prop}

\begin{proof}
We observe that any non-trivial string can be realized as a truncation
$$\big(\xi(\delta)\big)_{[i,j]}=(\varepsilon_1^{\delta_1}\beta_{\textup{max}}\varepsilon_{-1}^{\delta_2}\beta_{\textup{max}}^{-1}\varepsilon_1^{\delta_3}\beta_{\textup{max}}\cdots\beta_{\textup{max}}\varepsilon_{-1}^{\delta_{2k}})_{[i,j]}$$
for some $\delta=\delta_1\delta_2\cdots\delta_{2k}\in\mathcal{I}_{e}$ and $1\leq i\leq j\leq 4kn-2n-2k+2$. Since the construction below is compatible with truncation, it suffices to treat the case $w=\xi(\delta)$.

According to the construction of the string modules, we have
$$M(w)_{u}=\left\{\begin{array}{ll}
{\textup{span}_{\bf k}}\{x_i\mid i\equiv 1+u\;\textup{or}\;2-u(\textup{mod}\;4n-2)\}     & 1\leq u\leq n-1, \\
{\textup{span}_{\bf k}}\{x_i\mid i\equiv 1+n\;\textup{or}\;2-n(\textup{mod}\;4n-2)\}&u=0,\\
{\textup{span}_{\bf k}}\{x_i\mid i\equiv 2n+1+u\;\textup{or}\;2n-u(\textup{mod}\;4n-2)\}     & 1-n\leq u\leq -1.
\end{array}\right.$$
Denoting $x'_i=F_{\sigma}(x_i)$, we obtain that $F_{\sigma}(M(w))$ is a string module with
$$F_{\sigma}\big(M(w)\big)_{u}=\left\{\begin{array}{ll}
{\textup{span}_{\bf k}}\{x'_i\mid i\equiv 1-u\;\textup{or}\;2+u(\textup{mod}\;4n-2)\}     & 1-n\leq u\leq -1, \\
{\textup{span}_{\bf k}}\{x'_i\mid i\equiv 1+n\;\textup{or}\;2-n(\textup{mod}\;4n-2)\}&u=0,\\
{\textup{span}_{\bf k}}\{x'_i\mid i\equiv 2n+1-u\;\textup{or}\;2n+u(\textup{mod}\;4n-2)\}     & 1\leq u\leq n-1.
\end{array}\right.$$

By Example \ref{EXAMPLEcd+c}, for $1\leq i\leq n-1$, the primitive idempotent at vertex $i$ in $H(\widetilde{CD}_n)$ identifies---via the functor $\Phi\circ\textup{Ind}$---with the primitive idempotent at vertex $i$ in $H(\widetilde{C}_{2n-2})$, while the arrow $\gamma_i$ coincides with $\beta_i$. Meanwhile, the primitive idempotents at vertices $0^{\pm}$ take the form $(e_0\pm e_0\sigma)/2$, and the arrows $\gamma_{0^{\pm}}$ are represented by $(\beta_{n-1}\pm\beta_{1-n}\sigma)/2$.

Hence, the module $\Phi\big(\textup{Ind}(M(w))\big)$ is determined by the following vector spaces.
$$\big(\Phi\circ\textup{Ind}(M(w))\big)_u=\left\{
\begin{array}{ll}
 M(w)_u\oplus F_{\sigma}(M(w))_{u}    &   1\leq u\leq n-1, \\
{\textup{span}_{\bf k}}\{\frac{x_i+x'_i}{2}\mid i\equiv 1+n\;\textup{or}\;2-n(\textup{mod}\;4n-2)\}     & u=0^{+},\\
{\textup{span}_{\bf k}}\{\frac{x_i-x'_i}{2}\mid i\equiv 1+n\;\textup{or}\;2-n(\textup{mod}\;4n-2)\}     & u=0^{-}.
\end{array}
\right.$$
This implies that $$\Phi\circ\textup{Ind}\big(M(w)\big)=N(\varepsilon_1^{\delta_1}\gamma_{\textup{max}}\varepsilon_{1}^{\delta_2}\gamma_{\textup{max}}\cdots\gamma_{\textup{max}}\varepsilon_{1}^{\delta_{2k}})=N(\theta(\delta))=N(\phi(w)).$$
\end{proof}

For simplicity, we write $w_{[i,-j]}:=w_{[i,l(w)+1-j]}$. Recall that a non-trivial string $w$ is symmetric if $\sigma(w)=w^{-1}$. Then $w$ necessarily takes the form
$$\big(\xi(\delta)\big)_{[i,-i]}=(\varepsilon_1^{\delta_1}\beta_{\textup{max}}\varepsilon_{-1}^{\delta_2}\beta_{\textup{max}}^{-1}\varepsilon_1^{\delta_3}\beta_{\textup{max}}\cdots\beta_{\textup{max}}\varepsilon_{-1}^{\delta_{2k}})_{[i,-i]}$$
for some $\delta=\delta_1\delta_2\cdots\delta_{2k}\in\mathcal{I}_{e}$ with $\delta=(\delta_1\delta_2\cdots\delta_{k})\circ(\delta_1\delta_2\cdots\delta_{k})^{-1}$, and some $1\leq i\leq 4n-2$.

\begin{prop}\label{ClassifyStToExst} Let $w$ be a non-trivial string in $H(\widetilde{C}_{2n-2})$.
\begin{enumerate}
    \item If $w$ is non-symmetric, then $\Phi\circ\textup{Ind}\big(M(w)\big)$ is indecomposable.
    \item If $w$ is symmetric, i.e., $w=(\xi(\delta))_{[i,-i]}$ with $\delta=\delta'\circ\delta'^{-1}\in\mathcal{I}_{e}$ for some $1\leq i\leq 4n-2$, then $\Phi\circ\textup{Ind}\big(M(w)\big)$ admits a decomposition
$$\Phi\circ\textup{Ind}\big(M(w)\big)\cong N\Big(\big(\gamma_{\textup{semi}+}^{-1}\cdot\theta(\delta')\big)_{\leq-i}\Big)\oplus N\Big(\big(\gamma_{\textup{semi}-}^{-1}\theta(\delta')\big)_{\leq-i}\Big).$$   
\end{enumerate}
\end{prop}

\begin{proof}
(1) Since $w$ is non-symmetric, we have $$F_{\sigma}(M(w))=M({\sigma}(w))\ncong M(w).$$
This implies that $\textup{Ind}(M(w))$ is indecomposable.

(2) It suffices to treat the case $i=1$, since the general case is obtained by applying the corresponding truncation to $\xi(\delta)$ below. We assume that $w=\xi(\delta)$. By Proposition \ref{StToExst}, we have $$\Phi\circ\textup{Ind}\big(M(w)\big)\cong N(\varepsilon_1^{\delta_1}\gamma_{\textup{max}}\varepsilon_{1}^{\delta_2}\gamma_{\textup{max}}\cdots\gamma_{\textup{max}}\varepsilon_{1}^{\delta_{2k}}).$$

Applying the two basis changes from Lemma \ref{A3caseIso} yields
\begin{align*}
    N\big(\theta(\delta)\big)&\cong N(\varepsilon_1^{\delta_1}\gamma_{\textup{max}}\cdots\gamma_{\textup{max}}\varepsilon_{1}^{\delta_{k}}\gamma_{\textup{semi}+})\oplus N(\gamma_{\textup{semi}-}^{-1}\varepsilon_1^{\delta_{k+1}}\gamma_{\textup{max}}\cdots\gamma_{\textup{max}}\varepsilon_{1}^{\delta_{2k}})\\
    &\cong N\big(\gamma_{\textup{semi}+}^{-1}\theta(\delta')\big)\oplus N\big(\gamma_{\textup{semi}-}^{-1}\theta(\delta')\big).
\end{align*}
This completes the proof.
\end{proof}

\section{Equivariantization of band modules}
In this appendix, we similarly determine the equivariant correspondence for band modules.

\begin{prop}\label{FROMBdTOExBd}
Let $w$ be a band of $H(\widetilde{C}_{2n-2})$ and $\varphi$ an automorphism of ${{\bf k}}^s$. Then there exists an isomorphism of $H(\widetilde{CD}_{n})$-modules
$$\Phi\circ\textup{Ind}\big(M(w,\varphi)\big)\cong N\big(\phi(w),\varphi\big).$$
\end{prop}

\begin{proof}
Assume $w=\xi(\delta)\cdot\beta_{\textup{max}}^{-1}$ with $\delta\in\mathcal{I}_{3}$. By the definition of a band module, we can write the vector spaces of $M(w,\varphi)$ at each vertex $u$ as
$$M(w,\varphi)_u=\left\{
\begin{array}{ll}
\bigoplus_{x_i\in w,i\equiv u\textup{\;or\;}1-u(\textup{mod\;}4n-2)}V_{x_i}     &  1\leq u\leq n-1,\\
\bigoplus_{x_i\in w,i\equiv 2n+u\textup{\;or\;}2n-1-u(\textup{mod\;}4n-2)}V_{x_i}   & 1-n\leq u\leq -1,\\
\bigoplus_{x_i\in w,i\equiv n(\textup{mod\;}2n-1)}V_{x_i}   & u=0,
\end{array}\right.$$
where each $V_{x_i}\cong{{\bf k}}^s$ has basis $\{v_{i,1},v_{i,2},\dots,v_{i,s}\}$. Under the action of $F_{\sigma}$, we have
$$F_{\sigma}\big(M(w,\varphi)\big)_u=\left\{
\begin{array}{ll}
\bigoplus_{x_i\in w,i\equiv -u\textup{\;or\;}1+u(\textup{mod\;}4n-2)}V'_{x_i}     &  1-n\leq u\leq -1,\\
\bigoplus_{x_i\in w,i\equiv 2n-u\textup{\;or\;}2n-1+u(\textup{mod\;}4n-2)}V'_{x_i}   & 1\leq u\leq n-1,\\
\bigoplus_{x_i\in w,i\equiv n(\textup{mod\;}2n-1)}V'_{x_i}   & u=0,
\end{array}\right.$$
where $V'_{x_i}$ has basis $\{v'_{i,1},v'_{i,2},\dots,v'_{i,s}\}$ with $v'_{i,j}=F_{\sigma}(v_{i,j})$. Now consider $\Phi\circ\textup{Ind}\big(M(w,\varphi)\big)$. For the vertices we have 
$$\Phi\circ\textup{Ind}\big(M(w,\varphi)\big)_u=\left\{
\begin{array}{ll}
M(w,\varphi)_u\oplus F_{\sigma}\big(M(w,\varphi)\big)_{u}& 1\leq u\leq n-1, \\
{\textup{span}_{\bf k}}\{\frac{v_{i,j}+v'_{i,j}}{2}\mid i\equiv n(\textup{mod}\;2n-1),1\leq j\leq s\}     & u=0^{+},\\ 
{\textup{span}_{\bf k}}\{\frac{v_{i,j}-v'_{i,j}}{2}\mid i\equiv n(\textup{mod}\;2n-1),1\leq j\leq s\}     & u=0^{-}
\end{array}
\right.$$
Consequently,
$$\Phi\circ\textup{Ind}\big(M(w,\varphi)\big)=N(\varepsilon_1^{\delta_1}\gamma_{\textup{max}}\varepsilon_{1}^{\delta_2}\gamma_{\textup{max}}\cdots\gamma_{\textup{max}}\varepsilon_{1}^{\delta_{2k}}\gamma_{\textup{max}},\varphi)=N(\phi(w),\varphi).$$
\end{proof}

Let $J(\lambda,r)$ denote the $r\times r$ Jordan block with eigenvalue $\lambda$. It is well known that every indecomposable automorphism $\varphi:\;{\bf{k}}^r\rightarrow {\bf{k}}^r$ can be represented by a Jordan matrix with respect to a suitable basis. The following lemma is important for investigating the change of morphisms under the group action.

\begin{lem}\label{MatrixLemma}
Let $r$ be a positive integer and $\lambda=\pm1$. Then there exist involutive matrices $L(\lambda,r)$ satisfying $$L(\lambda,r)=L(\lambda,r)^{-1}\;\;\;\textup{and}\;\;\;L(\lambda,r)J(\lambda,r)^{-1}=J(\lambda,r)L(\lambda,r)$$
\end{lem}

\begin{proof}
Let $\vec{e}_i$ denote the $i$-th canonical basis column vector in an $r$-dimensional vector space. 

If $\lambda=1$, we set $\vec{l}_1=\vec{e}_1$ and for $2\leq i\leq r$, 
$$\vec{l}_i=(-1)^{i-1}\sum_{j=0}^{i-2}C_{i-2}^j \vec{e}_{j+2}.$$
It can then be verified that the matrix
$$L(1,r)=(\vec{l}_1,\vec{l}_2,\cdots,\vec{l}_r)$$ satisfies the required condition.

If $\lambda=-1$, we define $\vec{l}'_1=\vec{l}_1$ and for $2\leq i\leq r$, $$\vec{l}'_i={\sum_{j=0}^{i-2}(-1)^{j+1}C_{i-2}^j \vec{e}_{j+2}}.$$ 
One may check that $$L(-1,r)=(\vec{l}'_1,\vec{l}'_2,\cdots,\vec{l}'_r)$$ satisfies the desired property.
\end{proof}

For $\lambda\neq0$, denote by $\sqrt{J(\lambda,r)}$ a square root of the Jordan block $J(\lambda,r)$. Indeed, the Jordan form of $J(\sqrt{\lambda},r)^2$ is precisely $J(\lambda,r)$; hence there exists an invertible matrix $P$ such that
$$PJ(\sqrt{\lambda},r)^2P^{-1}=J(\lambda,r).$$
This yields
$$J(\lambda,r)=(PJ(\sqrt{\lambda},r)P^{-1})^2,$$
so we may take $\sqrt{J(\lambda,r)}=PJ(\sqrt{\lambda},r)P^{-1}$. 

Let $\varphi$ be the linear automorphism represented by the Jordan block $J=J(\lambda,r)$ with $\lambda\neq0$. For notational convenience, we also denote by  $\sqrt{\varphi}$ and $\lambda$ the linear automorphisms represented by the matrices $\sqrt{J}$ and $J(\lambda,1)$, respectively.

\begin{prop}\label{ClassifyBaToExba}
    Let $w=\xi(\delta)\cdot\beta_{\textup{max}}^{-1}$ be a band of $H(\widetilde{C}_{2n-2})$ with $\delta\in\mathcal{I}_{3}$, and $\varphi$ an indecomposable automorphism of a ${\bf k}$-vector space. 
\begin{enumerate}
    \item If $w$ is non-symmetric, then $\Phi\circ\textup{Ind}\big(M(w,\varphi)\big)$ is indecomposable.
    \item If $w$ is reversed with $\delta=\delta'\circ\delta'$, then
    $$\Phi\circ\textup{Ind}\big(M(w,\varphi)\big)\cong N\big(\theta(\delta')\cdot\gamma_{\textup{max}},\sqrt{\varphi}\big)\oplus N\big(\theta(\delta')\cdot\gamma_{\textup{max}},-\sqrt{\varphi}\big).$$
    \item If $w$ is rotated with $\delta=\delta'\circ\delta'^{-1}$, then
$$\Phi\circ\textup{Ind}\big(M(w,1)\big)\cong N\big(\gamma_{\textup{semi}+}^{-1}\cdot\theta(\delta')\cdot\gamma_{\textup{semi}+}\big)\oplus N\big(\gamma_{\textup{semi}-}^{-1}\cdot\theta(\delta')\cdot\gamma_{\textup{semi}-}\big),$$
    and  $$\Phi\circ\textup{Ind}\big(M(w,-1)\big)\cong N\big(\gamma_{\textup{semi}-}^{-1}\cdot\theta(\delta')\cdot\gamma_{\textup{semi}+}\big)\oplus N\big(\gamma_{\textup{semi}+}^{-1}\cdot\theta(\delta')\cdot\gamma_{\textup{semi}-}\big).$$
\end{enumerate}
\end{prop}

\begin{proof}
(1) By Proposition \ref{StringEquai}, $$F_{\sigma}\big(M(w,\varphi)\big)\ncong M(w,\varphi).$$
This implies that $\textup{Ind}(M(w,\varphi))$ is indecomposable, and consequently $\Phi\circ\textup{Ind}\big(M(w,\varphi)\big)$ is also indecomposable.

(2) If $w$ is reversed, we may assume that $k$ is the odd integer such that the length of $\delta$ equals $2k$ and $\delta_1=\delta_{k+1}=1$. Then we may view $N(\phi(w),\varphi)$ as obtained from the corresponding ex-band module by replacing both $$\varphi:\; U_{(2n,2k-1)}\rightarrow U_{(1,0)}\;\;\textup{and}\;\;\operatorname{id}:\; U_{(2n,k-1)}\rightarrow U_{(1,k)}$$
by $\sqrt{\varphi}$. In this $N(\phi(w),\varphi)$, we perform the basis change 
$$\begin{pmatrix}
    \operatorname{id} & \operatorname{id}\\
    \operatorname{id} & -\operatorname{id}
\end{pmatrix}:\;U_{j,l}\oplus U_{j,l+k}\rightarrow U'_{j,l}\oplus U''_{j,l}$$
for each $1\leq j\leq 2n$ and $0\leq l<k$. This yields an $H(\widetilde{CD}_{n})$-module isomorphism 
$$N(\phi(w),\varphi)\cong N'\oplus N''$$
where $$N'=\oplus_{1\leq j\leq 2n,\;0\leq l<k}\;U'_{j,l}\;\;\textup{and}\;\;N''=\oplus_{1\leq j\leq 2n,\;0\leq l<k}\;U''_{j,l}$$
as ${{\bf k}}$-vector spaces. Consequently, by Proposition \ref{FROMBdTOExBd}, we obtain
$$\Phi\circ\textup{Ind}\big(M(w,\varphi)\big)\cong N'\oplus N''\cong N\big(\theta(\delta')\cdot\gamma_{\textup{max}},\sqrt{\varphi}\big)\oplus N\big(\theta(\delta')\cdot\gamma_{\textup{max}},-\sqrt{\varphi}\big).$$

(3) In this case, we assume that the index $\delta$ has length $2k$. Since all $U_{j,l}$ are one-dimensional, we identify ${\bf k}x_i$ with $U_{j,l}$ via $i=2ln+j$. Here, we allow the index $i$ of $x_i$ (and later of $z_i$) to lie in $\mathbb{Z}_{4kn}$. 

For $\Phi\circ\textup{Ind}\big(M(w,1)\big)$, we perform the basis change
$$z_{i}=\left\{
\begin{array}{ll}
x_i+x_{-2n+1-i} & i\not\equiv n,n+1\;(\textup{mod}\;2n)\;\textup{and}\;-n+2\leq i<2kn-n, \\
x_i-x_{-2n+1-i} & i\not\equiv n,n+1\;(\textup{mod}\;2n)\;\textup{and}\;2kn-n+2\leq i<4kn-n, \\
x_i+x_{-2n-i} & i\equiv n\;(\textup{mod}\;2n)\;\textup{and}\;-n+2\leq i<2kn-n, \\
x_i-x_{-2n+2-i} & i\equiv n+1\;(\textup{mod}\;2n)\;\textup{and}\;-n+2\leq i<2kn-n, \\
x_i-x_{-2n-i} & i\equiv n\;(\textup{mod}\;2n)\;\textup{and}\;2kn-n+2\leq i<4kn-n, \\
x_i+x_{-2n+2-i} & i\equiv n+1(\textup{mod}\;2n)\;\textup{and}\;2kn-n+2\leq i<4kn-n, \\
2x_i    & i=2kn-n\textup{\;or\;} 4kn-n,\\
-2x_i   &i=2kn-n+1\textup{\;or\;} 4kn-n+1.
\end{array}
\right.$$

By assumption, $\delta_{k+1+i}=-\delta_{k-i}$ for each $0\leq i<k$. This yields an $H(\widetilde{CD}_{n})$-module isomorphism
$$N(\phi(w),\varphi)\cong N'\oplus N'',$$
where $N'$ has a ${{\bf k}}$-basis $$\{z_{4kn-n}\}\cup\{z_i\mid -n+2\leq i\leq 2kn-n\},$$
and $N''$ has a ${{\bf k}}$-basis $$\{z_{4kn-n+1}\}\cup\{z_i\mid 2kn-n+1\leq i< 4kn-n\}.$$
Consequently, by Proposition \ref{FROMBdTOExBd}, we obtain
$$\Phi\circ\textup{Ind}\big(M(w,1)\big)\cong N'\oplus N''\cong N\big(\gamma_{\textup{semi}+}^{-1}\cdot\theta(\delta')\cdot\gamma_{\textup{semi}+}\big)\oplus N\big(\gamma_{\textup{semi}-}^{-1}\cdot\theta(\delta')\cdot\gamma_{\textup{semi}-}\big).$$

The case $\Phi\circ\textup{Ind}\big(M(w,-1)\big)$ follows by a similar calculation, which we omit here.
\end{proof}

\begin{rem}
We establish the equivariant correspondence for all band modules except those lying in the second and higher layers of a homogeneous tube associated to a rotated band with eigenvalues $\pm1$. These modules are neither ex-string modules nor ex-band modules. However, since they do not affect the description of the index sets in the Auslander--Reiten quiver, we omit their correspondence to avoid tedious calculations.
\end{rem}

\bibliographystyle{amsplain}

\begin{thebibliography}{20}

\bibitem{Elements1}
I.~Assem, D.~Simson, and A.~Skowro\'nski.
\newblock {\em Elements of the representation theory of associative algebras.
  {V}ol. 1}, volume~65 of {\em London Mathematical Society Student Texts}.
\newblock Cambridge University Press, Cambridge, 2006.
\newblock Techniques of representation theory.

\bibitem{RepTheIII}
M.~Auslander and I.~Reiten.
\newblock Representation theory of {A}rtin algebras. {III}. {A}lmost split
  sequences.
\newblock {\em Comm. Algebra}, 3:239--294, 1975.

\bibitem{ARS1995}
M.~Auslander, I.~Reiten, and S.~O. Smal\o.
\newblock {\em Representation theory of {A}rtin algebras}, volume~36 of {\em
  Cambridge Studies in Advanced Mathematics}.
\newblock Cambridge University Press, Cambridge, 1995.

\bibitem{BGP1973}
I.~N. Bernstein, I.~M. Gel'fand, and V.~A. Ponomarev.
\newblock Coxeter functors, and {G}abriel's theorem.
\newblock {\em Uspehi Mat. Nauk}, 28(2(170)):19--33, 1973.

\bibitem{Bongartz1983}
K.~Bongartz.
\newblock Algebras and quadratic forms.
\newblock {\em J. London Math. Soc. (2)}, 28(3):461--469, 1983.

\bibitem{Bong1984}
K.~Bongartz.
\newblock A criterion for finite representation type.
\newblock {\em Math. Ann.}, 269(1):1--12, 1984.

\bibitem{Bongartz1984}
K.~Bongartz.
\newblock Critical simply connected algebras.
\newblock {\em Manuscripta Math.}, 46(1-3):117--136, 1984.

\bibitem{BG1982}
K.~Bongartz and P.~Gabriel.
\newblock Covering spaces in representation-theory.
\newblock {\em Invent. Math.}, 65(3):331--378, 1981/82.

\bibitem{BR1987}
M.~C.~R. Butler and C.~M. Ringel.
\newblock Auslander-{R}eiten sequences with few middle terms and applications
  to string algebras.
\newblock {\em Comm. Algebra}, 15(1-2):145--179, 1987.

\bibitem{Car}
R.~W. Carter.
\newblock {\em Lie algebras of finite and affine type}, volume~96 of {\em
  Cambridge Studies in Advanced Mathematics}.
\newblock Cambridge University Press, Cambridge, 2005.

\bibitem{CCR}
J.~Chen, X.-W. Chen, and S.~Ruan.
\newblock The dual actions, equivariant autoequivalences and stable tilting
  objects.
\newblock {\em Ann. Inst. Fourier (Grenoble)}, 70(6):2677--2736, 2020.

\bibitem{ChenXW2017}
X.-W. Chen.
\newblock Equivariantization and {S}erre duality {I}.
\newblock {\em Appl. Categ. Structures}, 25(4):539--568, 2017.

\bibitem{ChenWang2019}
X.-W. Chen and R.~Wang.
\newblock The finite {EI} categories of {C}artan type.
\newblock {\em J. Algebra}, 546:62--84, 2020.

\bibitem{DR1976}
V.~Dlab and C.~M. Ringel.
\newblock Indecomposable representations of graphs and algebras.
\newblock {\em Mem. Amer. Math. Soc.}, 6(173):v+57, 1976.

\bibitem{DRZ}
Q.~Dong, S.~Ruan, and H.~Zhang.
\newblock Equivariant approach to weighted projective curves.
\newblock {\em J. Algebra}, 608:388--411, 2022.

\bibitem{DF}
P.~Donovan and M.~R. Freislich.
\newblock {\em The representation theory of finite graphs and associated
  algebras}, volume No. 5 of {\em Carleton Mathematical Lecture Notes}.
\newblock Carleton University, Ottawa, ON, 1973.

\bibitem{Dro1977}
Y. A. Drozd.
\newblock Tame and wild matrix problems.
\newblock In {\em Matrix problems ({R}ussian)}, pages 104--114. Akad. Nauk
  Ukrain. SSR, Inst. Mat., Kiev, 1977.

\bibitem{Dro1980}
Y. A. Drozd.
\newblock Tame and wild matrix problems.
\newblock In {\em Representation theory, {II} ({P}roc. {S}econd {I}nternat.
  {C}onf., {C}arleton {U}niv., {O}ttawa, {O}nt., 1979)}, volume 832 of {\em
  Lecture Notes in Math.}, pages 242--258. Springer, Berlin, 1980.

\bibitem{Gab1972}
P.~Gabriel.
\newblock Unzerlegbare {D}arstellungen. {I}.
\newblock {\em Manuscripta Math.}, 6: 71--103; correction, ibid. 6(1972), 309, 1972.

% \bibitem{Gabriel1980}
% P.~Gabriel.
% \newblock Auslander-{R}eiten sequences and representation-finite algebras.
% \newblock In {\em Representation theory, {I} ({P}roc. {W}orkshop, {C}arleton {U}niv., {O}ttawa, {O}nt., 1979)}, volume 831 of {\em Lecture Notes in Math.}, pages 1--71. Springer, Berlin, 1980.

\bibitem{Gar1981}
P.~Gabriel.
\newblock The universal cover of a representation-finite algebra.
\newblock In {\em Representations of algebras ({P}uebla, 1980)}, volume 903 of {\em Lecture Notes in Math.}, pages 68--105. Springer, Berlin-New York, 1981.

\bibitem{GLSIII}
C.~Gei\ss, B.~Leclerc, and J.~Schr\"oer.
\newblock Quivers with relations for symmetrizable {C}artan matrices {III}:
  {C}onvolution algebras.
\newblock {\em Represent. Theory}, 20:375--413, 2016.

\bibitem{GLSI}
C.~Gei\ss, B.~Leclerc, and J.~Schr\"oer.
\newblock Quivers with relations for symmetrizable {C}artan matrices {I}:
  {F}oundations.
\newblock {\em Invent. Math.}, 209(1):61--158, 2017.

\bibitem{GLSII}
C.~Gei\ss, B.~Leclerc, and J.~Schr\"oer.
\newblock Quivers with relations for symmetrizable {C}artan matrices {II}:
  change of symmetrizers.
\newblock {\em Int. Math. Res. Not. IMRN}, (9):2866--2898, 2018.

\bibitem{GLSIV}
C.~Gei\ss, B.~Leclerc, and J.~Schr\"oer.
\newblock Quivers with relations for symmetrizable {C}artan matrices {IV}:
  crystal graphs and semicanonical functions.
\newblock {\em Selecta Math. (N.S.)}, 24(4):3283--3348, 2018.

\bibitem{GLSV}
C.~Gei\ss, B.~Leclerc, and J.~Schr\"oer.
\newblock Quivers with relations for symmetrizable {C}artan matrices {V}:
  {C}aldero-{C}hapoton formulas.
\newblock {\em Proc. Lond. Math. Soc. (3)}, 117(1):125--148, 2018.

\bibitem{HLS2023}
H.~Huang, Z.~Lin, and X.~Su.
\newblock Components of {AR}-quivers for string algebras of type
  {$\tilde{\Bbb{C}}$} and a conjecture by {G}eiss-{L}eclerc-{S}chr\"oer.
\newblock {\em J. Algebra}, 632:331--362, 2023.

\bibitem{Ler1987}
M.~Lersch.
\newblock Minimal wilde Algebren.
\newblock Diplomarbeit, D\"usseldorf, 1987.

\bibitem{LS}
Z.~Leszczy\'nski and A.~Skowro\'nski.
\newblock Tame tensor products of algebras.
\newblock {\em Colloq. Math.}, 98(1):125--145, 2003.

\bibitem{LS2025}
Z.~Lin and X.~Su.
\newblock Affine root systems, stable tubes, and a conjecture by
  {G}eiss-{L}eclerc-{S}chr\"oer.
\newblock {\em Int. Math. Res. Not. IMRN}, 2025(2):rnae279, 2025.

% \bibitem{Lus00}
% G.~Lusztig.
% \newblock Semicanonical bases arising from enveloping algebras.
% \newblock {\em Adv. Math.}, 151 (2000), no. 2, 129--139.

% \bibitem{MP1983}
% R.~Mart\'inez-Villa and J.~A. de~la Pe\~na.
% \newblock The universal cover of a quiver with relations.
% \newblock {\em J. Pure Appl. Algebra}, 30(3): 277--292, 1983.

\bibitem{Naz1973}
L. A. Nazarova.
\newblock Representations of quivers of infinite type.
\newblock {\em Izv. Akad. Nauk SSSR Ser. Mat.} 37 (1973), 752–791. In Russian; translated in Math. USSR-Izv. 7:4 (1973), 749--792.

\bibitem{NS1997}
R.~N\"orenberg and A.~Skowro\'nski.
\newblock Tame minimal non-polynomial growth simply connected algebras.
\newblock {\em Colloq. Math.}, 73(2): 301--330, 1997.

\bibitem{delaPena1990}
J.~A. de~la Pe\~na.
\newblock Algebras with hypercritical {T}its form.
\newblock In {\em Topics in algebra, {P}art 1 ({W}arsaw, 1988)}, volume 26,
  Part 1 of {\em Banach Center Publ.}, pages 353--369. PWN, Warsaw, 1990.

\bibitem{delaPena1991} 
J.~A. de~la Pe\~na.
\newblock On the dimension of the module-varieties of tame and wild algebras.
\newblock {\em Comm. Algebra}, 19(6):1795--1807, 1991.

\bibitem{ReitenRiedtmann1985Skew}
I.~Reiten and C.~Riedtmann.
\newblock Skew group algebras in the representation theory of {A}rtin algebras.
\newblock {\em J. Algebra}, 92(1): 224--282, 1985.

\bibitem{Covering15}
C.~Riedtmann.
\newblock Algebren, {D}arstellungsk\"{o}cher, \"{U}berlagerungen und zur\"{u}ck.
\newblock {\em Comment. Math. Helv.}, 55(2):199--224, 1980.

\bibitem{[Rin]}
C.~M. Ringel.
\newblock {\em Tame algebras and integral quadratic forms}, volume 1099 of {\em Lecture Notes in Mathematics}.
\newblock Springer-Verlag, Berlin, 1984.

\bibitem{Elements2}
D.~Simson and A.~Skowro\'nski.
\newblock {\em Elements of the representation theory of associative algebras. {V}ol. 2}, volume~71 of {\em London Mathematical Society Student Texts}.
\newblock Cambridge University Press, Cambridge, 2007.
\newblock Tubes and concealed algebras of Euclidean type.

\bibitem{Skow1987}
A.~Skowro\'nski.
\newblock Group algebras of polynomial growth.
\newblock {\em Manuscr. Math.}, 59:499–516, 1987.

\bibitem{Unger1990}
L.~Unger.
\newblock The concealed algebras of the minimal wild, hereditary algebras.
\newblock {\em Bayreuther Math. Schr.}, (31): 145--154, 1990.

\bibitem{Wittman1990}
J.~Wittman.
\newblock Verkleidete zahme und minimal wilde Algebren.
\newblock Diplomarbeit, Bayreuth, 1990.

\end{thebibliography}

\end{document}